%% file: ArXiv_2010_1001.0509_Gerolymos.tex
\documentclass[final,3p,times]{elsarticle}
\newtheorem{theorem}{Theorem}[section]
\newtheorem{lemma}[theorem]{Lemma}
\newtheorem{corollary}[theorem]{Corollary}
\newtheorem{proposition}[theorem]{Proposition}

\newtheorem{result}[theorem]{Result}

\newdefinition{definition}[theorem]{Definition}
\newdefinition{example}[theorem]{Example}
\newdefinition{xca}[theorem]{Exercise}

\newdefinition{remark}[theorem]{Remark}
\newproof{proof}{Proof}

\usepackage{newcent}
\usepackage{amssymb}
\usepackage{amsfonts}
\usepackage{amsmath}
\usepackage{euscript}
\usepackage{color}
\usepackage{graphicx}
\usepackage{rotating}
\usepackage{bm}
\usepackage{curves}
\usepackage{pstricks}
\usepackage{lettrine}
\usepackage{mathrsfs}
\usepackage{booktabs}
\DeclareGraphicsExtensions{{},.eps}
\usepackage[nodots]{numcompress}

\providecommand{\tsc}[1]{{\text{\sc#1}}}
\providecommand{\tscn}[2]{{\text{{\sc#1}{\small#2}}}}

\providecommand{\etal}{et al.}
\providecommand{\tabref}[1]{{\textup{(Tab.~\ref{#1})}}}

\providecommand{\thmref}[1]{{\textup{(Theorem~\ref{#1})}}}
\providecommand{\lemref}[1]{{\textup{(Lemma~\ref{#1})}}}
\providecommand{\crlref}[1]{{\textup{(Corollary~\ref{#1})}}}
\providecommand{\prpref}[1]{{\textup{(Proposition~\ref{#1})}}}
\providecommand{\defref}[1]{{\textup{(Definition~\ref{#1})}}}

\providecommand{\rmkref}[1]{{\textup{(Remark~\ref{#1})}}}

\providecommand{\thmrefnp}[1]{{\textup{Theorem~\ref{#1}}}}
\providecommand{\lemrefnp}[1]{{\textup{Lemma~\ref{#1}}}}
\providecommand{\crlrefnp}[1]{{\textup{Corollary~\ref{#1}}}}
\providecommand{\prprefnp}[1]{{\textup{Proposition~\ref{#1}}}}
\providecommand{\defrefnp}[1]{{\textup{Definition~\ref{#1}}}}

\providecommand{\rmkrefnp}[1]{{\textup{Remark~\ref{#1}}}}

\providecommand{\const}{{\rm const}}
\providecommand{\strlngfk}[2]{\genfrac{[}{]}{0pt}{}{#1}{#2}}

\begin{document}

\begin{frontmatter}

\title{Approximation error of the Lagrange reconstructing polynomial}

\author{G.A. Gerolymos}
\ead{georges.gerolymos@upmc.fr}
\address{Universit\'e Pierre-et-Marie-Curie (\tsc{upmc}),
         Case 161, 4 place Jussieu, 75005 Paris, France}
\journal{J. Approx. Theory}


\begin{abstract}
The reconstruction approach [Shu C.W.: {\em SIAM Rev.} {\bf 51} (2009) 82--126] for the numerical approximation of $f'(x)$
is based on the construction of a dual function $h(x)$ whose sliding averages over the interval $[x-\tfrac{1}{2}\Delta x,x+\tfrac{1}{2}\Delta x]$ are equal to $f(x)$
(assuming an homogeneous grid of cell-size $\Delta x$).
We study the deconvolution problem [Harten A., Engquist B., Osher S., Chakravarthy S.R.: {\em J. Comp. Phys.} {\bf 71} (1987) 231--303]
which relates the Taylor polynomials of $h(x)$ and $f(x)$, and obtain its explicit solution,
by introducing rational numbers $\tau_n$ defined by a recurrence relation, or determined by their generating function, $g_\tau(x)$, related with the reconstruction pair of ${\rm e}^x$.
We then apply these results to the specific case of Lagrange-interpolation-based polynomial reconstruction, and determine explicitly the approximation error of the Lagrange reconstructing
polynomial (whose sliding averages are equal to the Lagrange interpolating polynomial) on an arbitrary stencil defined on a homogeneous grid.
\end{abstract}

\end{frontmatter}

%-----------------------------------------------------------------------------------------------------------------------------------
%
%
%
%
%
%
%
%
%
\section{Introduction}\label{AELRP_s_I}
%
%
%
%
%
%
%
%
%
%-----------------------------------------------------------------------------------------------------------------------------------

The Godunov approach~\cite{Toro_1997a} to hyperbolic conservation laws
\begin{equation}
\partial_t u+\partial_x F(u)=0
                                                                                                       \label{Eq_AELRP_s_I_001}
\end{equation}
is based on space-time averaging of the \tsc{pde}~\eqref{Eq_AELRP_s_I_001}. Assuming an homogeneous time-independent grid ($\Delta x ={\rm const}$),
space-averaging of~\eqref{Eq_AELRP_s_I_001}, over the interval $[x-\tfrac{1}{2}\Delta x,x+\tfrac{1}{2}\Delta x]$,
leads to the \underline{exact} relation~\cite{Toro_1997a}
\begin{alignat}{6}
\frac{\partial}{\partial t}\bar u(x,t)+\frac{1}{\Delta x}\left[F\left(u(x+\tfrac{1}{2}\Delta x,t)\right)-F\left(u(x-\tfrac{1}{2}\Delta x,t)\right)\right]=0
                                                                                                       \label{Eq_AELRP_s_I_002}
\end{alignat}
where
\begin{equation}
\bar u(x,t):=\int_{-\frac{1}{2}}^{+\frac{1}{2}} u(x+\xi\Delta x,t)d\xi
                                                                                                       \label{Eq_AELRP_s_I_003}
\end{equation}
are the sliding cell-averages of the solution. Defining the sliding cell-averages $\overline{F(u)}$, by applying the operator \eqref{Eq_AELRP_s_I_003} on $F(u)$,
we have immediately by differentiation, provided that $\Delta x ={\rm const}$,
\begin{equation}
\frac{\partial\overline{F\left(u(x,t)\right)}}{\partial x}=\dfrac{F\Bigl(u(x+\tfrac{1}{2}\Delta x,t)\Bigr)-F\Bigl(u(x-\tfrac{1}{2}\Delta x,t)\Bigr)}{\Delta x}
                                                                                                       \label{Eq_AELRP_s_I_004}
\end{equation}
exactly, so that, combining \eqref{Eq_AELRP_s_I_002} and \eqref{Eq_AELRP_s_I_004}
\begin{equation}
\partial_t \bar  u+\partial_x \overline{F(u)}=0
                                                                                                       \label{Eq_AELRP_s_I_005}
\end{equation}
{\em ie} the equation for the sliding cell-averages, for $\Delta x ={\rm const}$, has the same form as the original equation~\cite{Harten_Engquist_Osher_Chakravarthy_1987a}.
For this reason, it is assumed that what is computed (and stored at the nodes of the computational grid~\cite{Shu_1998a,
                                                                                                              Shu_2009a}) are the cell-averages of the solution.

In the discretization of \eqref{Eq_AELRP_s_I_002} we are led to consider the computation of the derivative
of a function $f(x)$ (corresponding to $\bar u$) sampled on the computational grid, by differences at $x\pm\tfrac{1}{2}\Delta x$ of the values
of an unknown function $h(x)$ (corresponding to $u$), which has to be reconstructed~\cite{Harten_Engquist_Osher_Chakravarthy_1987a,
                                                                                          Liu_Osher_Chan_1994a,
                                                                                          Jiang_Shu_1996a,
                                                                                          Shu_1998a,
                                                                                          Shu_2009a}
from the values of its cell-averages sampled on the grid.
In the following, we concentrate on the spatial discretization problem, {\em viz} compute $f'(x)$ via reconstruction of $h(x\pm\tfrac{1}{2}\Delta x)$~\cite{Harten_Engquist_Osher_Chakravarthy_1987a,
                                                                                                                                                            Liu_Osher_Chan_1994a,
                                                                                                                                                            Jiang_Shu_1996a,
                                                                                                                                                            Shu_1998a,
                                                                                                                                                            Shu_2009a}.

Reconstruction~\defref{Def_AELRP_s_RPERR_ss_RP_001} is the basis of \tsc{eno}~\cite{Harten_Osher_1987a,
                                                                                   Harten_Engquist_Osher_Chakravarthy_1987a,
                                                                                   Shu_Osher_1988a,
                                                                                   Shu_Osher_1989a}
and \tsc{weno}~\cite{Liu_Osher_Chan_1994a,
                     Jiang_Shu_1996a,
                     Shu_1998a,
                     Levy_Puppo_Russo_1999a,
                     Balsara_Shu_2000a,
                     Qiu_Shu_2002a,
                     Henrick_Aslam_Powers_2005a,
                     Borges_Carmona_Costa_Don_2008a,
                     Shu_2009a,
                     Gerolymos_Senechal_Vallet_2009a}
schemes.
Although the term polynomial reconstruction is quite general, including Hermite-interpolation~\cite{Qiu_Shu_2003a},
spectral methods~\cite{Gottlieb_Shu_1997a}, and spectral element techniques~\cite{Xu_Liu_Shu_2009a},
we concentrate in the present work on reconstruction approaches based on Lagrangian interpolation of the function averages.
There exist several algorithms for Lagrangian-interpolation-based polynomial reconstruction~\cite{Harten_Engquist_Osher_Chakravarthy_1987a,
                                                                                                  Shu_1998a,
                                                                                                  Shu_2009a},
and these have been successfully used for the construction of progressively higher-order schemes~\cite{Jiang_Shu_1996a,
                                                                                                       Balsara_Shu_2000a,
                                                                                                       Gerolymos_Senechal_Vallet_2009a},
using symbolic calculation~\cite{Henrick_Aslam_Powers_2005a,
                                 Borges_Carmona_Costa_Don_2008a}.
The {\em reconstruction via primitive} approach~\cite{Harten_Engquist_Osher_Chakravarthy_1987a} is probably the most widely used in order-of-accuracy proofs~\cite{Shu_1998a,
                                                                                                                                                                   Shu_2009a},
while the {\em reconstruction via deconvolution} approach~\cite{Harten_Engquist_Osher_Chakravarthy_1987a} has been formulated with respect to the solution of linear systems.
Most of these schemes and associated order-of-accuracy relations~\cite{Jiang_Shu_1996a,
                                                                       Balsara_Shu_2000a,
                                                                       Henrick_Aslam_Powers_2005a,
                                                                       Borges_Carmona_Costa_Don_2008a,
                                                                       Gerolymos_Senechal_Vallet_2009a}
were developed for particular values of the order-parameter $r$ (determining the discretization stencil), using symbolic computation.
On the other hand, analytical relations for the order-of-accuracy of the approximation of $h(x)$, for arbitrary reconstruction-order-parameter $r$ are not available.
To obtain such relations it seems necessary to study in detail the relations between a function $h(x)$ (which is reconstructed) and its cell-averages $f(x)$.
This is the {\em reconstruction via deconvolution} approach defined by Harten \etal~\cite{Harten_Engquist_Osher_Chakravarthy_1987a}. Up to now, these relations were obtained
by solving, using symbolic calculation, the associated linear system~\cite[(3.13b), p. 244]{Harten_Engquist_Osher_Chakravarthy_1987a}, up to a certain order.
Although the solution by symbolic computation of the linear system~\cite[(3.13b), p. 244]{Harten_Engquist_Osher_Chakravarthy_1987a} is not difficult,
it is only valid up to a certain $O(\Delta x^q)$,
and the non availability of an explicit solution hinders the development of general expressions of the approximation error of the reconstruction.

Along with the numerous successful developments of practical \tsc{weno} schemes based on the reconstruction via primitive approach~\cite{Shu_1998a,
                                                                                                                                         Shu_2009a},
the unknown function $h(x)$ which is reconstructed by its cell-averages $f(x)$
appears explicitly in recent analyses~\cite{Henrick_Aslam_Powers_2005a,
                                            Borges_Carmona_Costa_Don_2008a}
of the truncation error. Analyzing the reconstruction error in terms of the unknown function $h(x)$ and its derivatives ({\em reconstruction via deconvolution}~\cite{Harten_Engquist_Osher_Chakravarthy_1987a})
is a more intuitive approach,
especially when considering the discretization error of the \tsc{weno} approximation to $f'(x)$ and potential improvements in the formulation of the
nonlinear weights~\cite{Henrick_Aslam_Powers_2005a,
                        Borges_Carmona_Costa_Don_2008a}.
Analyses based on the {\em reconstruction via primitive}~\cite{Harten_Engquist_Osher_Chakravarthy_1987a} approach are somehow less straightforward as they involve the
primitive of the reconstructed function $\int_{x_1}^{x} h(\zeta)d\zeta$.\footnote{\label{ff_AELRP_s_I_001}The lower bound of the integral being of no particular consequence~\cite{Shu_1998a,
                                                                                                                                                                                  Shu_2009a,
                                                                                                                                                                                  Liu_Shu_Zhang_2009a}
                                                                                                         we can chose the coordinate of grid-point $x_1$ instead of the usual (but more abstract)
                                                                                                         lower bound at $-\infty$~\cite{Shu_1998a,
                                                                                                                                        Shu_2009a,
                                                                                                                                        Liu_Shu_Zhang_2009a}.
                                                                                 }
The main motivation of the present work is to contribute to the development of analytical tools, applicable to the study of the truncation error~\cite{Henrick_Aslam_Powers_2005a,
                                                                                                                                                       Borges_Carmona_Costa_Don_2008a},
determination of the loss-of-accuracy at smooth extrema~\cite{Henrick_Aslam_Powers_2005a}
and research for the improvement of \tsc{weno} schemes~\cite{Borges_Carmona_Costa_Don_2008a},
maintaining the in-built scalability (with the stencil width) towards higher order-of-accuracy of \tsc{weno} schemes~\cite{Jiang_Shu_1996a,
                                                                                                                            Balsara_Shu_2000a,
                                                                                                                            Gerolymos_Senechal_Vallet_2009a}.
For these reasons, in the present work we are not interested in the development of a new algorithm for the solution of the reconstruction problem.
Instead, we focus on reconstruction relations of general validity, {\em ie} stencil-independent, which are necessary for the study of the approximate reconstruction order-of-accuracy.

In \S\ref{AELRP_s_RPERR} we study the general relations underlying the reconstruction approach for the numerical approximation
of the 1-derivative $f'(x)$ of a function $f(x)$. Initially we study the relations between the derivatives of a function $f(x)$ and those of a dual function
$h(x)$, whose sliding averages, over a constant length $\Delta x$, are equal to $f(x)$. We will call the functions, $f(x)$ and $h(x)$, satisfying this relation a reconstruction pair
for the discretization of $f'(x)$ \defref{Def_AELRP_s_RPERR_ss_RP_001}. We introduce the rational numbers $\tau_n\in{\mathbb Q}$, defined either by a recurrence relation \lemref{Lem_AELRP_s_RPERR_ss_D_001}
or through a generating function \thmref{Thm_AELRP_s_RPERR_ss_GFtaunRPexp_001}, which are used to develop explicit series representations of $h(x)$ (and of its derivatives) with respect to powers of $\Delta x$ and
the derivatives of $f(x)$. The principal new result in \S\ref{AELRP_s_RPERR} is that we are able to give explicit solutions to the fundamental relations of the
reconstruction via deconvolution approach~\cite[(3.13), pp. 244--246]{Harten_Engquist_Osher_Chakravarthy_1987a}, which \lemref{Lem_AELRP_s_RPERR_ss_D_001} are widely used throughout the paper.
The generating function of the rational numbers $\tau_n\in{\mathbb Q}$ appears in the expression of the reconstruction pair of ${\rm e}^x$ \thmref{Thm_AELRP_s_RPERR_ss_GFtaunRPexp_001}.

In \S\ref{AELRP_s_RP} we study the particular case of polynomial reconstruction. We show \lemref{Lem_AELRP_s_RP_ss_PRP_001} that for every polynomial $p_f(x)$ of degree $M$ in $x$
we can define, using the numbers $\tau_n$ \lemref{Lem_AELRP_s_RPERR_ss_D_001}, a polynomial $p_h(x)$, also of degree $M$ in $x$, so that $p_f$ and $p_h$ are a unique reconstruction
pair \defref{Def_AELRP_s_RPERR_ss_RP_001}. Initially (\S\ref{AELRP_s_RP_ss_MIPPRPLem}) the numbers $\tau_n$ \lemref{Lem_AELRP_s_RPERR_ss_D_001} were introduced, using a matrix algebra approach to study
the relation between $p_f(x)$ and $p_h(x)$. This part of the paper (\S\ref{AELRP_s_RP_ss_MIPPRPLem}) gives the explicit inversion of the matrix appearing in the
reconstruction via deconvolution theory~\cite[(3.13b), p. 244]{Harten_Engquist_Osher_Chakravarthy_1987a}.

In practice $f(x)$ is usually approximated by its Lagrange interpolating polynomial $p_f(x;\tsc{s}_{i,M_-,M_+},\Delta x)$ on a given stencil $\tsc{s}_{i,M_-,M_+}$ \defref{Def_AELRP_s_EPR_ss_PR_001},
and $h(x)$ is approximated by the reconstruction pair of $p_f(x;\tsc{s}_{i,M_-,M_+},\Delta x)$, $p_h(x;\tsc{s}_{i,M_-,M_+},\Delta x)$, which we will call the {\em Lagrange reconstructing polynomial}~\cite{Shu_2009a}.
In \S\ref{AELRP_s_EPR} we study the approximation error of the Lagrange reconstructing polynomial, 
$E_h(x;\tsc{s}_{i,M_-,M_+},\Delta x):= p_h(x;\tsc{s}_{i,M_-,M_+},\Delta x)-h(x)$, and obtain an explicit relation for the expansion of this error in powers of $\Delta x$ \prpref{Prp_AELRP_s_EPR_ss_AELPR_002}.
This is only possible through the explicit solution of the deconvolution problem \lemref{Lem_AELRP_s_RPERR_ss_D_001}.
In \S\ref{AELRP_s_IRP} we briefly summarize the existence and uniqueness results concerning the reconstructing polynomial.
Finally, in \S\ref{AELRP_s_EA} we briefly describe some applications of the present results to practical \tsc{weno} schemes, and discretization methods in general,
highlighting the merits of the reconstruction via deconvolution approach, as developed in the present work.

Standard results referring to the Lagrange interpolating polynomial~\cite{Henrici_1964a,
                                                                          Phillips_2003a}
are included only when they are necessary for the proof of the new results concerning the reconstructing polynomial.
Useful relations for summation indices in multiple sums~\cite{Knuth_1992a,
                                                              Graham_Knuth_Patashnik_1994a},
and other identities,
used throughout the paper, are summarized in \ref{AELRP_s_AppendixA}.

%-----------------------------------------------------------------------------------------------------------------------------------
%
%
%
%
%
%
%
%
%
\section{Reconstruction pairs and exact reconstruction relations}\label{AELRP_s_RPERR}
%
%
%
%
%
%
%
%
%
%-----------------------------------------------------------------------------------------------------------------------------------

Before proceeding to a detailed examination of the reconstruction of polynomials we examine the general
relations underlying the reconstruction approach for the evaluation of the derivative $f'(x)$ of a function $f(x)$,
via the construction of a function $h(x)$ (reconstruction pair of $f(x)$;~\defrefnp{Def_AELRP_s_RPERR_ss_RP_001}), whose sliding (with $x$) averages~\cite{Shu_1998a,
                                                                                                                                                          Shu_2009a}
on the interval $[x-\tfrac{1}{2}\Delta x,x+\tfrac{1}{2}\Delta x]$ are equal to $f(x)$, over an appropriate interval $x\in I\subset{\mathbb R}$.
We express in particular the derivatives of $h(x)$ as series of the derivatives of $f(x)$, with coefficients determined
by the derivatives at $\Delta x=0$ of the function $g_\tau(\Delta x)$ appearing in the reconstruction pair of the exponential function~\thmref{Thm_AELRP_s_RPERR_ss_GFtaunRPexp_001}.

%-----------------------------------------------------------------------------------------------------------------------------------
%
%
%
%
%
\subsection{Reconstruction pairs}\label{AELRP_s_RPERR_ss_RP}
%
%
%
%
%
%-----------------------------------------------------------------------------------------------------------------------------------

The basic idea underlying reconstruction procedures to compute the derivative $f'(x)$ of a function $f(x)$
follows directly from the Leibniz rule~\cite[pp. 411--412]{Zorich_2004b} giving the derivative of a definite integral
with respect to its (variable) bounds. To this end we need to construct a function $h(x)$ whose sliding (with $x$) average over an
interval $[x-\frac{1}{2}\Delta x,x+\frac{1}{2}\Delta x]$ of \underline{constant} width $\Delta x$ is equal to $f(x)$.
%-----------------------------------------------------------------------------------------------------------------------------------
%
\begin{definition}[Reconstruction pair]
\label{Def_AELRP_s_RPERR_ss_RP_001}
Assume that $\Delta x\in{\mathbb R}_{>0}$ is a \underline{constant} length, and that the functions
$f:I\longrightarrow{\mathbb R}$ and $h:I\longrightarrow{\mathbb R}$ are defined on the interval $I=[a-\frac{1}{2}\Delta x,b+\frac{1}{2}\Delta x]\subset{\mathbb R}$,
satisfying everywhere
\begin{subequations}
                                                                                                       \label{Eq_Def_AELRP_s_RPERR_ss_RP_001_001}
\begin{equation}
f(x)=\dfrac{1}{\Delta x}\int_{x-\frac{1}{2}\Delta x}^{x+\frac{1}{2}\Delta x}{h(\zeta)d\zeta}\quad\forall x\in[a,b]
                                                                                                       \label{Eq_Def_AELRP_s_RPERR_ss_RP_001_001a}
\end{equation}
assuming the existence of the integral in~\eqref{Eq_Def_AELRP_s_RPERR_ss_RP_001_001a}.
We will note the functions $f(x)$ and $h(x)$ related by~\eqref{Eq_Def_AELRP_s_RPERR_ss_RP_001_001a}
\begin{alignat}{6}
h=&R_{(1;\Delta x)}(f)
                                                                                                       \label{Eq_Def_AELRP_s_RPERR_ss_RP_001_001b}\\
f=&R^{-1}_{(1;\Delta x)}(h)
                                                                                                       \label{Eq_Def_AELRP_s_RPERR_ss_RP_001_001c}
\end{alignat}
\end{subequations}
and will call $f$ and $h$
a reconstruction pair on $[a,b]$,
in view of the computation of the 1-derivative.\qed
\end{definition}
%
%-----------------------------------------------------------------------------------------------------------------------------------
By definition \eqref{Eq_Def_AELRP_s_RPERR_ss_RP_001_001a}, $R^{-1}_{(1;\Delta x)}$ \eqref{Eq_Def_AELRP_s_RPERR_ss_RP_001_001c} is defined by
\begin{equation}
[R^{-1}_{(1;\Delta x)}(h)](x):=\dfrac{1}{\Delta x}\int_{x-\frac{1}{2}\Delta x}^{x+\frac{1}{2}\Delta x}{h(\zeta)d\zeta}
\quad\left\{\begin{array}{l}\forall x\in[a,b]\\
                            ~                \\
                            \left\{h:[a-\frac{1}{2}\Delta x,b+\frac{1}{2}\Delta x]\longrightarrow{\mathbb R}\;
                                  \Big|\displaystyle \int_{x-\frac{1}{2}\Delta x}^{x+\frac{1}{2}\Delta x} {h(\zeta)d\zeta}\;\in{\mathbb R}\;\forall x\in[a,b]\right\}\\\end{array}\right.
                                                                                                       \label{Eq_AELRP_s_RPERR_ss_RP_001}
\end{equation}
{\em ie} $R^{-1}_{(1;\Delta x)}$ is a mapping applicable to all real functions defined on $[a-\frac{1}{2}\Delta x,b+\frac{1}{2}\Delta x]\subset{\mathbb R}$
for which the integral \eqref{Eq_Def_AELRP_s_RPERR_ss_RP_001_001a} exists. Then, $R_{(1;\Delta x)}$ is defined as the inverse mapping of $R^{-1}_{(1;\Delta x)}$ \eqref{Eq_AELRP_s_RPERR_ss_RP_001},
assuming that the inverse mapping exists ({\em cf} \rmkrefnp{Rmk_AELRP_s_RPERR_ss_D_002}).
%-----------------------------------------------------------------------------------------------------------------------------------
%
\begin{lemma}[Reconstruction]
\label{Lem_AELRP_s_RPERR_ss_RP_001}
Consider the functions $f(x)$ and $h(x)$ constituting a reconstruction pair on $[a,b]\subset{\mathbb R}$~\defref{Def_AELRP_s_RPERR_ss_RP_001}.
Assume that $f(x)$ and $h(x)$ are of class $C^N$ \textup{(}$N\in{\mathbb N}$\textup{)} on the interval $I=[a-\frac{1}{2}\Delta x,b+\frac{1}{2}\Delta x]\subset{\mathbb R}$.
Then
\begin{equation}
f^{(n)}(x)=\dfrac{h^{(n-1)}(x+\frac{1}{2}\Delta x)-h^{(n-1)}(x-\frac{1}{2}\Delta x)}{\Delta x}\;\forall x\in[a,b]\;\forall n\in\{1,\cdots,N\}
                                                                                                       \label{Eq_Lem_AELRP_s_RPERR_ss_RP_001_001}
\end{equation}
\end{lemma}
%
%-----------------------------------------------------------------------------------------------------------------------------------
%-----------------------------------------------------------------------------------------------------------------------------------
%
\begin{proof}
Direct differentiation of~\eqref{Eq_Def_AELRP_s_RPERR_ss_RP_001_001a}, yields
\begin{equation}
f'(x)=\dfrac{h(x+\frac{1}{2}\Delta x)-h(x-\frac{1}{2}\Delta x)}{\Delta x}\;\forall x\in[a,b]
                                                                                                       \label{Eq_Lem_AELRP_s_RPERR_ss_RP_001_002}
\end{equation}
by application of the Leibniz rule~\cite[pp. 411--412]{Zorich_2004b}, and taking into account that $\Delta x$ is constant $\forall x$.
Successive differentiation of~\eqref{Eq_Lem_AELRP_s_RPERR_ss_RP_001_002} yields~\eqref{Eq_Lem_AELRP_s_RPERR_ss_RP_001_001}.\qed
\end{proof}
%
%-----------------------------------------------------------------------------------------------------------------------------------

All reconstruction-based approaches~\cite{Harten_Osher_1987a,
                                          Harten_Engquist_Osher_Chakravarthy_1987a,
                                          Liu_Osher_Chan_1994a,
                                          Jiang_Shu_1996a,
                                          Balsara_Shu_2000a,
                                          Titarev_Toro_2004a,
                                          Henrick_Aslam_Powers_2005a,
                                          Dumbser_Kaser_2007a,
                                          Borges_Carmona_Costa_Don_2008a}
for the numerical approximation of \tsc{pde}s
are based on, or can be shown to be related to, \lemrefnp{Lem_AELRP_s_RPERR_ss_RP_001}.
These relations~\eqref{Eq_Lem_AELRP_s_RPERR_ss_RP_001_001} are exact relations concerning the continuous functions $f$ and $h$. When
$f(x)$ and $h(x)$ are numerically approximated consistently, {\em ie} in a way satisfying~\eqref{Eq_Def_AELRP_s_RPERR_ss_RP_001_001} up to a given order $\Delta x^{M+1}$, then~\eqref{Eq_Lem_AELRP_s_RPERR_ss_RP_001_001}
are satisfied up to some order $\leq M+1$.
%-----------------------------------------------------------------------------------------------------------------------------------
%
\begin{definition}[Lagrange reconstructing polynomial]
\label{Def_AELRP_s_RPERR_ss_RP_002}
Let $p_f$ be the Lagrange interpolating polynomial of the function $f$ on the arbitrary stencil $\left\{i-M_-,\cdots,i+M_+\right\}$ of $M+1$ equidistant points ($M:=M_-+M_+$) around point $i$.
Its reconstruction pair \defref{Def_AELRP_s_RPERR_ss_RP_001} will be called the Lagrange reconstructing polynomial on the stencil $\left\{i-M_-,\cdots,i+M_+\right\}$.\qed
\end{definition}
%
%-----------------------------------------------------------------------------------------------------------------------------------
%-----------------------------------------------------------------------------------------------------------------------------------
%
\begin{remark}[Homogeneous grid]
\label{Rmk_AELRP_s_RPERR_ss_RP_001}
The basic relations underlying reconstruction, which are given in \lemrefnp{Lem_AELRP_s_RPERR_ss_RP_001}, hold iff $\Delta x ={\rm const}$, {\em ie}, when used as basis for the numerical approximation
of $f'(x)$, these relations are only applicable on a homogeneous grid. In the case of an inhomogeneous grid, where the spacing $\Delta x(x)$ is a function of position ($\Delta x: {\mathbb R}\longrightarrow{\mathbb R}_{>0}$)
these relations should be modified to include $\Delta x'$ and $(\partial_{\Delta x} h)\Delta x'$. The general case of an inhomogeneous grid requires specific study.\qed
\end{remark}
%
%-----------------------------------------------------------------------------------------------------------------------------------

%-----------------------------------------------------------------------------------------------------------------------------------
%
%
%
%
%
\subsection{Deconvolution}\label{AELRP_s_RPERR_ss_D}
%
%
%
%
%
%-----------------------------------------------------------------------------------------------------------------------------------

Obviously, the relations between $f$ and $h$~\lemref{Lem_AELRP_s_RPERR_ss_RP_001} imply that
the Taylor-polynomials of $f(x)$ can be expressed with respect to the derivatives $h^{(n)}(x\pm\frac{1}{2}\Delta x)$,
which can themselves be replaced by Taylor-polynomials of $h(x)$.
We have
%-----------------------------------------------------------------------------------------------------------------------------------
%
\begin{lemma}[Deconvolution of $h=R_{(1;\Delta x)}(f)$]
\label{Lem_AELRP_s_RPERR_ss_D_001}
Let $f(x)$ and $h(x)=[R_{(1;\Delta x)}(f)](x)$ be a reconstruction pair \defref{Def_AELRP_s_RPERR_ss_RP_001},
satisfying the conditions of \lemrefnp{Lem_AELRP_s_RPERR_ss_RP_001}. Then $\forall N_\tsc{tj}\in{\mathbb N}:N_\tsc{tj}<N-1$
\begin{subequations}
                                                                                                       \label{Eq_Lem_AELRP_s_RPERR_ss_D_001_001}
\begin{alignat}{6}
f^{(n)}(x)=&\sum_{\ell=0}^{\lfloor\frac{N_\tsc{tj}}{2}\rfloor}\dfrac{\Delta x^{2\ell}     }
                                                                    {2^{2\ell}\;(2\ell+1)!} h^{(n+2\ell)}(x)+O(\Delta x^{2\lfloor\frac{N_\tsc{tj}}{2}\rfloor+2})
&\;\begin{array}{c}\forall x\in[a,b]                                                \\
                   \forall n\in{\mathbb N}_0:n<N-2\lfloor\frac{N_\tsc{tj}}{2}\rfloor\\\end{array}
                                                                                                       \label{Eq_Lem_AELRP_s_RPERR_ss_D_001_001a}
\end{alignat}
Inversely,
\begin{alignat}{6}
h^{(n)}(x)                       =&\sum_{\ell=0}^{\lfloor\frac{N_\tsc{tj}}{2}\rfloor}\tau_{2\ell}\Delta x^{2\ell}f^{(n+2\ell)}(x)+O(\Delta x^{2\lfloor\frac{N_\tsc{tj}}{2}\rfloor+2})
&\;\begin{array}{c}\forall x\in[a,b]                                                \\
                   \forall n\in{\mathbb N}_0:n<N-2\lfloor\frac{N_\tsc{tj}}{2}\rfloor\\\end{array}
                                                                                                       \label{Eq_Lem_AELRP_s_RPERR_ss_D_001_001b}
\end{alignat}
where the numbers $\tau_{2\ell}$ \tabref{Tab_Lem_AELRP_s_RPERR_ss_D_001_001} are defined by the recurrence relations
\begin{alignat}{6}
\tau_0 &=& 1
&\qquad;\qquad&
\tau_{2k} &=& \sum_{s=0}^{k-1}\frac{-\tau_{2s}           }
                                   {2^{2k-2s}\;(2k-2s+1)!}
          &=& \sum_{s=1}^{k}  \frac{-\tau_{2k-2s}  }
                                   {2^{2s}\;(2s+1)!} &\qquad k > 0
                                                                                                       \label{Eq_Lem_AELRP_s_RPERR_ss_D_001_001c}
\end{alignat}
\end{subequations}
\end{lemma}
%
%-----------------------------------------------------------------------------------------------------------------------------------
%-----------------------------------------------------------------------------------------------------------------------------------
%
\begin{proof}
Approximating $h(\zeta)$ (which was assumed to be of class $C^N$ in \lemrefnp{Lem_AELRP_s_RPERR_ss_RP_001}) in \eqref{Eq_Def_AELRP_s_RPERR_ss_RP_001_001}
by the corresponding Taylor-polynomial (Taylor-jet) of order $N_\tsc{tj}$~\cite[pp. 219--232]{Zorich_2004a} around $\zeta=x$ yields, $\forall N_\tsc{tj}\in{\mathbb N}: N_\tsc{tj}<N-1$,
\begin{alignat}{6}
f(x)=&\dfrac{1}{\Delta x}\int_{x-\frac{1}{2}\Delta x}^{x+\frac{1}{2}\Delta x}{\left\lgroup\left(
                                                                              \sum_{\ell=0}^{N_\tsc{tj}}\dfrac{(\zeta-x)^\ell}
                                                                                                              {\ell!}       h^{(\ell)}(x)\right)
                                                                             +O\left((\zeta-x)^{N_\tsc{tj}+1}\right)\right\rgroup d\zeta}
                                                                                                       \notag\\
    =&\dfrac{1}{\Delta x}\int_{x-\frac{1}{2}\Delta x}^{x+\frac{1}{2}\Delta x}{\left(\sum_{\ell=0}^{N_\tsc{tj}}\dfrac{(\zeta-x)^\ell}
                                                                                                                    {\ell!}       h^{(\ell)}(x)\right)d\zeta}
                                                                              +O(\Delta x^{N_\tsc{tj}+1})
                                                                                                       \notag\\
    =&\dfrac{1}{\Delta x}\sum_{\ell=0}^{N_\tsc{tj}}\left(\int_{-\frac{1}{2}\Delta x}^{\frac{1}{2}\Delta x}{\dfrac{\eta^\ell}
                                                                                                                 {\ell!}    d\eta}\right)h^{(\ell)}(x)
                                                                              +O(\Delta x^{N_\tsc{tj}+1})
                                                                                                       \notag\\
    =&\dfrac{1}{\Delta x}\sum_{\ell=0}^{N_\tsc{tj}}\left(\dfrac{\Delta x^{\ell+1}}
                                                               {2^{\ell}\;(\ell+1)!}\frac{1-(-1)^{\ell+1}}{2}\right)h^{(\ell)}(x)
                                                                              +O(\Delta x^{N_\tsc{tj}+1})
&\qquad\forall x\in[a,b]
                                                                                                       \label{Eq_Lem_AELRP_s_RPERR_ss_D_001_002}
\end{alignat}
and since $\forall k\in{\mathbb N}_0$
\begin{subequations}
                                                                                                       \label{Eq_Lem_AELRP_s_RPERR_ss_D_001_003}
\begin{alignat}{6}
\ell+1=2k+1         &\quad (k\in{\mathbb N}_0)\quad&\;\Longrightarrow\;& 1-(-1)^{\ell+1}=2
                                                                                                       \label{Eq_Lem_AELRP_s_RPERR_ss_D_001_003a}\\
\ell+1=2k           &\quad (k\in{\mathbb N}_0)\quad&\;\Longrightarrow\;& 1-(-1)^{\ell+1}=0
                                                                                                       \label{Eq_Lem_AELRP_s_RPERR_ss_D_001_003b}
\end{alignat}
\end{subequations}
we obtain
\begin{equation}
f(x)=\sum_{\ell=0}^{\lfloor\frac{N_\tsc{tj}}{2}\rfloor}\dfrac{\Delta x^{2\ell}     }
                                                             {2^{2\ell}\;(2\ell+1)!} h^{(2\ell)}(x)+O(\Delta x^{2\lfloor\frac{N_\tsc{tj}}{2}\rfloor+2})
                                                                                                       \label{Eq_Lem_AELRP_s_RPERR_ss_D_001_004}
\end{equation}
which is \eqref{Eq_Lem_AELRP_s_RPERR_ss_D_001_001a} for $n=0$.
Successive differentiation of~\eqref{Eq_Lem_AELRP_s_RPERR_ss_D_001_004} by $x$ yields~\eqref{Eq_Lem_AELRP_s_RPERR_ss_D_001_001a}.

To invert~\eqref{Eq_Lem_AELRP_s_RPERR_ss_D_001_001a} we search for numbers $\tau_{2s}$ ($s\in{\mathbb N}_0)$ satisfying $\forall M_\tsc{tj}\in{\mathbb N}:M_\tsc{tj}<N-1$
and $\forall n \in{\mathbb N}_0:n<N-2\lfloor \frac{M_\tsc{tj}}{2}\rfloor$
\begin{alignat}{6}
 h^{(n)}(x)=&\sum_{s=0}^{M_\tsc{tj}}\tau_{2s}\Delta x^{2s}f^{(n+2s)}(x)+O(\Delta x^{2M_\tsc{tj}+2})
                                                                                                       \notag\\
           =&\sum_{s=0}^{M_\tsc{tj}}\left(\sum_{\ell=0}^{M_\tsc{tj}}\dfrac{\Delta x^{2\ell}}
                                                                          {2^{2\ell}\;(2\ell+1)!} h^{(n+2s+2\ell)}(x)+O(\Delta x^{2M_\tsc{tj}+2})\right)\tau_{2s}\Delta x^{2s}+O(\Delta x^{2M_\tsc{tj}+2})
                                                                                                       \notag\\
           =&\sum_{s=0}^{M_\tsc{tj}}\sum_{\ell=0}^{M_\tsc{tj}} \left(\dfrac{\tau_{2s}\Delta x^{2s+2\ell}}
                                                                                              {2^{2\ell}\;(2\ell+1)!}\;h^{(n+2s+2\ell)}(x)\right)
                                                                              +O(\Delta x^{2M_\tsc{tj}+2})
                                                                                                       \notag\\
           =&\sum_{k=0}^{2M_\tsc{tj}}\left(\sum_{s=\max(0,k-M_\tsc{tj})}^{\min(k,M_\tsc{tj})}\dfrac{\tau_{2s}          }
                                                                                                   {2^{2k-2s}\;(2k-2s+1)!}\right)\;\Delta x^{2k}\;h^{(n+2k)}(x)
                                                                              +O(\Delta x^{2M_\tsc{tj}+2})
                                                                                                       \notag\\
           =&\sum_{k=0}^{ M_\tsc{tj}}\left(\sum_{s=0}^{k}\dfrac{\tau_{2s}          }
                                                               {2^{2k-2s}\;(2k-2s+1)!}\right)\;\Delta x^{2k}h\;^{(n+2k)}(x)
                                                                              +O(\Delta x^{M_\tsc{tj}+2})
                                                                                                       \label{Eq_Lem_AELRP_s_RPERR_ss_D_001_005}
\end{alignat}
because of \eqref{Eq_AELRP_s_AppendixA_003}. \eqref{Eq_Lem_AELRP_s_RPERR_ss_D_001_005} holds, provided that ($\delta_{k0}$ is the Kronecker $\delta$)
\begin{equation}
\sum_{s=0}^{k}\dfrac{\tau_{2s}           }
                    {2^{2k-2s}\;(2k-2s+1)!} = \delta_{k0} \qquad \forall k \in {\mathbb N}_0
                                                                                                       \label{Eq_Lem_AELRP_s_RPERR_ss_D_001_006}
\end{equation}
which is satisfied if the numbers $\tau_{2k}$ are defined by~\eqref{Eq_Lem_AELRP_s_RPERR_ss_D_001_001c}.
Truncating~\eqref{Eq_Lem_AELRP_s_RPERR_ss_D_001_004} to $O(\Delta x^{2\lfloor\frac{N_\tsc{tj}}{2}\rfloor})$
yields~\eqref{Eq_Lem_AELRP_s_RPERR_ss_D_001_001b}.
The remainder in \eqref{Eq_Lem_AELRP_s_RPERR_ss_D_001_001a} and \eqref{Eq_Lem_AELRP_s_RPERR_ss_D_001_001b} is $O(\Delta x^{2\lfloor\frac{N_\tsc{tj}}{2}\rfloor+2})$,
because only even powers of $\Delta x$ appear in these expressions.
\qed
\end{proof}
%
%-----------------------------------------------------------------------------------------------------------------------------------

%-----------------------------------------------------------------------------------------------------------------------------------
%
\begin{table}[h!]
\caption{Numbers $\tau_{n}$~\eqref{Eq_Thm_AELRP_s_RPERR_ss_GFtaunRPexp_001_001c} satisfying recurrence~\eqref{Eq_Lem_AELRP_s_RPERR_ss_D_001_001c}, for $0\leq n\leq 21$.}
\begin{center}
\input{table_tauk}
\end{center}
\label{Tab_Lem_AELRP_s_RPERR_ss_D_001_001}
\end{table}
%
%-----------------------------------------------------------------------------------------------------------------------------------
%-----------------------------------------------------------------------------------------------------------------------------------
%
\begin{remark}[Relation to previous work~\cite{Harten_Engquist_Osher_Chakravarthy_1987a,
                                               Bianco_Puppo_Russo_1999a}]
\label{Rmk_AELRP_s_RPERR_ss_D_001}
The results in \lemrefnp{Lem_AELRP_s_RPERR_ss_D_001} expressing the derivatives of the sliding cell-averages $f(x)$ with respect to the derivatives of the function $h(x)=[R_{(1;\Delta x)}(f)](x)$, are straightforward.
In particular \eqref{Eq_Lem_AELRP_s_RPERR_ss_D_001_001a} corresponds to \cite[(2.15), p. 299]{Bianco_Puppo_Russo_1999a}.
The new results of \lemrefnp{Lem_AELRP_s_RPERR_ss_D_001} are the inversion relations \eqref{Eq_Lem_AELRP_s_RPERR_ss_D_001_001b},
which are based on the introduction of the numbers $\tau_n$ \eqref{Eq_Lem_AELRP_s_RPERR_ss_D_001_001c}.
These results are the general explicit solution of the linear system written in Harten \etal~\cite[(3.13b), p. 244]{Harten_Engquist_Osher_Chakravarthy_1987a},
and provide the exact deconvolution relation between $f(x)$ and $[R_{(1;\Delta x)}(f)](x)$~\defref{Def_AELRP_s_RPERR_ss_RP_001},
in the case of a homogeneous ($\Delta x={\rm const}$) grid. The general case of an inhomogeneous grid requires specific study.
The inversion relations \eqref{Eq_Lem_AELRP_s_RPERR_ss_D_001_001b} are the main building block of the present work,
as far as error analysis of the reconstruction is concerned. We will show that the numbers $\tau_n$ \eqref{Eq_Lem_AELRP_s_RPERR_ss_D_001_001c}
can also be defined by a generating function \thmref{Thm_AELRP_s_RPERR_ss_GFtaunRPexp_001}.\qed
\end{remark}
%
%-----------------------------------------------------------------------------------------------------------------------------------
%-----------------------------------------------------------------------------------------------------------------------------------
%
\begin{corollary}[Taylor-polynomial of $h(x+\xi\Delta x)$]
\label{Crl_AELRP_s_RPERR_ss_D_001}
Assume the conditions of \lemrefnp{Lem_AELRP_s_RPERR_ss_D_001}. Then
\begin{equation}
h(x+\xi\Delta x)= \sum_{s=0}^{N_\tsc{tj}}\left(\sum_{\ell=0}^{\lfloor\frac{s}{2}\rfloor}\frac{\tau_{2\ell}\; \xi^{s-2\ell}}
                                                                                             {(s-2\ell)!                  }\right)\Delta x^s\;f^{(s)}(x)+O(\Delta x^{N_\tsc{tj}+1})
                                                                                                       \label{Eq_Crl_AELRP_s_RPERR_ss_D_001_001}
\end{equation}
\end{corollary}
%
%-----------------------------------------------------------------------------------------------------------------------------------
%-----------------------------------------------------------------------------------------------------------------------------------
%
\begin{proof}
Since
\begin{subequations}
                                                                                                       \label{Eq_Crl_AELRP_s_RPERR_ss_D_001_002}
\begin{alignat}{6}
2\left\lfloor\frac{n}{2}\right\rfloor+2=\left\{\begin{array}{cclcl}n+1&\quad&\forall n=2k-1&\quad&k\in{\mathbb N}\\
                                                                   n+2&\quad&\forall n=2k  &\quad&k\in{\mathbb N}\\\end{array}\right.
                                                                                                       \label{Eq_Crl_AELRP_s_RPERR_ss_D_001_002a}
\end{alignat}
\eqref{Eq_Lem_AELRP_s_RPERR_ss_D_001_001b} can be rewritten as
\begin{alignat}{6}
\frac{\Delta x^m h^{(m)}(x)}
     {m!                   }=&\sum_{\ell=0}^{\lfloor\frac{N_\tsc{tj}-m}{2}\rfloor}\frac{\tau_{2\ell}(m+2\ell)!}
                                                                                     {m!                    }\frac{\Delta x^{m+2\ell}f^{(m+2\ell)}(x)}
                                                                                                                  {(m+2\ell)!                        }+O(\Delta x^{2\lfloor\frac{N_\tsc{tj}}{2}\rfloor+2})
                                                                                                       \label{Eq_Crl_AELRP_s_RPERR_ss_D_001_002b}
\end{alignat}
\end{subequations}
In that form~\eqref{Eq_Crl_AELRP_s_RPERR_ss_D_001_002b} we have a relation between the coefficients of the Taylor-polynomials of $f(x+\xi\Delta x)$ and of $h(x+\xi\Delta x)$,
expressed in powers of $\xi$. In particular, using~\eqref{Eq_Crl_AELRP_s_RPERR_ss_D_001_002b}, we have
\begin{subequations}
                                                                                                       \label{Eq_Crl_AELRP_s_RPERR_ss_D_001_003}
\begin{alignat}{6}
h(x+\xi\Delta x)&=\sum_{m=0}^{N_\tsc{tj}}\frac{\xi^m\;\Delta x^m\;h^{(m)}(x)}
                                              {m!                           }+O(\Delta x^{N_\tsc{tj}+1})
                                                                                                       \notag\\
                &=\sum_{m=0}^{N_\tsc{tj}}\sum_{\ell=0}^{\lfloor\frac{N_\tsc{tj}-m}{2}\rfloor}\frac{\tau_{2\ell}\;\Delta x^{m+2\ell}\;f^{(m+2\ell)}(x)}
                                                                                                  {m!                                                }\xi^m+O(\Delta x^{N_\tsc{tj}+1})
                                                                                                       \label{Eq_Crl_AELRP_s_RPERR_ss_D_001_003a}\\
                &=\sum_{s=0}^{N_\tsc{tj}}\sum_{\ell=0}^{\lfloor\frac{s}{2}\rfloor}\frac{\tau_{2\ell}\;\Delta x^s\;f^{(s)}(x)}
                                                                                       {(s-2\ell)!                          }\xi^{s-2\ell}+O(\Delta x^{N_\tsc{tj}+1})
                                                                                                       \label{Eq_Crl_AELRP_s_RPERR_ss_D_001_003b}
\end{alignat}
\end{subequations}
where we used \eqref{Eq_AELRP_s_AppendixA_003} and \eqref{Eq_AELRP_s_AppendixA_002}, and the fact that $N_\tsc{tj}+1\leq2\lfloor\frac{N_\tsc{tj}}{2}\rfloor+2$.
This completes the proof.\qed
\end{proof}
%
%-----------------------------------------------------------------------------------------------------------------------------------

This expression \eqref{Eq_Crl_AELRP_s_RPERR_ss_D_001_001} is useful
in computing the error of numerical approximations to $h(x)$~\prpref{Prp_AELRP_s_EPR_ss_AELPR_001}.
%-----------------------------------------------------------------------------------------------------------------------------------
%
\begin{remark}[Existence and uniqueness]
\label{Rmk_AELRP_s_RPERR_ss_D_002}
From~\defrefnp{Def_AELRP_s_RPERR_ss_RP_001} it follows immediately (proof by contradiction) that every reconstruction pair $h=R_{(1;\Delta x)}(f)$, with $h(x)$ continuous, if it exists, is unique.
For every $h(x)$ analytic in $I$ with radius of convergence $r_{\tsc{c}_h}(x)$,
the series~\eqref{Eq_Lem_AELRP_s_RPERR_ss_D_001_001a} with $n=0$ converges, as $N_\tsc{tj}\longrightarrow\infty$, $\forall\Delta x\in\Big(0,2 r_{\tsc{c}_h}(x)\Big)$, so that (because of uniqueness), for every analytic function $h(x)$
there exists a unique function $f=R^{-1}_{(1;\Delta x)}(h)$.
Whether the converse is always true, is an open question. Assuming $f(x)$ analytic in $I$ with radius of convergence $r_{\tsc{c}_f}(x)$, does not automatically imply the convergence of \eqref{Eq_Lem_AELRP_s_RPERR_ss_D_001_001b}
with $n=0$ as $N_\tsc{tj}\longrightarrow\infty$, because $\lim_{n\rightarrow\infty}(\tau_{2n}(2n)!)=\infty$. The necessary conditions of existence require further study.
Nonetheless, since $\lim_{n\rightarrow\infty}\tau_{2n}=0$ \tabref{Tab_Lem_AELRP_s_RPERR_ss_D_001_001} and $\tau_{2n}\tau_{2n+2}<0$ $\forall n\in{\mathbb N}_0$ \tabref{Tab_Lem_AELRP_s_RPERR_ss_D_001_001},
the class of functions $f(x)$ for which ~\eqref{Eq_Lem_AELRP_s_RPERR_ss_D_001_001b} with $n=0$ is convergent as $N_\tsc{tj}\longrightarrow\infty$ is not empty.
It is easy to verify that most of the basic functions $f(x)$ have reconstruction pairs $h=R_{(1;\Delta x)}(f)$, as do all polynomials of finite degree (\S\ref{AELRP_s_RP_ss_PRP}).
Whenever any of the series~\eqref{Eq_Lem_AELRP_s_RPERR_ss_D_001_001} converges as $N_\tsc{tj}\longrightarrow\infty$, the upper limit of the sums can be readily replaced by $\infty$, to yield complete
converging expansions (power-series).
The Godunov approach~\cite{Toro_1997a} to hyperbolic conservation laws $\partial_t u+\partial_x F=0$ \eqref{Eq_AELRP_s_I_001},
is based on space-time averaging of the \tsc{pde} \eqref{Eq_AELRP_s_I_001},
to obtain the corresponding \tsc{pde}, $\partial_t\bar u+\partial_x\bar F=0$ \eqref{Eq_AELRP_s_I_005},
for the cell-averages $\bar u$ \eqref{Eq_AELRP_s_I_003}.
Therefore, with respect to the notation used in \defrefnp{Def_AELRP_s_RPERR_ss_RP_001},
$\bar u$ corresponds to $f$ and $u$ corresponds to $h$. In the context of reconstruction procedures~\cite{Harten_Engquist_Osher_Chakravarthy_1987a,
                                                                                                          Liu_Osher_Chan_1994a,
                                                                                                          Jiang_Shu_1996a,
                                                                                                          Shu_1998a,
                                                                                                          Shu_2009a}
for the discretization of hyperbolic conservation laws, the existence of the solution (integrable function) $u$ ({\em ie} $h$) is assumed,
so that the existence of the sliding-averages $\bar u$ ({\em ie} $f$) follows~\rmkref{Rmk_AELRP_s_RPERR_ss_D_002}.
Hence, the results obtained in~\S\ref{AELRP_s_RPERR} (where the existence of $h$ is assumed) are directly applicable to
the Godunov approach for the numerical computation of hyperbolic conservation laws.\qed
\end{remark}
%
%-----------------------------------------------------------------------------------------------------------------------------------

%-----------------------------------------------------------------------------------------------------------------------------------
%
%
%
%
%
\subsection{Generating function of $\tau_n$ and the reconstruction pair of $\exp(x):={\rm e}^x$}\label{AELRP_s_RPERR_ss_GFtaunRPexp}
%
%
%
%
%
%-----------------------------------------------------------------------------------------------------------------------------------

As mentioned above \rmkref{Rmk_AELRP_s_RPERR_ss_D_002} most of the basic functions have reconstruction pairs. The reconstruction pair of the exponential function plays an important role
in the reconstruction relations \lemref{Lem_AELRP_s_RPERR_ss_D_001}, because it defines the generating function of the numbers $\tau_n$ \tabref{Tab_Lem_AELRP_s_RPERR_ss_D_001_001}.

%-----------------------------------------------------------------------------------------------------------------------------------
%
\begin{theorem}[$R_{(1;\Delta x)}({\exp})$]
\label{Thm_AELRP_s_RPERR_ss_GFtaunRPexp_001}
The reconstruction pair of $\exp(x):={\rm e}^x$ is
\begin{subequations}
                                                                                                       \label{Eq_Thm_AELRP_s_RPERR_ss_GFtaunRPexp_001_001}
\begin{alignat}{6}
[R_{(1;\Delta x)}({\exp})](x)=\frac{\tfrac{1}{2}\Delta x}
                  {\sinh{\tfrac{1}{2}\Delta x}}{\rm e}^x=g_\tau(\Delta x){\rm e}^x
                                                                                                       \label{Eq_Thm_AELRP_s_RPERR_ss_GFtaunRPexp_001_001a}
\end{alignat}
where the function
\begin{alignat}{6}
g_\tau(x):=\frac{\tfrac{1}{2} x      }
                {\sinh{\tfrac{1}{2}x}}
                                                                                                       \label{Eq_Thm_AELRP_s_RPERR_ss_GFtaunRPexp_001_001b}
\end{alignat}
is the generating function of the numbers $\tau_n$ \tabref{Tab_Lem_AELRP_s_RPERR_ss_D_001_001} satisfying~\eqref{Eq_Lem_AELRP_s_RPERR_ss_D_001_001c}
\begin{alignat}{6}
\tau_{n}:=\frac{1}{n!}g_\tau^{(n)}(0)\qquad\forall n\in{\mathbb N}_0
                                                                                                       \label{Eq_Thm_AELRP_s_RPERR_ss_GFtaunRPexp_001_001c}
\end{alignat}
Furthermore
\begin{alignat}{6}
\tau_{2n+1}:=\frac{1}{(2n+1)!}g_\tau^{(2n+1)}(0) = 0\qquad\forall n\in{\mathbb N}_0
                                                                                                       \label{Eq_Thm_AELRP_s_RPERR_ss_GFtaunRPexp_001_001d}
\end{alignat}
\end{subequations}
\end{theorem}
%
%-----------------------------------------------------------------------------------------------------------------------------------
%-----------------------------------------------------------------------------------------------------------------------------------
%
\begin{proof}
From \eqref{Eq_Lem_AELRP_s_RPERR_ss_D_001_001b}, since ${\rm e}^x$ is of class $C^\infty$, we have $\forall N_\tsc{tj}\in{\mathbb N}$
\begin{alignat}{6}
[R_{(1;\Delta x)}({\rm exp})](x) = \sum_{n=0}^{N_\tsc{tj}}\tau_{2n}\Delta x^{2n}\frac{d^{2n}}{dx^{2n}}{\rm e}^x+O(\Delta x^{2N_\tsc{tj}+2})
                                 = \left(\sum_{n=0}^{N_\tsc{tj}}\tau_{2n}\Delta x^{2n}\right){\rm e}^x+O(\Delta x^{2N_\tsc{tj}+2})
                                                                                                       \label{Eq_Thm_AELRP_s_RPERR_ss_GFtaunRPexp_001_002}
\end{alignat}
Since $\lim_{n\rightarrow\infty}\tau_{2n}=0$ and $\tau_{2n}\tau_{2n+2}<0\;\forall n\in{\mathbb N}_0$,
the alternating ($\Delta x^{2n}>0\;\forall n\in{\mathbb N}_0$) series in~\eqref{Eq_Thm_AELRP_s_RPERR_ss_GFtaunRPexp_001_002} converges as $N_\tsc{tj}\longrightarrow\infty$,
at least $\forall\Delta x\in(0,1)$. Defining the function $g_\tau(x)$
\begin{equation}
g_\tau(x):=\sum_{n=0}^{\infty}\tau_{2n} x^{2n}
                                                                                                       \label{Eq_Thm_AELRP_s_RPERR_ss_GFtaunRPexp_001_003}
\end{equation}
suggests that $\exists\;g_\tau:{\mathbb R}\longrightarrow{\mathbb R}$ such that
\begin{equation}
[R_{(1;\Delta x)}({\rm exp})](x)=g_\tau(\Delta x) {\rm e}^x
                                                                                                       \label{Eq_Thm_AELRP_s_RPERR_ss_GFtaunRPexp_001_004}
\end{equation}
Using~\eqref{Eq_Thm_AELRP_s_RPERR_ss_GFtaunRPexp_001_004} in~\eqref{Eq_Def_AELRP_s_RPERR_ss_RP_001_001a}
\begin{equation}
{\rm e}^x=\dfrac{1}{\Delta x}\int_{x-\frac{1}{2}\Delta x}^{x+\frac{1}{2}\Delta x}{g_\tau(\Delta x){\rm e}^\zeta d\zeta}
         =\dfrac{1}{\Delta x}g_\tau(\Delta x)\left({\rm e}^{x+\frac{1}{2}\Delta x}-{\rm e}^{x-\frac{1}{2}\Delta x}\right)
                                                                                                       \label{Eq_Thm_AELRP_s_RPERR_ss_GFtaunRPexp_001_005}
\end{equation}
gives
\begin{equation}
g_\tau(\Delta x)=\frac{\Delta x}{{\rm e}^{\frac{1}{2}\Delta x}-{\rm e}^{-\frac{1}{2}\Delta x}}
                =\frac{\tfrac{1}{2}\Delta x}
                      {\sinh{\tfrac{1}{2}\Delta x}}
                                                                                                       \label{Eq_Thm_AELRP_s_RPERR_ss_GFtaunRPexp_001_006}
\end{equation}
proving~\eqref{Eq_Thm_AELRP_s_RPERR_ss_GFtaunRPexp_001_001a}.
It is a simple exercise to show that the function $g_\tau(x)$~\eqref{Eq_Thm_AELRP_s_RPERR_ss_GFtaunRPexp_001_001b} is continuous at $x=0$, and has
continuous derivatives of arbitrary order at $x=0$, satisfying
\begin{subequations}
                                                                                                       \label{Eq_Thm_AELRP_s_RPERR_ss_GFtaunRPexp_001_007}
\begin{alignat}{6}
g_\tau(0)=&1
                                                                                                       \label{Eq_Thm_AELRP_s_RPERR_ss_GFtaunRPexp_001_007a}\\
g_\tau^{(2n+1)}(0)=& 0\qquad\forall n\in{\mathbb N}_0
                                                                                                       \label{Eq_Thm_AELRP_s_RPERR_ss_GFtaunRPexp_001_007b}
\end{alignat}
\end{subequations}
Comparing the Taylor-series of $g_\tau(x)$~\eqref{Eq_Thm_AELRP_s_RPERR_ss_GFtaunRPexp_001_001b} with the series definition of $g_\tau(x)$~\eqref{Eq_Thm_AELRP_s_RPERR_ss_GFtaunRPexp_001_003},
and taking into account~\eqref{Eq_Thm_AELRP_s_RPERR_ss_GFtaunRPexp_001_007} proves~\eqref{Eq_Thm_AELRP_s_RPERR_ss_GFtaunRPexp_001_001c}.
\eqref{Eq_Thm_AELRP_s_RPERR_ss_GFtaunRPexp_001_007b} yields~\eqref{Eq_Thm_AELRP_s_RPERR_ss_GFtaunRPexp_001_001d}.\qed
\end{proof}
%
%-----------------------------------------------------------------------------------------------------------------------------------

%-----------------------------------------------------------------------------------------------------------------------------------
%
%
%
%
%
%
%
%
%
\section{Reconstruction of polynomials}\label{AELRP_s_RP}
%
%
%
%
%
%
%
%
%
%-----------------------------------------------------------------------------------------------------------------------------------

Reconstruction of polynomials~\defref{Def_AELRP_s_RPERR_ss_RP_001} is the basis of \tsc{eno}~\cite{Harten_Osher_1987a,
                                                                                                  Harten_Engquist_Osher_Chakravarthy_1987a}
and \tsc{weno}~\cite{Liu_Osher_Chan_1994a,
                     Jiang_Shu_1996a,
                     Balsara_Shu_2000a,
                     Henrick_Aslam_Powers_2005a,
                     Borges_Carmona_Costa_Don_2008a,
                     Gerolymos_Senechal_Vallet_2009a}
reconstructions.
We investigate in detail the coefficients of polynomial (\S\ref{AELRP_s_RP_ss_PRP}) reconstruction pairs~\defref{Def_AELRP_s_RPERR_ss_RP_001}.

%-----------------------------------------------------------------------------------------------------------------------------------
%
%
%
%
%
\subsection{Polynomial reconstruction pair}\label{AELRP_s_RP_ss_PRP}
%
%
%
%
%
%-----------------------------------------------------------------------------------------------------------------------------------

In this section we consider the case where either $f(x)$ or $h(x)$ in~\defrefnp{Def_AELRP_s_RPERR_ss_RP_001}
is a polynomial.
%-----------------------------------------------------------------------------------------------------------------------------------
%
\begin{lemma}[{\rm Polynomial reconstruction pair}]
\label{Lem_AELRP_s_RP_ss_PRP_001}
Let $p_h(x,x_i,\Delta x)$ be a polynomial of degree $M$
\begin{subequations}
                                                                                                       \label{Eq_Lem_AELRP_s_RP_ss_PRP_001_001}
\begin{alignat}{6}
p_h(x;x_i,\Delta x)&:=&\sum_{m=0}^{M} c_{h_m}\left(\frac{x-x_i}{\Delta x}\right)^m
                                                                                                       \label{Eq_Lem_AELRP_s_RP_ss_PRP_001_001a}
\end{alignat}
Then $p_f(x;x_i,\Delta x)$ defined by~\defref{Def_AELRP_s_RPERR_ss_RP_001}
\begin{alignat}{6}
p_f(x;x_i,\Delta x)&:=&\dfrac{1}{\Delta x}\int_{x-\frac{1}{2}\Delta x}^{x+\frac{1}{2}\Delta x}p_h(\zeta;x_i,\Delta x)d\zeta
                                                                                                       \label{Eq_Lem_AELRP_s_RP_ss_PRP_001_001b}
\end{alignat}
is a polynomial also of degree $M$, with coefficients $c_{f_m}$ which can be
computed from the coefficients $c_{h_m}$ of $p_h(x;x_i,\Delta x)$
\begin{alignat}{6}
p_f(x;x_i,\Delta x)=&\sum_{m=0}^{M} c_{f_m}\left(\frac{x-x_i}{\Delta x}\right)^m
                                                                                                       \label{Eq_Lem_AELRP_s_RP_ss_PRP_001_001c}\\
c_{f_m}            =&\sum_{k=0}^{\lfloor\frac{M-m}{2}\rfloor}\frac{c_{h_{m+2k}}}
                                                                  {2^{2k}\;(2k+1)}\binom{m+2k}
                                                                                        {2k  }  &\qquad\forall m\in\{0,\cdots,M\}
                                                                                                       \label{Eq_Lem_AELRP_s_RP_ss_PRP_001_001d}\\
m!\;c_{f_m}        =&\sum_{k=0}^{\lfloor\frac{M-m}{2}\rfloor}\frac{(m+2k)!        }
                                                                  {2^{2k}\;(2k+1)!}c_{h_{m+2k}} &\qquad\forall m\in\{0,\cdots,M\}
                                                                                                       \label{Eq_Lem_AELRP_s_RP_ss_PRP_001_001e}
\end{alignat}
Inversely, the coefficients $c_{h_m}$ of $p_h(x;x_i,\Delta x)$ can be computed from the coefficients $c_{f_m}$ of $p_f(x;x_i,\Delta x)$
\begin{alignat}{6}
c_{h_m}          =\frac{1}{m!}\sum_{k=0}^{\lfloor\frac{M-m}{2}\rfloor}\tau_{2k}\;c_{f_{m+2k}}\;(m+2k)!  \qquad\forall m\in\{0,\cdots,M\}
                                                                                                       \label{Eq_Lem_AELRP_s_RP_ss_PRP_001_001f}
\end{alignat}
where the numbers $\tau_{2k}$ \tabref{Tab_Lem_AELRP_s_RPERR_ss_D_001_001} are defined by~\eqref{Eq_Thm_AELRP_s_RPERR_ss_GFtaunRPexp_001_001c} and satisfy the recurrence~\eqref{Eq_Lem_AELRP_s_RPERR_ss_D_001_001c}.
\end{subequations}
\end{lemma}
%
%-----------------------------------------------------------------------------------------------------------------------------------
%-----------------------------------------------------------------------------------------------------------------------------------
%
\begin{proof}
Computing the integral in \eqref{Eq_Lem_AELRP_s_RP_ss_PRP_001_001b} gives 
\begin{alignat}{6}
p_f(x;x_i,\Delta x)= \int_{\frac{x}{\Delta x}-\frac{1}{2}}^{\frac{x}{\Delta x}+\frac{1}{2}}\left(\sum_{m=0}^{M} c_{h_m}\left(\zeta-\frac{x_i}{\Delta x}\right)^m\right) d\zeta
                   =&\sum_{m=0}^{M}\frac{c_{h_m}}{m+1}\left(\frac{x-x_i}{\Delta x}+\tfrac{1}{2}\right)^{m+1}
                   - \sum_{m=0}^{M}\frac{c_{h_m}}{m+1}\left(\frac{x-x_i}{\Delta x}-\tfrac{1}{2}\right)^{m+1}
                                                                                                       \notag\\
                   =&\sum_{m=0}^{M}\frac{c_{h_m}}{m+1}\left(\sum_{n=0}^{m+1}\binom{m+1}{n}\left(\frac{x-x_i}{\Delta x}\right)^n\frac{1}{2^{m-n}}\;\frac{1-(-1)^{m+1-n}}{2}\right)
                                                                                                       \notag\\
                   =&\sum_{m=0}^{M}\frac{c_{h_m}}{m+1}\left(\sum_{n=0}^{m}  \binom{m+1}{n}\left(\frac{x-x_i}{\Delta x}\right)^n\frac{1}{2^{m-n}}\;\frac{1-(-1)^{m+1-n}}{2}\right)
                                                                                                       \label{Eq_Lem_AELRP_s_RP_ss_PRP_001_002}
\end{alignat}
where in the last line of \eqref{Eq_Lem_AELRP_s_RP_ss_PRP_001_002} $\sum_{n=0}^{m+1}$ was changed to $\sum_{n=0}^{m}$ because $n=m+1\Longrightarrow1-(-1)^{m+1-n}=1-(-1)^0=0$.
This proves that both $p_h(x;x_i,\Delta x)$ and $p_f(x;x_i,\Delta x)$ are of degree $M$. Since
\begin{subequations}
                                                                                                       \label{Eq_Lem_AELRP_s_RP_ss_PRP_001_003}
\begin{alignat}{6}
m+1-n=2k+1         &\quad k\in{\mathbb N}_0&\;\Longrightarrow\;& 1-(-1)^{m+1-n}=2
                                                                                                       \label{Eq_Lem_AELRP_s_RP_ss_PRP_001_003a}\\
m+1-n=2k           &\quad k\in{\mathbb N}_0&\;\Longrightarrow\;& 1-(-1)^{m+1-n}=0
                                                                                                       \label{Eq_Lem_AELRP_s_RP_ss_PRP_001_003b}\\
0\leq n=m-2k\leq m &\quad k\in{\mathbb N}_0&\;\Longrightarrow\;& 0\leq2k\leq m    &\;\Longleftrightarrow\;&0\leq k\leq\lfloor\frac{m}{2}\rfloor
                                                                                                       \label{Eq_Lem_AELRP_s_RP_ss_PRP_001_003c}
\end{alignat}
\end{subequations}
upon substituting $2k:=m-n$, \eqref{Eq_Lem_AELRP_s_RP_ss_PRP_001_002} becomes
\begin{alignat}{6}
p_f(x;x_i,\Delta x)= \sum_{m=0}^{M}\frac{c_{h_m}}{m+1}\left(\sum_{k=0}^{\lfloor\frac{m}{2}\rfloor}\frac{1}{2^{2k}}\binom{m+1}{m-2k}\left(\frac{x-x_i}{\Delta x}\right)^{m-2k}\right)
                   = \sum_{m=0}^{M}\sum_{k=0}^{\lfloor\frac{m}{2}\rfloor}
                                                                        \frac{c_{h_m}}{2^{2k}\;(m+1)}\binom{m+1}{m-2k}\left(\frac{x-x_i}{\Delta x}\right)^{m-2k}
                                                                                                       \label{Eq_Lem_AELRP_s_RP_ss_PRP_001_004}
\end{alignat}
and, using \eqref{Eq_AELRP_s_AppendixA_003}, \eqref{Eq_Lem_AELRP_s_RP_ss_PRP_001_004} reads
\begin{equation}
p_f(x;x_i,\Delta x)=
\sum_{\ell=0}^{M}\left(\sum_{k=0}^{\lfloor\frac{M-\ell}{2}\rfloor}\frac{c_{h_{\ell+2k}}    }
                                                                       {2^{2k}\;(\ell+2k+1)}\binom{\ell+2k+1}
                                                                                                  {\ell     }\right)\left(\frac{x-x_i}{\Delta x}\right)^\ell
                                                                                                       \label{Eq_Lem_AELRP_s_RP_ss_PRP_001_005}
\end{equation}
Using the identity \eqref{Eq_AELRP_s_AppendixA_004} and changing the summation index $\ell$ to $m$ gives
\begin{equation}
p_f(x;x_i,\Delta x)=
\sum_{m=0}^{M}\left(\sum_{k=0}^{\lfloor\frac{M-m}{2}\rfloor}\frac{c_{h_{m+2k}}  }
                                                                 {2^{2k}\;(2k+1)}\binom{m+2k}
                                                                                       {2k  }\right)\left(\frac{x-x_i}{\Delta x}\right)^m
                                                                                                       \label{Eq_Lem_AELRP_s_RP_ss_PRP_001_006}
\end{equation}
which proves~\eqref{Eq_Lem_AELRP_s_RP_ss_PRP_001_001d}. In practice, the coefficients $c_{f_m}$ are computed by solving a Vandermonde system~\cite{Macon_Spitzbart_1958a},
and the linear system~\eqref{Eq_Lem_AELRP_s_RP_ss_PRP_001_001d} must be solved to compute the coefficients $c_{h_m}$~\cite{Harten_Engquist_Osher_Chakravarthy_1987a}.
The general solution can be obtained using backward substitution without making reference to the basic reconstruction relations
(\S\ref{AELRP_s_RPERR}). This alternative, matrix-algebra-oriented, proof of~\lemrefnp{Lem_AELRP_s_RP_ss_PRP_001} is given in \S\ref{AELRP_s_RP_ss_MIPPRPLem}.

However, the solution can be obtained immediately, by observing \eqref{Eq_Lem_AELRP_s_RP_ss_PRP_001_001e} that the relation between $c_{f_m} m!$ and $c_{h_{m+2k}} (m+2k)!$ in~\eqref{Eq_Lem_AELRP_s_RP_ss_PRP_001_001d}
is identical to the relation between $f^{(m)}(x)$ and $\Delta x^{2k}h^{(m+2k)}(x)$ in~\eqref{Eq_Lem_AELRP_s_RPERR_ss_D_001_001a},
with the only difference that the upper limit of the sum is finite. The inverse relation is exactly analogous to~\eqref{Eq_Lem_AELRP_s_RPERR_ss_D_001_001b},
because, using~\eqref{Eq_Lem_AELRP_s_RP_ss_PRP_001_001e} in the right-hand-side of~\eqref{Eq_Lem_AELRP_s_RP_ss_PRP_001_001f}
\begin{alignat}{6}
 &\frac{1}{m!}\sum_{s=0}^{\lfloor\frac{M-m}{2}\rfloor}\tau_{2s}\;c_{f_{m+2s}}\;(m+2s)!= 
  \frac{1}{m!}\sum_{s=0}^{\lfloor\frac{M-m}{2}\rfloor}\tau_{2s}\left(\sum_{\ell=0}^{\lfloor\frac{M-m}{2}\rfloor-s}\frac{(m+2s+2\ell)!\;c_{h_{m+2s+2\ell}}}
                                                                                                                      {2^{2\ell}\;(2\ell+1)!            }\right)
= \sum_{s=0}^{\lfloor\frac{M-m}{2}\rfloor}\sum_{\ell=0}^{\lfloor\frac{M-m}{2}\rfloor-s}\left(\tau_{2s}\frac{(m+2s+2\ell)!\;c_{h_{m+2s+2\ell}}}
                                                                                                           {2^{2\ell}\;(2\ell+1)!\;m!        }\right)
                                                                                                       \notag\\
=&\sum_{k=0}^{\lfloor\frac{M-m}{2}\rfloor}\left(\sum_{s=0}^{k}\frac{\tau_{2s}            }
                                                                   {2^{2k-2s}\;(2k-2s+1)!}\right)\frac{(m+2k)!\;c_{h_{m+2k}}}
                                                                                                     {m!                   }
= \sum_{k=0}^{\lfloor\frac{M-m}{2}\rfloor}\delta_{k0} \frac{(m+2k)!\;c_{h_{m+2k}}}
                                                           {m!                   }=c_{h_m}
                                                                                                       \label{Eq_Lem_AELRP_s_RP_ss_PRP_001_007}
\end{alignat}
where we used~\eqref{Eq_Lem_AELRP_s_RPERR_ss_D_001_006}, and~\eqref{Eq_AELRP_s_AppendixA_003} and \eqref{Eq_AELRP_s_AppendixA_002}.
This completes the proof.\qed
\end{proof}
%
%-----------------------------------------------------------------------------------------------------------------------------------

The extension of the above results~\lemref{Lem_AELRP_s_RP_ss_PRP_001} to infinite power-series (assuming that they are convergent) is straightforward.

%-----------------------------------------------------------------------------------------------------------------------------------
%
%
%
%
%
\subsection{Matrix inversion proof of~\lemrefnp{Lem_AELRP_s_RP_ss_PRP_001}}\label{AELRP_s_RP_ss_MIPPRPLem}
%
%
%
%
%
%-----------------------------------------------------------------------------------------------------------------------------------

In this section we summarize the matrix inversion relations which can be used for an alternative, matrix-algebra-oriented, proof \lemref{Lem_AELRP_s_RP_ss_MIPPRPLem_003} of \lemrefnp{Lem_AELRP_s_RP_ss_PRP_001}.
By \eqref{Eq_Lem_AELRP_s_RP_ss_PRP_001_001d} the coefficients $c_{f_n}$ of $p_f$ are expressed as linear combinations of the coefficients $c_{h_n}$ of $p_h$.
This system \eqref{Eq_Lem_AELRP_s_RP_ss_PRP_001_001d}, whose solution expresses $c_{h_n}$ as linear combinations of $c_{f_n}$
is the deconvolution linear system~\cite[(3.13b), p. 244]{Harten_Engquist_Osher_Chakravarthy_1987a}.\footnote{\label{ff_AELRP_s_RP_ss_MIPPRPLem_001}more precisely,
                                                                                                                                                   the system in ~\cite[(3.13b), p. 244]{Harten_Engquist_Osher_Chakravarthy_1987a} relates
                                                                                                                                                   $n!c_{f_n}$ with $m!c_{h_m}$.
                                                                                                             }
Since the summation relations~\eqref{Eq_Lem_AELRP_s_RP_ss_PRP_001_001d} involve
increments with step 2, we can split~\eqref{Eq_Lem_AELRP_s_RP_ss_PRP_001_001d} into 2 independent linear systems
\begin{subequations}
                                                                                                       \label{Eq_AELRP_s_RP_ss_MIPPRPLem_001}
\begin{alignat}{6}
&c_{f_{M-2\ell}}   &=&\sum_{k=0}^{\ell}\frac{c_{h_{M-2\ell+2k}}}
                                            {(2k+1)\;2^{2k}    }\binom{M-2\ell+2k}
                                                                      {2k        }   &\quad \ell=&0,\cdots,\lfloor\frac{M}{2}\rfloor
                                                                                                       \label{Eq_AELRP_s_RP_ss_MIPPRPLem_001a}\\
&c_{f_{M-1-2\ell}} &=&\sum_{k=0}^{\ell}\frac{c_{h_{M-1-2\ell+2k}}}
                                            {(2k+1)\;2^{2k}      }\binom{M-1-2\ell+2k}
                                                                        {2k          } &\quad \ell=&0,\cdots,\lfloor\frac{M-1}{2}\rfloor
                                                                                                       \label{Eq_AELRP_s_RP_ss_MIPPRPLem_001b}
\end{alignat}
\end{subequations}
for
$[c_{h_{M-2\lfloor\frac{M}{2}\rfloor}},\cdots,c_{h_{M}}]^\tsc{t}$~\eqref{Eq_AELRP_s_RP_ss_MIPPRPLem_001a} and for
$[c_{h_{M-1-2\lfloor\frac{M-1}{2}\rfloor}},\cdots,c_{h_{M-1}}]^\tsc{t}$~\eqref{Eq_AELRP_s_RP_ss_MIPPRPLem_001b}, respectively.
In matrix-form, we have
\begin{subequations}
                                                                                                       \label{Eq_AELRP_s_RP_ss_MIPPRPLem_002}
\begin{alignat}{6}
\underbrace{
\left[\begin{array}{ccccc}     1&       &\cdots&                                           &                                         \\
                          \vdots&\ddots &      &\vdots                                     &\vdots                                   \\
                               0&0      &1     &{\displaystyle\frac{1       }
                                                                   {(2+1)\;2^2}\binom{M-2}
                                                                                     {2  }}&{\displaystyle\frac{1       }
                                                                                                               {(4+1)\;2^4}\binom{M}
                                                                                                                                 {4}}\\
                               0&0      &\cdots&1                                           &{\displaystyle\frac{1         }
                                                                                                                {(2+1)\;2^2}\binom{M}
                                                                                                                                  {2}}\\
                               0&0      &\cdots&0                                           &1                                        \\\end{array}\right]}_{\displaystyle U_{(\lfloor\frac{M}{2}\rfloor,M)}}
\left[\begin{array}{l}c_{h_{M-2\lfloor\frac{M}{2}\rfloor}}\\
                                                    \vdots\\
                                               c_{h_{M-4}}\\
                                                          \\
                                               c_{h_{M-2}}\\
                                                          \\
                                                   c_{h_M}\\\end{array}\right]
=
\left[\begin{array}{l}c_{f_{M-2\lfloor\frac{M}{2}\rfloor}}\\
                                                    \vdots\\
                                               c_{f_{M-4}}\\
                                                          \\
                                               c_{f_{M-2}}\\
                                                          \\
                                                   c_{f_M}\\\end{array}\right]
                                                                                                       \label{Eq_AELRP_s_RP_ss_MIPPRPLem_002a}\\
\underbrace{
\left[\begin{array}{ccccc}     1&       &\cdots&         &                                           \\
                          \vdots&\ddots &      &\vdots   &\vdots                                     \\
                               0&0      &1     &\vdots   &{\displaystyle\frac{1       }
                                                                             {(4+1)\;2^4}\binom{M-1}
                                                                                               {4  }}\\
                               0&0      &\cdots&1        &{\displaystyle\frac{1       }
                                                                             {(2+1)\;2^2}\binom{M-1}
                                                                                               {2  }}\\
                               0&0      &\cdots&0        &1                                          \\\end{array}\right]}_{\displaystyle U_{(\lfloor\frac{M-1}{2}\rfloor,M-1)}}
\left[\begin{array}{l}c_{h_{M-1-2\lfloor\frac{M-1}{2}\rfloor}}\\
                                                        \vdots\\
                                                   c_{h_{M-5}}\\
                                                              \\
                                                   c_{h_{M-3}}\\
                                                              \\
                                                   c_{h_{M-1}}\\\end{array}\right]
=
\left[\begin{array}{l}c_{f_{M-1-2\lfloor\frac{M-1}{2}\rfloor}}\\
                                                        \vdots\\
                                                   c_{f_{M-5}}\\
                                                              \\
                                                   c_{f_{M-3}}\\
                                                              \\
                                                   c_{f_{M-1}}\\\end{array}\right]
                                                                                                       \label{Eq_AELRP_s_RP_ss_MIPPRPLem_002b}
\end{alignat}
\end{subequations}
where the  matrices $U_{(\lfloor\frac{M}{2}\rfloor,M)}$~\eqref{Eq_AELRP_s_RP_ss_MIPPRPLem_002a} and $U_{(\lfloor\frac{M-1}{2}\rfloor,M-1)}$~\eqref{Eq_AELRP_s_RP_ss_MIPPRPLem_002b}
are upper unitriangular~\cite{Golub_vanLoan_1989a}. The corresponding linear systems~\eqref{Eq_AELRP_s_RP_ss_MIPPRPLem_002}
can be solved using backward-substitution~\cite{Golub_vanLoan_1989a}. To obtain the general solution,
we initially remind, without going
into the details of a formal proof, a standard result of matrix calculus~\cite{Golub_vanLoan_1989a}, concerning the inverse of an upper unitriangular matrix.
%-----------------------------------------------------------------------------------------------------------------------------------
%
\begin{lemma}[Inverse of an upper unitriangular matrix]
\label{Lem_AELRP_s_RP_ss_MIPPRPLem_001}
Let $U\in{\mathbb R}^{n\times n}$ be an upper unitriangular matrix
\begin{subequations}
                                                                                                       \label{Eq_Lem_AELRP_s_RP_ss_MIPPRPLem_001_001}
\begin{alignat}{6}
u_{i,i}=&1\quad 1\leq i \leq n
                                                                                                       \label{Eq_Lem_AELRP_s_RP_ss_MIPPRPLem_001_001a}\\
u_{i,j}=&0\quad \begin{array}{c}j <   i       \\
                                1 <   i \leq n\\\end{array}
                                                                                                       \label{Eq_Lem_AELRP_s_RP_ss_MIPPRPLem_001_001b}\\
U=&\left[\begin{array}{ccccc}     1&u_{1,2}&\cdots&u_{1,n-1}&  u_{1,n}\\
                                  0&1      &\cdots&u_{2,n-1}&  u_{2,n}\\
                             \vdots&       &\ddots&\vdots   &\vdots   \\
                                  0&0      &\cdots&1        &u_{n-1,n}\\
                                  0&0      &\cdots&0        &1        \\\end{array}\right]
                                                                                                       \label{Eq_Lem_AELRP_s_RP_ss_MIPPRPLem_001_001c}
\end{alignat}
\end{subequations}
Its inverse $U^{-1}$ exists and is also an upper unitriangular matrix
\begin{subequations}
                                                                                                       \label{Eq_Lem_AELRP_s_RP_ss_MIPPRPLem_001_002}
\begin{alignat}{6}
\check u_{i,i}=&1\qquad 1\leq i \leq n
                                                                                                       \label{Eq_Lem_AELRP_s_RP_ss_MIPPRPLem_001_002a}\\
\check u_{i,j}=&0\qquad \begin{array}{c}j <   i       \\
                                        1 <   i \leq n\\\end{array}
                                                                                                       \label{Eq_Lem_AELRP_s_RP_ss_MIPPRPLem_001_002b}\\
U^{-1}=&\left[\begin{array}{ccccc}     1&\check u_{1,2}&\cdots&\check u_{1,n-1}&\check u_{1,n}  \\
                                       0&1             &\cdots&\check u_{2,n-1}&\check u_{2,n}  \\
                                  \vdots&              &\ddots&\vdots          &\vdots          \\
                                       0&0             &\cdots&1               &\check u_{n-1,n}\\
                                       0&0             &\cdots&0               &1        \\\end{array}\right]
                                                                                                       \label{Eq_Lem_AELRP_s_RP_ss_MIPPRPLem_001_002c}
\end{alignat}
whose nonzero elements $\check u_{i,j}$ ($j\geq i$) satisfy the recurrence relations
\begin{alignat}{6}
\check u_{n,n}=&1
                                                                                                       \label{Eq_Lem_AELRP_s_RP_ss_MIPPRPLem_001_002d}\\
\check u_{n-k,n-k+s}=&-\sum_{\ell=1}^{s}u_{n-k,n-k+\ell}\;\check u_{n-k+\ell,n-k+s}\qquad \begin{array}{c} 1\leq k < n\\
                                                                                                         1\leq s \leq k\\\end{array}
                                                                                                       \label{Eq_Lem_AELRP_s_RP_ss_MIPPRPLem_001_002e}
\end{alignat}
\end{subequations}
\end{lemma}
%
%-----------------------------------------------------------------------------------------------------------------------------------
%-----------------------------------------------------------------------------------------------------------------------------------
%
\begin{proof}
It is straightforward to show, by induction, that ${\rm det}U=1$.
The proof by induction of~\eqref{Eq_Lem_AELRP_s_RP_ss_MIPPRPLem_001_002} is a simple exercise of matrix calculus, directly obtained from the backward-substitution algorithm
for solving $Ux=b$~\cite{Golub_vanLoan_1989a}.\qed
\end{proof}
%
%-----------------------------------------------------------------------------------------------------------------------------------

This recurrence is applied to compute the inverse of the upper unitriangular matrices \eqref{Eq_AELRP_s_RP_ss_MIPPRPLem_002} of the linear system~\eqref{Eq_Lem_AELRP_s_RP_ss_PRP_001_001d}
of \lemrefnp{Lem_AELRP_s_RP_ss_PRP_001}.
 
%-----------------------------------------------------------------------------------------------------------------------------------
%
\begin{lemma}[Inverse of the matrices in~\lemrefnp{Lem_AELRP_s_RP_ss_PRP_001}]
\label{Lem_AELRP_s_RP_ss_MIPPRPLem_002}
Assume $N\leq\lfloor\frac{M}{2}\rfloor+1$. Let $U_{(N,M)}\in{\mathbb R}^{N\times N}$ be an upper unitriangular matrix whose elements are given by
\begin{subequations}
                                                                                                       \label{Eq_Lem_AELRP_s_RP_ss_MIPPRPLem_002_001}
\begin{alignat}{6}
\begin{array}{lcll}
(U_{(N,M)})_{N-\ell,N-\ell-k}&=&0                                                          &\;0\leq k\leq N-1-\ell\\
(U_{(N,M)})_{N-\ell,N-\ell}  &=&1                                                          &                      \\
(U_{(N,M)})_{N-\ell,N-\ell+k}&=&{\displaystyle\frac{1}{(2k+1)2^{2k}}\binom{M-2\ell+2k}{2k}}&\;0\leq k\leq \ell    \\\end{array}
\quad;\quad 0\leq \ell < N-1 \quad;\quad N\leq\lfloor\frac{M}{2}\rfloor+1
                                                                                                       \label{Eq_Lem_AELRP_s_RP_ss_MIPPRPLem_002_001a}
\end{alignat}
Its inverse $U_{(N,M)}^{-1}$ is also an upper unitriangular matrix whose elements are given by
\begin{alignat}{6}
\begin{array}{lcll}
(U^{-1}_{(N,M)})_{N-\ell,N-\ell-k}&=&0                                                    &\;0\leq k\leq N-1-\ell\\
(U^{-1}_{(N,M)})_{N-\ell,N-\ell}  &=&1                                                    &                      \\
(U^{-1}_{(N,M)})_{N-\ell,N-\ell+k}&=&{\displaystyle\tau_{2k}\frac{(M-2\ell+2k)!}{(M-2\ell)!}}&\;0\leq k\leq \ell    \\\end{array}
\quad;\quad 0\leq \ell < N-1 \quad;\quad N\leq\lfloor\frac{M}{2}\rfloor+1
                                                                                                       \label{Eq_Lem_AELRP_s_RP_ss_MIPPRPLem_002_001b}
\end{alignat}
where the numbers $\tau_{2k}$ \tabref{Tab_Lem_AELRP_s_RPERR_ss_D_001_001} are defined by the recurrence~\eqref{Eq_Lem_AELRP_s_RPERR_ss_D_001_001c}.
\end{subequations}
\end{lemma}
%
%-----------------------------------------------------------------------------------------------------------------------------------
%-----------------------------------------------------------------------------------------------------------------------------------
%
\begin{proof}
To simplify notation let $(U_{(N,M)})_{ij}=u_{ij}$ and $(U^{-1}_{(N,M)})_{ij}=\check u_{ij}$
By \lemrefnp{Lem_AELRP_s_RP_ss_MIPPRPLem_001} $U_{(N,M)}^{-1}$ is also an upper unitriangular matrix.
It is easy to verify, by straightforward computation, using~\eqref{Eq_Lem_AELRP_s_RP_ss_MIPPRPLem_001_002},
that~\eqref{Eq_Lem_AELRP_s_RP_ss_MIPPRPLem_002_001b} holds for $0\leq\ell\leq3$. To prove that~\eqref{Eq_Lem_AELRP_s_RP_ss_MIPPRPLem_002_001b} is valid for $0\leq\ell\leq N-1$,
by induction, suppose that~\eqref{Eq_Lem_AELRP_s_RP_ss_MIPPRPLem_002_001b} is valid for $1\leq\ell\leq m$. Then, from~\eqref{Eq_Lem_AELRP_s_RP_ss_MIPPRPLem_001_002e}
\begin{subequations}
                                                                                                       \label{Eq_Lem_AELRP_s_RP_ss_MIPPRPLem_002_003}
\begin{alignat}{6}
\check u_{N-(m+1),N-(m+1)+k}=&-\sum_{s=1}^{k}u_{N-(m+1),N-(m+1)+s}\;\check u_{N-(m+1)+s,N-(m+1)+k}
                                                                                                       \notag\\
                            =&-\sum_{s=1}^{k}u_{N-(m+1),N-(m+1)+s}\;\check u_{N-(m+1-s),N-(m+1-s)+(k-s)}
                                                                                                       \label{Eq_Lem_AELRP_s_RP_ss_MIPPRPLem_002_003a}
\end{alignat}
and since $s\geq1\Longrightarrow m+1-s\leq m$, we may replace $\check u_{N-(m+1-s),N-(m+1-s)+(k-s)}$ in~\eqref{Eq_Lem_AELRP_s_RP_ss_MIPPRPLem_002_003a}
by~\eqref{Eq_Lem_AELRP_s_RP_ss_MIPPRPLem_002_001b}, so that
\begin{alignat}{6}
\check u_{N-(m+1),N-(m+1)+k}=&\sum_{s=1}^{k}\frac{-1}{2^{2s}\;(2s+1)}\binom{M-2(m+1)+2s}{2s}\tau_{2k-2s}\frac{(M-2(m+1-s)+2(k-s))!}{(M-2(m+1-s))!}
                                                                                                       \notag\\
                            =&\sum_{s=1}^{k}\frac{-\tau_{2k-2s}}{2^{2s}\;(2s+1)!}\frac{(M-2(m+1)+2k)!}{(M-2(m+1))!}
                                                                                                       \notag\\
                            =&\left(\sum_{s=1}^{k}\frac{-\tau_{2k-2s}}{2^{2s}\;(2s+1)!}\right)\frac{(M-2(m+1)+2k)!}{(M-2(m+1))!}
                            = \tau_{2k}\frac{(M-2(m+1)+2k)!}{(M-2(m+1))!}
                                                                                                       \label{Eq_Lem_AELRP_s_RP_ss_MIPPRPLem_002_003b}
\end{alignat}
because, setting $\ell:=k-s$
\begin{equation}
\sum_{s=1}^{k}\frac{-\tau_{2k-2s}}{2^{2s}\;(2s+1)!}=\sum_{\ell=0}^{k-1}\frac{-\tau_{2\ell}}{2^{2k-2\ell}\;(2k-2\ell+1)!}=\tau_{2k}
                                                                                                       \label{Eq_Lem_AELRP_s_RP_ss_MIPPRPLem_002_003c}
\end{equation}
by~\eqref{Eq_Lem_AELRP_s_RPERR_ss_D_001_001c}.
\end{subequations}
This completes the proof of~\eqref{Eq_Lem_AELRP_s_RP_ss_MIPPRPLem_002_001b} by induction.\qed
\end{proof}
%
%-----------------------------------------------------------------------------------------------------------------------------------
%-----------------------------------------------------------------------------------------------------------------------------------
%
\begin{lemma}[Solution of the linear system~\eqref{Eq_Lem_AELRP_s_RP_ss_PRP_001_001d}]
\label{Lem_AELRP_s_RP_ss_MIPPRPLem_003}
The solution of the linear system~\eqref{Eq_Lem_AELRP_s_RP_ss_PRP_001_001d} is given by~\eqref{Eq_Lem_AELRP_s_RP_ss_PRP_001_001f}.
\end{lemma}
%
%-----------------------------------------------------------------------------------------------------------------------------------
%-----------------------------------------------------------------------------------------------------------------------------------
%
\begin{proof}
The unitriangular matrices $U_{(\lfloor\frac{M}{2}\rfloor,M)}$~\eqref{Eq_AELRP_s_RP_ss_MIPPRPLem_002a} and $U_{(\lfloor\frac{M-1}{2}\rfloor,M-1)}$~\eqref{Eq_AELRP_s_RP_ss_MIPPRPLem_002b}
are of the type defined in \lemrefnp{Lem_AELRP_s_RP_ss_MIPPRPLem_002}. Using the result~\eqref{Eq_Lem_AELRP_s_RP_ss_MIPPRPLem_002_001b} of \lemrefnp{Lem_AELRP_s_RP_ss_MIPPRPLem_002}
for the inverse matrices $U^{-1}_{(\lfloor\frac{M}{2}\rfloor,M)}$ and $U^{-1}_{(\lfloor\frac{M-1}{2}\rfloor,M-1)}$, the solution of the linear systems~\eqref{Eq_AELRP_s_RP_ss_MIPPRPLem_002} is
\begin{subequations}
                                                                                                       \label{Eq_Lem_AELRP_s_RP_ss_MIPPRPLem_003_001}
\begin{alignat}{6}
&c_{h_{M-2\ell}}  &=&\sum_{k=0}^{\ell}\tau_{2k} c_{f_{M-2\ell+2k}}\frac{(M-2\ell+2k)!}
                                                                       {(M-2\ell)!   }     &\quad \ell=&0,\cdots,\lfloor\frac{M}{2}\rfloor
                                                                                                       \label{Eq_Lem_AELRP_s_RP_ss_MIPPRPLem_003_001a}\\
&c_{h_{M-1-2\ell}}&=&\sum_{k=0}^{\ell}\tau_{2k} c_{f_{M-1-2\ell+2k}}\frac{(M-1-2\ell+2k)!}
                                                                         {(M-1-2\ell)!   } &\quad \ell=&0,\cdots,\lfloor\frac{M-1}{2}\rfloor
                                                                                                       \label{Eq_Lem_AELRP_s_RP_ss_MIPPRPLem_003_001b}
\end{alignat}
\end{subequations}
where the numbers $\tau_{2k}$ \tabref{Tab_Lem_AELRP_s_RPERR_ss_D_001_001} are defined by the recurrence~\eqref{Eq_Lem_AELRP_s_RPERR_ss_D_001_001c}.
Since
\begin{subequations}
                                                                                                       \label{Eq_Lem_AELRP_s_RP_ss_MIPPRPLem_003_002}
\begin{alignat}{6}
&m=M-2\ell  &\Longrightarrow&2\ell  &=&M-m&\Longrightarrow&\ell&=&\lfloor\frac{M-m}{2}\rfloor
                                                                                                       \label{Eq_Lem_AELRP_s_RP_ss_MIPPRPLem_003_002a}\\
&m=M-2\ell-1&\Longrightarrow&2\ell+1&=&M-m&\Longrightarrow&\ell&=&\lfloor\frac{M-m}{2}\rfloor
                                                                                                       \label{Eq_Lem_AELRP_s_RP_ss_MIPPRPLem_003_002b}
\end{alignat}
\end{subequations}
the 2 solutions~\eqref{Eq_Lem_AELRP_s_RP_ss_MIPPRPLem_003_001} can be grouped into~\eqref{Eq_Lem_AELRP_s_RP_ss_PRP_001_001f}, which completes the proof.\qed
\end{proof}
%
%-----------------------------------------------------------------------------------------------------------------------------------

%-----------------------------------------------------------------------------------------------------------------------------------
%
%
%
%
%
%
%
%
%
%
\section{Error of polynomial reconstruction}\label{AELRP_s_EPR}
%
%
%
%
%
%
%
%
%
%
%-----------------------------------------------------------------------------------------------------------------------------------

We consider in this paper reconstruction on a homogeneous grid (recall that~\eqref{Eq_Lem_AELRP_s_RPERR_ss_RP_001_001} hold iff $\Delta x=\const$).
The reconstruction polynomials are computed by interpolating $f(x)$ sampled on an appropriately chosen stencil~\defref{Def_AELRP_s_EPR_ss_PR_001}. We examine the relations and order-of-accuracy of
polynomial reconstruction~\defref{Def_AELRP_s_RPERR_ss_RP_002} on an arbitrary stencil $\tsc{s}_{i,M_-,M_+}$~\defref{Def_AELRP_s_EPR_ss_PR_001} defined on a homogeneous grid.
The \tsc{weno}~\cite{Liu_Osher_Chan_1994a,
                     Jiang_Shu_1996a,
                     Balsara_Shu_2000a,
                     Henrick_Aslam_Powers_2005a,
                     Borges_Carmona_Costa_Don_2008a,
                     Gerolymos_Senechal_Vallet_2009a}
schemes are based on the convex combination of polynomial reconstructions on a family of substencils. For the development of the order-of-accuracy relations,
it is necessary to develop results on the approximation-error of polynomial reconstruction for the general stencil $\tsc{s}_{i,M_-,M_+}$, around point $i$ (not necessarily contained in the stencil),
with $M_-$ neighbours on the left, and $M_+$ neighbours on the right~\defref{Def_AELRP_s_EPR_ss_PR_001}.

%-----------------------------------------------------------------------------------------------------------------------------------
%
%
%
%
%
\subsection{Polynomial reconstruction}\label{AELRP_s_EPR_ss_PR}
%
%
%
%
%
%-----------------------------------------------------------------------------------------------------------------------------------

The part concerning the approximation of $f(x)$ by a polynomial $p_f(x;\tsc{s}_{i,M_-,M_+},\Delta x)$ is found in most textbooks of numerical analysis~\cite{Henrici_1964a,
                                                                                                                                                             Phillips_2003a}.
It is only briefly included here for use in deriving the results concerning the approximation of $h(x)$ by the polynomial $p_h(x;\tsc{s}_{i,M_-,M_+},\Delta x)$
which forms a reconstruction pair with $p_f$~\defref{Def_AELRP_s_RPERR_ss_RP_001}. To obtain the relations concerning $p_h(x;\tsc{s}_{i,M_-,M_+},\Delta x)$ it is not
very practical to work with the Newton divided-differences form of $p_f$~\cite{Henrici_1964a,
                                                                               Phillips_2003a},
which are widely used in \tsc{weno} theory~\cite{Harten_Osher_1987a,
                                                 Harten_Engquist_Osher_Chakravarthy_1987a,
                                                 Liu_Osher_Chan_1994a,
                                                 Shu_1998a,
                                                 Shu_2009a}.
It is, instead, preferable to work with the standard form of $p_f$ expanded in powers of $(x-x_i)$,
whose coefficients can be readily expressed~\prpref{Prp_AELRP_s_EPR_ss_PR_001} from the coefficients of the inverse of the Vandermonde matrix~\cite{Klinger_1967a,
                                                                                                                                                  Rushanan_1989a}
corresponding to the stencil $\tsc{s}_{i,M_-,M_+}$~\defref{Def_AELRP_s_EPR_ss_PR_001}.
This representation of $p_f$ allows direct use of the formulas relating the coefficients of $p_h$ and $p_f$ \lemref{Lem_AELRP_s_RP_ss_PRP_001}.
%-----------------------------------------------------------------------------------------------------------------------------------
%
\begin{definition}[Stencil]
\label{Def_AELRP_s_EPR_ss_PR_001}
Consider a 1-D homogeneous computational mesh
\begin{subequations}
                                                                                                       \label{Eq_Def_AELRP_s_EPR_ss_PR_001_001}
\begin{equation}
x_i=x_1+(i-1)\Delta x\qquad\qquad \Delta x = {\rm const}\in\mathbb{R}_{>0}
                                                                                                       \label{Eq_Def_AELRP_s_EPR_ss_PR_001_001a}
\end{equation}
Assume
\begin{equation}
M:=M_-+M_+\geq0
                                                                                                       \label{Eq_Def_AELRP_s_EPR_ss_PR_001_001b}
\end{equation}
The set of contiguous points
\begin{equation}
\tsc{s}_{i,M_-,M_+}:=\left\{i-M_-,\cdots,i+M_+\right\}
                                                                                                       \label{Eq_Def_AELRP_s_EPR_ss_PR_001_001c}
\end{equation}
is defined as the discretization-stencil in the neighbourhood of $i$, with $M_-$ neighbours to the left and $M_+$ neighbours to the right.
The stencil $\tsc{s}_{i,M_-,M_+}$~\eqref{Eq_Def_AELRP_s_EPR_ss_PR_001_001c} contains $M+1>0$ points and has a length of $M$ intervals.
If $M_\pm\geq0$ then the stencil contains the pivot-point $i$. If $M_-M_+<0$ then the stencil does not contain the pivot-point $i$.
We will note
\begin{equation}
[\tsc{s}_{i,M_-,M_+}]:=[x_{i-M_-},x_{i+M_+}]\;\subset\;{\mathbb R}
                                                                                                       \label{Eq_Def_AELRP_s_EPR_ss_PR_001_001d}
\end{equation}
the interval defined by the extreme points of the stencil.
\end{subequations}
\qed
\end{definition}
%
%-----------------------------------------------------------------------------------------------------------------------------------
%-----------------------------------------------------------------------------------------------------------------------------------
%
\begin{remark}[Stencils and notation]
\label{Rmk_AELRP_s_EPR_ss_PR_001}
In our notation the stencil is defined by a reference (pivot) point $i$, and by the number of neighbours $M_\pm$ on each side of point $i$~\defref{Def_AELRP_s_EPR_ss_PR_001}.
The position of the pivot point $i$ in the stencil is arbitrary. This is necessary for obtaining relations for all of the \tsc{weno} stencils with reference to the same point $i$.
In the following developments, there appear quantities depending both on $M_\pm$ and on $i$ (and eventually on the values of $f$ sampled at the points of the stencil). We will systematically
note these quantities as functions of the stencil $\tsc{s}_{i,M_-,M_+}$. On the other hand, there appear quantities, which depend on $M_\pm$ but not on the pivot point $i$
(neither on the values of $f$ sampled at the points of the stencil). We will systematically
note these quantities as functions of $M_-$ and $M_+$, and not of $\tsc{s}_{i,M_-,M_+}$. This difference is important when considering order-of-accuracy relations ({\em eg}~\crlrefnp{Crl_AELRP_s_EPR_ss_AEhdf_002}).\qed
\end{remark}
%
%-----------------------------------------------------------------------------------------------------------------------------------
%-----------------------------------------------------------------------------------------------------------------------------------
%
\begin{definition}[Vandermonde matrix on $\tsc{s}_{i,M_-,M_+}$]
\label{Def_AELRP_s_EPR_ss_PR_002}
Let $M:=M_-+M_+$ and assume $M\geq 0$. The matrix ${^{M_+}_{M_-}V}\in{\mathbb R}^{(M+1)\times(M+1)}$ with elements $({^{M_+}_{M_-}V})_{ij}$
\begin{alignat}{6}
{^{M_+}_{M_-}V}:=\left[\begin{array}{cccc}(-M_-)^0&(-M_-)^1&\cdots&(-M_-)^M\\
                                          \vdots  &        &      &        \\
                                          (+M_+)^0&(+M_+)^1&\cdots&(+M_+)^M\\\end{array}\right]\qquad M:=M_-+M_+\geq0
                                                                                                       \label{Eq_Def_AELRP_s_EPR_ss_PR_002_001}
\end{alignat}
is the Vandermonde matrix~{\textup{\cite{Klinger_1967a,
                                         Rushanan_1989a}}} defined on the stencil $\tsc{s}_{i,M_-,M_+}$ \defref{Def_AELRP_s_EPR_ss_PR_001}.
Since ${^{M_+}_{M_-}V}$ is a Vandermonde matrix, its inverse ${^{M_+}_{M_-}V}^{-1}$ exists~{\textup{\cite{Macon_Spitzbart_1958a,
                                                                                                          Eisinberg_Fedele_Imbrogno_2006a}}}.
The elements of ${^{M_+}_{M_-}V}^{-1}\in{\mathbb R}^{(M+1)\times(M+1)}$ will be noted $({^{M_+}_{M_-}V^{-1}})_{ij}$.\qed
\end{definition}
%
%-----------------------------------------------------------------------------------------------------------------------------------
%-----------------------------------------------------------------------------------------------------------------------------------
%
\begin{lemma}[Inverse Vandermonde matrix on $\tsc{s}_{i,M_-,M_+}$]
\label{Lem_AELRP_s_EPR_ss_PR_001}
\begin{subequations}
                                                                                                       \label{Eq_Lem_AELRP_s_EPR_ss_PR_001_001}
Assume the conditions of \defrefnp{Def_AELRP_s_EPR_ss_PR_002}. Then the entries of the inverse of the Vandermonde matrix ${^{M_+}_{M_-}V}$ \eqref{Eq_Def_AELRP_s_EPR_ss_PR_002_001} on $\tsc{s}_{i,M_-,M_+}$
are given by
\begin{alignat}{6}
({^{M_+}_{M_-}V}^{-1})_{ij}=\sum_{n=0}^{M+1-i}(M_-)^n\;\binom{n+i-1}{n}\;({^{M}_{0}V}^{-1})_{i+n,j}\qquad\begin{array}{l}\forall i,j\in\{1,\cdots,M+1\}\\
                                                                                                                         M:=M_-+M_+\\\end{array}
                                                                                                       \label{Eq_Lem_AELRP_s_EPR_ss_PR_001_001a}
\end{alignat}
where ${^{M}_{0}V}^{-1}$ is the inverse of the Vandermonde matrix ${^{M}_{0}V}$ on $\tsc{s}_{i,0,M}=\{i,\cdots,i+M\}$ \defref{Def_AELRP_s_EPR_ss_PR_002},
whose entries are given by\footnote{\label{ff_Lem_AELRP_s_EPR_ss_PR_001_001}${\displaystyle\strlngfk{n}{k}}$ are the unsigned Stirling numbers of the first kind~\cite{Knuth_1992a,
                                                                                                                                                                       Graham_Knuth_Patashnik_1994a,
                                                                                                                                                                       Eisinberg_Fedele_Imbrogno_2006a}
satisfying
\begin{alignat}{6}
\strlngfk{n}{0}=&\delta_{n0}
                                                                                                       \notag\\
\strlngfk{n+1}{k}=&n\strlngfk{n}{k}+\strlngfk{n}{k-1}
                                                                                                       \notag\\
m\strlngfk{n}{n-m}=&\sum_{k=0}^{m-1}\binom{n-k  }
                                          {m+1-k}\strlngfk{  n}
                                                          {n-k}
                                                                                                       \notag\\
\sum_{k=1}^{n}(-1)^k(m-1)^{k-1}\strlngfk{n-1}{k-1}=&(-1)^n\;(n-1)!\;\binom{m-1}{n-1}
                                                                                                       \notag
\end{alignat}
}
\begin{alignat}{6}
({^{M}_{0}V}^{-1})_{ij}=(-1)^{i+j}\sum_{k=1}^{M+1}\frac{     1}
                                                       {(k-1)!}\binom{k-1}
                                                                     {j-1}\strlngfk{k-1}
                                                                                   {i-1}
                                                   \qquad\forall i,j\in\{1,\cdots,M+1\}
                                                                                                       \label{Eq_Lem_AELRP_s_EPR_ss_PR_001_001b}
\end{alignat}
Define
\begin{alignat}{6}
\nu_{M_-,M_+,m,k}:=\sum_{\ell=-M_-}^{M_+}({^{M_+}_{M_-}V}^{-1})_{m+1,\ell+M_-+1}\;\ell^k
                                                                                                       \label{Eq_Lem_AELRP_s_EPR_ss_PR_001_001c}
\end{alignat}
Then the following identities hold
\begin{alignat}{6}
\nu_{M_-,M_+,m,k}=
\sum_{\ell=-M_-}^{M_+}({^{M_+}_{M_-}V}^{-1})_{m+1,\ell+M_-+1}\;\ell^k=\delta_{mk}\qquad\left.\begin{array}{c} 0 \leq k \leq M\\
                                                                                                              0 \leq m \leq M\\\end{array}\right.
                                                                                                       \label{Eq_Lem_AELRP_s_EPR_ss_PR_001_001d}
\end{alignat}
\begin{alignat}{6}
\sum_{m=0}^{M}\nu_{M_-,M_+,m,k}\;\ell^m=\ell^k                                   \qquad\left.\begin{array}{l} \forall   k\in{\mathbb N}_0      \\
                                                                                                              \forall\ell\in\{-M_-,\cdots,M_+\}\\\end{array}\right.
                                                                                                       \label{Eq_Lem_AELRP_s_EPR_ss_PR_001_001e}
\end{alignat}
\end{subequations}
\end{lemma}
%
%-----------------------------------------------------------------------------------------------------------------------------------
%-----------------------------------------------------------------------------------------------------------------------------------
%
\begin{proof}\footnote{\label{ff_Lem_AELRP_s_EPR_ss_PR_001_002}
                       Proof of \eqref{Eq_Lem_AELRP_s_EPR_ss_PR_001_001d} is most easily obtained using \prprefnp{Prp_AELRP_s_EPR_ss_PR_001}, and
                       proof of \eqref{Eq_Lem_AELRP_s_EPR_ss_PR_001_001e} is most easily obtained using \prprefnp{Prp_AELRP_s_EPR_ss_AELPR_001},
                       which are proved below. Notice that \eqref{Eq_Lem_AELRP_s_EPR_ss_PR_001_001d} is not used in the proof of \prprefnp{Prp_AELRP_s_EPR_ss_PR_001}, nor is
                                                           \eqref{Eq_Lem_AELRP_s_EPR_ss_PR_001_001e} in the proof of \prprefnp{Prp_AELRP_s_EPR_ss_AELPR_001}.}
\begin{subequations}
                                                                                                       \label{Eq_Lem_AELRP_s_EPR_ss_PR_001_002}
Since ${^{M_+}_{M_-}V}$~\eqref{Eq_Def_AELRP_s_EPR_ss_PR_002_001} is an $(M+1)\times(M+1)$ Vandermonde matrix on $M+1$ distinct nodes
its inverse ${^{M_+}_{M_-}V}^{-1}$ exists~\cite{Macon_Spitzbart_1958a,
                                                Eisinberg_Fedele_Imbrogno_2006a}.
Macon and Spitzbart~\cite{Macon_Spitzbart_1957a,
                          Macon_Spitzbart_1958a}
have given explicit expressions for the inverse of the Vandermonde matrix on integer nodes. To prove \eqref{Eq_Lem_AELRP_s_EPR_ss_PR_001_001b}
we start from~\cite[Theorem 1, p. 973]{Eisinberg_Fedele_Imbrogno_2006a},
giving the inverse of the Vandermonde matrix on $n$ equidistant nodes on $[0,1]$, {\em ie} on $(n-1)x_i=(i-1)\;\forall\;i\in\{1,\cdots,n\}$, as
\begin{alignat}{6}
\left[\left(\dfrac{i-1}
                  {n-1}\right)^{j-1},\;i,j\in\{1,\cdots,n  \}\right]^{-1}_{ij}=
                                                                               (-1)^{i+j}\;(n-1)^{i-1}\sum_{k=1}^{n}\frac{     1}
                                                                                                                         {(k-1)!}\binom{k-1}
                                                                                                                                       {j-1}\strlngfk{k-1}
                                                                                                                                                     {i-1}
                                                                                                       \label{Eq_Lem_AELRP_s_EPR_ss_PR_001_002a}
\end{alignat}
which directly implies, setting $n=M+1$,
\begin{alignat}{6}
\left[\left(\dfrac{i-1}
                  {M  }\right)^{j-1},\;i,j\in\{1,\cdots,M+1\}\right]^{-1}_{ij}=
                                                                               (-1)^{i+j}\;M^{i-1}\sum_{k=1}^{M+1}\frac{     1}
                                                                                                                       {(k-1)!}\binom{k-1}
                                                                                                                                   {j-1}\strlngfk{k-1}
                                                                                                                                                 {i-1}
                                                                                                       \label{Eq_Lem_AELRP_s_EPR_ss_PR_001_002b}
\end{alignat}
Obviously, $M^{i-1}$ and $M^{j-1}$ in \eqref{Eq_Lem_AELRP_s_EPR_ss_PR_001_002b} are scaling factors (for $M+1$ equidistant nodes on $[0,1]$ we have $M\;\Delta x=1$).
This is clearly seen by writing the Vandermonde matrix on $\tsc{s}_{i,0,M}$ \eqref{Eq_Def_AELRP_s_EPR_ss_PR_002_001} as
\begin{alignat}{6}
{^{M}_{0}V}:=&\left[(i-1)^{j-1},\;i,j\in\{1,\cdots,M+1\}\right]
                                                                                                       \notag\\
=&
\left[\left(\dfrac{i-1}
                  {M  }\right)^{\ell-1},\;i,\ell\in\{1,\cdots,M+1\}\right]\;
\left[M^{\ell-1}\;\delta_{\ell j},\;\ell,j\in\{1,\cdots,M+1\}\right]
                                                                                                       \label{Eq_Lem_AELRP_s_EPR_ss_PR_001_002c}
\end{alignat}
and since $\left[M^{\ell-1}\;\delta_{\ell j},\;\ell,j\in\{1,\cdots,M+1\}\right]$ is a diagonal matrix
\begin{alignat}{6}
{^{M}_{0}V}^{-1}=
\left[\dfrac{\delta_{i\ell}}{M^{i-1}},\;i,\ell\in\{1,\cdots,M+1\}\right]\;
\left[\left(\dfrac{\ell-1}
                  {M     }\right)^{j-1},\;\ell,j\in\{1,\cdots,M+1\}\right]^{-1}
                                                                                                       \label{Eq_Lem_AELRP_s_EPR_ss_PR_001_002d}
\end{alignat}
which, by \eqref{Eq_Lem_AELRP_s_EPR_ss_PR_001_002b}, proves \eqref{Eq_Lem_AELRP_s_EPR_ss_PR_001_001b}.

To obtain the final expression \eqref{Eq_Lem_AELRP_s_EPR_ss_PR_001_001a}, we observe that, for $M:=M_-+M_+$,
the stencils $\tsc{s}_{i,M_-,M_+}$ (corresponding Vandermonde matrix ${^{M+}_{M_-}V}$; \defrefnp{Def_AELRP_s_EPR_ss_PR_002})
and $\tsc{s}_{i-M_-,0,M}$ (corresponding Vandermonde matrix ${^{M}_{0}V}$; \defrefnp{Def_AELRP_s_EPR_ss_PR_002})
correspond by \defrefnp{Def_AELRP_s_EPR_ss_PR_001} to the same set of points
$\{i-M_-,\cdots,i+M_+\}$. Therefore, $\forall\;f\in C[x_{i-M_-},x_{i+M_+}]$, by the uniqueness of the Lagrange interpolating polynomial~\cite{Henrici_1964a}, we have (using the notation of \prprefnp{Prp_AELRP_s_EPR_ss_PR_001})
\begin{alignat}{6}
p_f(x;\tsc{s}_{i,M_-,M_+},\Delta x)=p_f(x;\tsc{s}_{i-M_-,0,M},\Delta x)\quad\begin{array}{c}\forall\;x\in{\mathbb R}\\
                                                                                            \forall\;f\in C[x_{i-M_-},x_{i+M_+}]\\\end{array}
                                                                                                       \label{Eq_Lem_AELRP_s_EPR_ss_PR_001_002e}
\end{alignat}
the only difference being in the choice of the pivot point ($x_i$ for $\tsc{s}_{i,M_-,M_+}$ and $x_{i-M_-}=x_i-M_-\;\Delta x$ for $\tsc{s}_{i-M_-,0,M}$) used
for the representation \eqref{Eq_Prp_AELRP_s_EPR_ss_PR_001_001b} of the interpolating polynomial of $f(x)$ on the nodes $\{i-M_-,\cdots,i+M_+\}$.
By \eqref{Eq_Prp_AELRP_s_EPR_ss_PR_001_001b}, \eqref{Eq_Lem_AELRP_s_EPR_ss_PR_001_002e} reads
\begin{alignat}{6}
\sum_{m=0}^{M} c_{f,\tsc{s}_{i,M_-,M_+},m}\left(\frac{x-x_i}{\Delta x}\right)^m=&\sum_{s=0}^{M} c_{f,\tsc{s}_{i-M_-,0,M},s}\left(\frac{x-x_{i-M_-}}{\Delta x}\right)^s
                                                                                                       \notag\\
                                                                               =&\sum_{s=0}^{M} c_{f,\tsc{s}_{i-M_-,0,M},s}\left(\frac{x-x_i}{\Delta x}+M_-\right)^s
                                                                               = \sum_{s=0}^{M}\sum_{n=0}^{s} c_{f,\tsc{s}_{i-M_-,0,M},s}\binom{s}{n}(M_-)^n\left(\frac{x-x_i}{\Delta x}\right)^{s-n}
                                                                                                       \notag\\
                                                            \stackrel{m:=s-n}{=}&\sum_{m=0}^{M}\left(\sum_{n=0}^{M-m} c_{f,\tsc{s}_{i-M_-,0,M},m+n}\binom{m+n}{n}(M_-)^n\right)\left(\frac{x-x_i}{\Delta x}\right)^m
                                                                       \quad\begin{array}{c}\forall\;x\in{\mathbb R}\\
                                                                                            \forall\;f\in C[x_{i-M_-},x_{i+M_+}]\\\end{array}
                                                                                                       \label{Eq_Lem_AELRP_s_EPR_ss_PR_001_002f}
\end{alignat}
implying 
\begin{alignat}{6}
c_{f,\tsc{s}_{i,M_-,M_+},m}=\sum_{n=0}^{M-m} c_{f,\tsc{s}_{i-M_-,0,M},m+n}\binom{m+n}{n}(M_-)^n\qquad\forall\;m\in\{0,\cdots,M\}
                                                                                                       \label{Eq_Lem_AELRP_s_EPR_ss_PR_001_002g}
\end{alignat}
which by \eqref{Eq_Prp_AELRP_s_EPR_ss_PR_001_005a} gives, $\forall\;f\in C[x_{i-M_-},x_{i+M_+}]$
\begin{alignat}{6}
                          \sum_{\ell=-M_-}^{M_+}({^{M_+}_{M_-}V}^{-1})_{m+1,\ell+M_-+1}\;f_{i+\ell}
                        =&\sum_{n=0}^{M-m}\left(\sum_{s=0}^{M}({^{M}_{0}V}^{-1})_{m+n+1,s+1}\;f_{i-M_-+s}\right)\;\binom{m+n}{n}\;(M_-)^n
                                                                                                       \notag\\
                        =&\sum_{s=0}^{M}\sum_{n=0}^{M-m}({^{M}_{0}V}^{-1})_{m+n+1,s+1}\;f_{i-M_-+s}\;\binom{m+n}{n}\;(M_-)^n
                                                                                                       \notag\\
\stackrel{\ell:=s-M_-}{=}&\sum_{\ell=-M_-}^{M_+}\left(\sum_{n=0}^{M-m}\binom{m+n}{n}\;(M_-)^n\;({^{M}_{0}V}^{-1})_{m+n+1,\ell+M_-+1}\right)\;f_{i+\ell}
                                                                                                       \label{Eq_Lem_AELRP_s_EPR_ss_PR_001_002h}
\end{alignat}
and since $f_{i+\ell}$ ($\ell\in\{-M_-,\cdots,M_+\}$) are linearly independent we have
\begin{alignat}{6}
({^{M_+}_{M_-}V}^{-1})_{m+1,\ell+M_-+1}=\sum_{n=0}^{M-m}\binom{m+n}{n}\;(M_-)^n\;({^{M}_{0}V}^{-1})_{m+n+1,\ell+M_-+1}
                                                                                                                  \quad\begin{array}{l}\forall\;m   \in\{0,\cdots,M\}     \\
                                                                                                                                       \forall\;\ell\in\{-M_-,\cdots,M_+\}\\\end{array}
                                                                                                       \label{Eq_Lem_AELRP_s_EPR_ss_PR_001_002i}
\end{alignat}
which proves \eqref{Eq_Lem_AELRP_s_EPR_ss_PR_001_001a}.

To prove the identities containing $\nu_{M_-,M_+,m,k}$ \eqref{Eq_Lem_AELRP_s_EPR_ss_PR_001_001c}, notice that
the elements of ${^{M_+}_{M_-}V}$~\eqref{Eq_Def_AELRP_s_EPR_ss_PR_002_001} read
\begin{alignat}{6}
({^{M_+}_{M_-}V})_{ij}=(i-1-M_-)^{j-1}\qquad\left.\begin{array}{c} 1 \leq i \leq M+1\\
                                                                   1 \leq j \leq M+1\\\end{array}\right.
                                                                                                       \label{Eq_Lem_AELRP_s_EPR_ss_PR_001_002j}
\end{alignat}
Explicit expression of the elements of the product $({^{M_+}_{M_-}V}^{-1})\cdot({^{M_+}_{M_-}V})=I_{M+1}$ (where $I_{M+1}\in{\mathbb R}^{(M+1)\times(M+1)}$ is the identity matrix)
yields
\begin{alignat}{6}
\delta_{m+1,k+1}=
\left(({^{M_+}_{M_-}V}^{-1})\cdot({^{M_+}_{M_-}V})\right)_{m+1,k+1}= \nu_{M_-,M_+,m,k}
                                                                                 \qquad\left.\begin{array}{c} 0 \leq k \leq M\\
                                                                                                              0 \leq m \leq M\\\end{array}\right.
                                                                                                       \label{Eq_Lem_AELRP_s_EPR_ss_PR_001_002k}
\end{alignat}
and as a consequence~\eqref{Eq_Lem_AELRP_s_EPR_ss_PR_001_001d}.
To prove \eqref{Eq_Lem_AELRP_s_EPR_ss_PR_001_001e}, consider the error \eqref{Eq_Prp_AELRP_s_EPR_ss_AELPR_001_001b} of the polynomial interpolation $p_f(x_i+\xi\Delta x;\tsc{s}_{i,M_-,M_+},\Delta x)$
on the stencil $\tsc{s}_{i,M_-,M_+}$ \prpref{Prp_AELRP_s_EPR_ss_AELPR_001}. By construction, we have
\begin{alignat}{6}
p_f(x_i+\ell\Delta x;\tsc{s}_{i,M_-,M_+},\Delta x)=f_{i+\ell}\stackrel{\eqref{Eq_Prp_AELRP_s_EPR_ss_AELPR_001_001b}}{\Longrightarrow}
E_f(x_i+\ell\Delta x;\tsc{s}_{i,M_-,M_+},\Delta x)=0 \qquad\forall\ell\in\{-M_-,\cdots,M_+\}
                                                                                                       \label{Eq_Lem_AELRP_s_EPR_ss_PR_001_002l}
\end{alignat}
which, using \eqref{Eq_Prp_AELRP_s_EPR_ss_AELPR_001_001e} and \eqref{Eq_Prp_AELRP_s_EPR_ss_AELPR_001_001g} in \eqref{Eq_Prp_AELRP_s_EPR_ss_AELPR_001_001b}, proves \eqref{Eq_Lem_AELRP_s_EPR_ss_PR_001_001e}.\qed
\end{subequations}
\end{proof}
%
%-----------------------------------------------------------------------------------------------------------------------------------

%-----------------------------------------------------------------------------------------------------------------------------------
%
\begin{proposition}[Lagrange polynomial reconstruction on $\tsc{s}_{i,M_-,M_+}$]
\label{Prp_AELRP_s_EPR_ss_PR_001}
Let
\begin{subequations}
                                                                                                       \label{Eq_Prp_AELRP_s_EPR_ss_PR_001_001}
\begin{alignat}{6}
p_h(x;\tsc{s}_{i,M_-,M_+},\Delta x)&:=&\sum_{m=0}^{M} c_{h,\tsc{s}_{i,M_-,M_+},m}\left(\frac{x-x_i}{\Delta x}\right)^m
                                                                                                       \label{Eq_Prp_AELRP_s_EPR_ss_PR_001_001a}\\
p_f(x;\tsc{s}_{i,M_-,M_+},\Delta x)&:=&\sum_{m=0}^{M} c_{f,\tsc{s}_{i,M_-,M_+},m}\left(\frac{x-x_i}{\Delta x}\right)^m
                                                                                                       \label{Eq_Prp_AELRP_s_EPR_ss_PR_001_001b}
\end{alignat}
be 2 polynomials of degree
\begin{equation}
M:=M_-+M_+
                                                                                                       \label{Eq_Prp_AELRP_s_EPR_ss_PR_001_001c}
\end{equation}
constituting a polynomial \lemref{Lem_AELRP_s_RP_ss_PRP_001} reconstruction pair \defref{Def_AELRP_s_RPERR_ss_RP_001} $p_h=R_{(1;\Delta x)}(p_f)$.
Assume that the polynomial $p_f(x;\tsc{s}_{i,M_-,M_+},\Delta x)$ is obtained by interpolation of the values of $f(x)$ on the
points of the stencil $\tsc{s}_{i,M_-,M_+}$ \defref{Def_AELRP_s_EPR_ss_PR_001}.
Then
\begin{alignat}{6}
p_h(x_i+\xi\Delta x;\tsc{s}_{i,M_-,M_+},\Delta x)&=&\sum_{\ell=-M_-}^{M_+}\alpha_{h,M_-,M_+,\ell}(\xi)f_{i+\ell}
                                                                                                       \label{Eq_Prp_AELRP_s_EPR_ss_PR_001_001d}\\
p_f(x_i+\xi\Delta x;\tsc{s}_{i,M_-,M_+},\Delta x)&=&\sum_{\ell=-M_-}^{M_+}\alpha_{f,M_-,M_+,\ell}(\xi)f_{i+\ell}
                                                                                                       \label{Eq_Prp_AELRP_s_EPR_ss_PR_001_001e}
\end{alignat}
where $\alpha_{h,M_-,M_+,\ell}(\xi)$ and $\alpha_{f,M_-,M_+,\ell}(\xi)$ are polynomials of degree $M$ in
\begin{alignat}{6}
\xi:=\frac{x-x_i}{\Delta x}
                                                                                                       \label{Eq_Prp_AELRP_s_EPR_ss_PR_001_001f}
\end{alignat}
with coefficients depending only on the 3 indices ($M_-,M_+,\ell$)
\begin{alignat}{6}
\alpha_{h,M_-,M_+,\ell}(\xi):=&\sum_{m=0}^{M}\left(\sum_{k=0}^{\lfloor\frac{M-m}{2}\rfloor}\frac{\tau_{2k}(m+2k)!}
                                                                                                {m!              }({^{M_+}_{M_-}V}^{-1})_{m+2k+1,\ell+M_-+1}\right)\xi^m
                                                                                                       \label{Eq_Prp_AELRP_s_EPR_ss_PR_001_001g}\\
\alpha_{f,M_-,M_+,\ell}(\xi):=&\sum_{m=0}^{M}({^{M_+}_{M_-}V}^{-1})_{m+1,\ell+M_-+1}\;\xi^m
                                                                                                       \label{Eq_Prp_AELRP_s_EPR_ss_PR_001_001h}
\end{alignat}
where $({^{M_+}_{M_-}V}^{-1})_{ij}$ are the elements of the inverse Vandermonde matrix on $\tsc{s}_{i,M_-,M_+}$ \lemref{Lem_AELRP_s_EPR_ss_PR_001},
and the numbers $\tau_{2k}$ \tabref{Tab_Lem_AELRP_s_RPERR_ss_D_001_001} are defined by~\eqref{Eq_Thm_AELRP_s_RPERR_ss_GFtaunRPexp_001_001c} and satisfy the recurrence~\eqref{Eq_Lem_AELRP_s_RPERR_ss_D_001_001c}.
\end{subequations}
\end{proposition}
%
%-----------------------------------------------------------------------------------------------------------------------------------
%-----------------------------------------------------------------------------------------------------------------------------------
%
\begin{proof}
Define
\begin{subequations}
                                                                                                       \label{Eq_Prp_AELRP_s_EPR_ss_PR_001_002}
\begin{alignat}{6}
x_{i+\ell}&:=&x_i+\ell\Delta x&\qquad -M_-\leq \ell\leq M_+
                                                                                                       \label{Eq_Prp_AELRP_s_EPR_ss_PR_001_002a}\\
f_{i+\ell}&:=&f(x_{i+\ell})   &\qquad -M_-\leq \ell\leq M_+
                                                                                                       \label{Eq_Prp_AELRP_s_EPR_ss_PR_001_002b}
\end{alignat}
\end{subequations}
The $M+1$ coefficients $c_{f,(\tsc{s}_{i,M_-,M_+}),m}\;(m=0,\cdots,M)$ are computed by equating the polynomial
$p_f(x_{i+\ell};\tsc{s}_{i,M_-,M_+},\Delta x)$~\eqref{Eq_Prp_AELRP_s_EPR_ss_PR_001_001b} to known values $f_{i+\ell}$
\begin{equation}
\begin{array}{lcl}
f_{i-M_-}&   =  &p_f(x_{i-M-};\tsc{s}_{i,M_-,M_+},\Delta x)\\
         &\vdots&                                          \\
f_{i+M_+}&   =  &p_f(x_{i-M+};\tsc{s}_{i,M_-,M_+},\Delta x)\\\end{array}
                                                                                                       \label{Eq_Prp_AELRP_s_EPR_ss_PR_001_003}
\end{equation}
Expanding~\eqref{Eq_Prp_AELRP_s_EPR_ss_PR_001_003} 
results in an $(M+1)\times(M+1)$ Vandermonde~\defref{Def_AELRP_s_EPR_ss_PR_002} linear system
\begin{alignat}{6}
\underbrace{
\left[\begin{array}{cccc}(-M_-)^0&(-M_-)^1&\cdots&(-M_-)^M\\
                         \vdots  &        &      &        \\
                         (+M_+)^0&(+M_+)^1&\cdots&(+M_+)^M\\\end{array}\right]}_{\displaystyle {^{M_+}_{M_-}V}}\left[\begin{array}{c}c_{f,\tsc{s}_{i,M_-,M_+},0}\\
                                                                                                                                     \vdots                     \\
                                                                                                                                     c_{f,\tsc{s}_{i,M_-,M_+},M}\\\end{array}\right]=\left[\begin{array}{c}f_{i-M_-}\\
                                                                                                                                                                                                           \vdots   \\
                                                                                                                                                                                                           f_{i+M_+}\\\end{array}\right]
                                                                                                       \label{Eq_Prp_AELRP_s_EPR_ss_PR_001_004}
\end{alignat}
Hence~\defref{Def_AELRP_s_EPR_ss_PR_002}
\begin{subequations}
                                                                                                       \label{Eq_Prp_AELRP_s_EPR_ss_PR_001_005}
\begin{alignat}{6}
c_{f,\tsc{s}_{i,M_-,M_+},m}=&\sum_{\ell=-M_-}^{M_+}({^{M_+}_{M_-}V}^{-1})_{m+1,\ell+M_-+1}\;f_{i+\ell}      &\qquad\forall m\in\{0,\cdots,M\}
                                                                                                       \label{Eq_Prp_AELRP_s_EPR_ss_PR_001_005a}\\
c_{h,\tsc{s}_{i,M_-,M_+},m}=&\frac{1}{m!}\sum_{k=0}^{\lfloor\frac{M-m}{2}\rfloor}\tau_{2k}\;c_{f,\tsc{s}_{i,M_-,M_+},m+2k}\;(m+2k)!  &\qquad\forall m\in\{0,\cdots,M\}
                                                                                                       \label{Eq_Prp_AELRP_s_EPR_ss_PR_001_005b}
\end{alignat}
\end{subequations}
where we used the deconvolution formula~\eqref{Eq_Lem_AELRP_s_RP_ss_PRP_001_001f} for computing $c_{h,\tsc{s}_{i,M_-,M_+},m}$.
Injecting~\eqref{Eq_Prp_AELRP_s_EPR_ss_PR_001_005a} into~\eqref{Eq_Prp_AELRP_s_EPR_ss_PR_001_001b} we have
\begin{subequations}
                                                                                                       \label{Eq_Prp_AELRP_s_EPR_ss_PR_001_006}
\begin{alignat}{6}
p_f(x_i+\xi\Delta x;\tsc{s}_{i,M_-,M_+},\Delta x)=\sum_{m=0}^{M}\left(\sum_{\ell=-M_-}^{M_+}({^{M_+}_{M_-}V}^{-1})_{m+1,\ell+M_-+1}\;f_{i+\ell}\right)\xi^m
                                                 =\sum_{\ell=-M_-}^{M_+}\left(\sum_{m=0}^{M}({^{M_+}_{M_-}V}^{-1})_{m+1,\ell+M_-+1}\;\xi^m\right)f_{i+\ell}
                                                                                                       \label{Eq_Prp_AELRP_s_EPR_ss_PR_001_006a}
\end{alignat}
proving~\eqref{Eq_Prp_AELRP_s_EPR_ss_PR_001_001e} and~\eqref{Eq_Prp_AELRP_s_EPR_ss_PR_001_001h}.
Injecting~\eqref{Eq_Prp_AELRP_s_EPR_ss_PR_001_005b} into~\eqref{Eq_Prp_AELRP_s_EPR_ss_PR_001_001a} we have
\begin{alignat}{6}
 p_h(x_i+\xi\Delta x;\tsc{s}_{i,M_-,M_+},\Delta x)=&\sum_{m=0}^{M}\left(\sum_{k=0}^{\lfloor\frac{M-m}{2}\rfloor}\frac{\tau_{2k}(m+2k)!}{m!}c_{f,\tsc{s}_{i,M_-,M_+},m+2k}\right)\;\xi^m
                                                                                                       \notag\\
                                                  =&\sum_{m=0}^{M}\left(\sum_{k=0}^{\lfloor\frac{M-m}{2}\rfloor}\frac{\tau_{2k}(m+2k)!}{m!}
                                                                  \left(\sum_{\ell=-M_-}^{M_+}({^{M_+}_{M_-}V}^{-1})_{m+2k+1,\ell+M_-+1}f_{i+\ell}\right)\right)\;\xi^m
                                                                                                       \notag\\
                                                  =&\sum_{\ell=-M_-}^{M_+}\left(\sum_{m=0}^{M}\left(\sum_{k=0}^{\lfloor\frac{M-m}{2}\rfloor}
                                                                                \frac{\tau_{2k}(m+2k)!}{m!}({^{M_+}_{M_-}V}^{-1})_{m+2k+1,\ell+M_-+1}\right)\;\xi^m\right) f_{i+\ell}
                                                                                                       \label{Eq_Prp_AELRP_s_EPR_ss_PR_001_006b}
\end{alignat}
\end{subequations}
proving~\eqref{Eq_Prp_AELRP_s_EPR_ss_PR_001_001d} and~\eqref{Eq_Prp_AELRP_s_EPR_ss_PR_001_001g}.\qed
\end{proof}
%
%-----------------------------------------------------------------------------------------------------------------------------------

%-----------------------------------------------------------------------------------------------------------------------------------
%
%
%
%
%
\subsection{Approximation error of Lagrange polynomial reconstruction}\label{AELRP_s_EPR_ss_AELPR}
%
%
%
%
%
%-----------------------------------------------------------------------------------------------------------------------------------

Of course the accuracy relations for the approximation of $f(x)$ are well-known~\cite{Henrici_1964a},
but this section (\S\ref{AELRP_s_EPR_ss_AELPR}) is concerned with the accuracy of the approximation of $h(x)$,
using Lagrange polynomial reconstruction based on the knowledge of the values of $f(x)$ on an arbitrary stencil defined on a homogeneous grid (\S\ref{AELRP_s_EPR_ss_PR}).
%-----------------------------------------------------------------------------------------------------------------------------------
%
\begin{proposition}[Error of Lagrange polynomial reconstruction on $\tsc{s}_{i,M_-,M_+}$]
\label{Prp_AELRP_s_EPR_ss_AELPR_001}
Let $p_f(x;\tsc{s}_{i,M_-,M_+},\Delta x)$ and $p_h(x;\tsc{s}_{i,M_-,M_+},\Delta x)$ be a polynomial \lemref{Lem_AELRP_s_RP_ss_PRP_001} reconstruction pair \defref{Def_AELRP_s_RPERR_ss_RP_001} $p_h=R_{(1;\Delta x)}(p_f)$,
satisfying the conditions of \prprefnp{Prp_AELRP_s_EPR_ss_PR_001}.
Then, $p_f(x;\tsc{s}_{i,M_-,M_+},\Delta x)$ approximates $f(x)$ to $O(\Delta x^{M+1})$, and $p_h(x;\tsc{s}_{i,M_-,M_+},\Delta x)$ approximates $h(x)$ to $O(\Delta x^{M+1})$
\begin{subequations}
                                                                                                       \label{Eq_Prp_AELRP_s_EPR_ss_AELPR_001_001}
\begin{alignat}{6}
p_h(x;\tsc{s}_{i,M_-,M_+},\Delta x)&=&h(x)&+&E_h(x;\tsc{s}_{i,M_-,M_+},\Delta x)&=&h(x)&+&O(\Delta x^{M+1})
                                                                                                       \label{Eq_Prp_AELRP_s_EPR_ss_AELPR_001_001a}\\
p_f(x;\tsc{s}_{i,M_-,M_+},\Delta x)&=&f(x)&+&E_f(x;\tsc{s}_{i,M_-,M_+},\Delta x)&=&f(x)&+&O(\Delta x^{M+1})
                                                                                                       \label{Eq_Prp_AELRP_s_EPR_ss_AELPR_001_001b}
\end{alignat}
where the approximation errors constitute a reconstruction pair $E_h=R_{(1;\Delta x)}(E_f)$~\defref{Def_AELRP_s_RPERR_ss_RP_001} and, $\forall N_\tsc{tj}\geq M+1$, are given by
\textup{(}assuming $f$ and $h$ are of class $C^{N_\tsc{tj}+1}$\textup{)}
\begin{alignat}{6}
E_h(x_i+\xi\Delta x;\tsc{s}_{i,M_-,M_+},\Delta x)=&\sum_{s=M+1}^{N_\tsc{tj}}\mu_{h,M_-,M_+,s}(\xi)\Delta x^s f_i^{(s)}+O(\Delta x^{N_\tsc{tj}+1})
                                                                                                       \label{Eq_Prp_AELRP_s_EPR_ss_AELPR_001_001c}\\
                                                 =&\sum_{s=M+1}^{N_\tsc{tj}}\left(\sum_{\ell=0}^{\lfloor\frac{s-M-1}{2}\rfloor}\frac{\mu_{h,M_-,M_+,s-2\ell}(\xi)}
                                                                                                                                    {2^{2\ell}\;(2\ell+1)!           }\right)\Delta x^s h_i^{(s)}+O(\Delta x^{N_\tsc{tj}+1})
                                                                                                       \label{Eq_Prp_AELRP_s_EPR_ss_AELPR_001_001d}\\
E_f(x_i+\xi\Delta x;\tsc{s}_{i,M_-,M_+},\Delta x)=&\sum_{s=M+1}^{N_\tsc{tj}}\mu_{f,M_-,M_+,s}(\xi)\Delta x^s f_i^{(s)}+O(\Delta x^{N_\tsc{tj}+1})
                                                                                                       \label{Eq_Prp_AELRP_s_EPR_ss_AELPR_001_001e}
\end{alignat}
where $\mu_{h,M_-,M_+,s}(\xi)$ and $\mu_{f,M_-,M_+,s}(\xi)$ are polynomials of degree $s$ in $\xi$~\eqref{Eq_Prp_AELRP_s_EPR_ss_PR_001_001f}
\begin{alignat}{6}
\mu_{h,M_-,M_+,s}(\xi):=&\sum_{k=0}^{\lfloor\frac{s}{2}\rfloor}\frac{-\tau_{2k}}
                                                                        {(s-2k)!   }\xi^{s-2k}
                           + \sum_{m=0}^{M}\left(\sum_{k=0}^{\lfloor\frac{M-m}{2}\rfloor}\tau_{2k}\nu_{M_-,M_+,m+2k,s}\frac{(m+2k)!}{s!\;m!}\right)\xi^m
                                                                                                       \label{Eq_Prp_AELRP_s_EPR_ss_AELPR_001_001f}\\
\mu_{f,M_-,M_+,s}(\xi):=&\frac{ 1}
                                  {s!}\left(-\xi^s+\sum_{m=0}^{M}\nu_{M_-,M_+,m,s}\xi^m\right)
                                                                                                       \label{Eq_Prp_AELRP_s_EPR_ss_AELPR_001_001g}
\end{alignat}
where $\nu_{M_-,M_+,m,s}$ are defined by \eqref{Eq_Lem_AELRP_s_EPR_ss_PR_001_001c},
and the numbers $\tau_{2k}$ \tabref{Tab_Lem_AELRP_s_RPERR_ss_D_001_001} are defined by~\eqref{Eq_Thm_AELRP_s_RPERR_ss_GFtaunRPexp_001_001c} and satisfy the recurrence~\eqref{Eq_Lem_AELRP_s_RPERR_ss_D_001_001c}.
\end{subequations}
\end{proposition}
%
%-----------------------------------------------------------------------------------------------------------------------------------
%-----------------------------------------------------------------------------------------------------------------------------------
%
\begin{proof}
To prove~\eqref{Eq_Prp_AELRP_s_EPR_ss_AELPR_001_001b} we start by Taylor-expanding $f_{i+\ell}$ in~\eqref{Eq_Prp_AELRP_s_EPR_ss_PR_001_005a}, and using~\eqref{Eq_Lem_AELRP_s_EPR_ss_PR_001_001d}
\begin{alignat}{6}
c_{f,\tsc{s}_{i,M_-,M_+},m}=&\sum_{\ell=-M_-}^{M_+}({^{M_+}_{M_-}V}^{-1})_{m+1,\ell+M_-+1}\left(\sum_{s=0}^{N_\tsc{tj}}\frac{\ell^s\Delta x^s f^{(s)}_i}{s!}+O(\Delta x^{N_\tsc{tj}+1})\right)
                                                                                                       \notag\\
                           =&\sum_{s=0}^{N_\tsc{tj}}\left(\sum_{\ell=-M_-}^{M_+}\left(({^{M_+}_{M_-}V}^{-1})_{m+1,\ell+M_-+1}\;\ell^s\right)\right)\frac{\Delta x^s f^{(s)}_i}{s!}+O(\Delta x^{N_\tsc{tj}+1})
                                                                                                       \notag\\
                           =&\sum_{s=0}^{N_\tsc{tj}}\nu_{M_-,M_+,m,s}\frac{\Delta x^s f^{(s)}_i}{s!}+O(\Delta x^{N_\tsc{tj}+1})
                           = \sum_{s=0}^{M}\delta_{ms}\frac{\Delta x^s f^{(s)}_i}{s!}+\sum_{s=M+1}^{N_\tsc{tj}}\nu_{M_-,M_+,m,s}\frac{\Delta x^s f^{(s)}_i}{s!}+O(\Delta x^{N_\tsc{tj}+1})
                                                                                                       \notag\\
                           =&\frac{\Delta x^m f^{(m)}_i}{m!}+\sum_{s=M+1}^{N_\tsc{tj}}\nu_{M_-,M_+,m,s}\frac{\Delta x^s f^{(s)}_i}{s!}+O(\Delta x^{N_\tsc{tj}+1})
                                                                                                       \label{Eq_Prp_AELRP_s_EPR_ss_AELPR_001_002}
\end{alignat}
Injecting~\eqref{Eq_Prp_AELRP_s_EPR_ss_AELPR_001_002} into~\eqref{Eq_Prp_AELRP_s_EPR_ss_PR_001_001b}, and replacing $f(x_i+\xi\Delta x)$ by its Taylor-polynomial, we have
\begin{alignat}{6}
E_f(x_i+\xi\Delta x;\tsc{s}_{i,M_-,M_+},\Delta x)=&p_f(x_i+\xi\Delta x;\tsc{s}_{i,M_-,M_+},\Delta x)-f(x_i+\xi\Delta x)
                                                                                                       \notag\\
=&\sum_{m=0}^{M}\left(\frac{\Delta x^m f^{(m)}_i}{m!}+\sum_{s=M+1}^{N_\tsc{tj}}\nu_{M_-,M_+,m,s}\frac{\Delta x^s f^{(s)}_i}{s!}+O(\Delta x^{N_\tsc{tj}+1})\right)\xi^m-f(x_i+\xi\Delta x)
                                                                                                       \notag\\
=&\left(\sum_{m=0}^{M}\frac{\Delta x^m f^{(m)}_i}{m!}\xi^m-f(x_i+\xi\Delta x)\right)+\sum_{m=0}^{M}\sum_{s=M+1}^{N_\tsc{tj}}\nu_{M_-,M_+,m,s}\frac{\Delta x^s f^{(s)}_i}{s!}\xi^m+O(\Delta x^{N_\tsc{tj}+1})
                                                                                                       \notag\\
=&\sum_{m=M+1}^{N_\tsc{tj}}\frac{-\Delta x^m f^{(m)}_i}{m!}\xi^m
+\sum_{m=0}^{M}\sum_{s=M+1}^{N_\tsc{tj}}\nu_{M_-,M_+,m,s}\frac{\Delta x^s f^{(s)}_i}{s!}\xi^m+O(\Delta x^{N_\tsc{tj}+1})
                                                                                                       \notag\\
=&\sum_{s=M+1}^{N_\tsc{tj}}\frac{-\Delta x^s f^{(s)}_i}{s!}\xi^s
+\sum_{s=M+1}^{N_\tsc{tj}}\left(\sum_{m=0}^{M}\nu_{M_-,M_+,m,s}\;\xi^m\right)\frac{\Delta x^s f^{(s)}_i}{s!}+O(\Delta x^{N_\tsc{tj}+1})
                                                                                                       \label{Eq_Prp_AELRP_s_EPR_ss_AELPR_001_003}
\end{alignat}
proving~\eqref{Eq_Prp_AELRP_s_EPR_ss_AELPR_001_001b}, \eqref{Eq_Prp_AELRP_s_EPR_ss_AELPR_001_001e} and~\eqref{Eq_Prp_AELRP_s_EPR_ss_AELPR_001_001g}.

To prove~\eqref{Eq_Prp_AELRP_s_EPR_ss_AELPR_001_001a} we use the expression~\eqref{Eq_Prp_AELRP_s_EPR_ss_AELPR_001_002} for $c_{f,\tsc{s}_{i,M_-,M_+},m}$
in~\eqref{Eq_Prp_AELRP_s_EPR_ss_PR_001_005b} to obtain
\begin{alignat}{6}
c_{h,\tsc{s}_{i,M_-,M_+},m}=&
  \frac{1}{m!}\sum_{k=0}^{\lfloor\frac{M-m}{2}\rfloor}\tau_{2k}\;(m+2k)!\left(\frac{\Delta x^{m+2k} f^{(m+2k)}_i}{(m+2k)!}
 +\sum_{s=M+1}^{N_\tsc{tj}}\nu_{M_-,M_+,m+2k,s}\frac{\Delta x^s f^{(s)}_i}{s!})\right)
+O(\Delta x^{N_\tsc{tj}+1})
                                                                                                       \notag\\
=&\sum_{k=0}^{\lfloor\frac{M-m}{2}\rfloor}\tau_{2k}\;\frac{\Delta x^{m+2k} f^{(m+2k)}_i}{m!}
+\sum_{k=0}^{\lfloor\frac{M-m}{2}\rfloor}\frac{\tau_{2k}(m+2k)!}{m!}\left(\sum_{s=M+1}^{N_\tsc{tj}}\nu_{M_-,M_+,m+2k,s}\frac{\Delta x^s f^{(s)}_i}{s!}\right)
+O(\Delta x^{N_\tsc{tj}+1})
                                                                                                       \label{Eq_Prp_AELRP_s_EPR_ss_AELPR_001_004}
\end{alignat}
Injecting~\eqref{Eq_Prp_AELRP_s_EPR_ss_AELPR_001_004} into~\eqref{Eq_Prp_AELRP_s_EPR_ss_PR_001_001a},
and replacing $h(x_i+\xi\Delta x)$ by its Taylor-polynomial~\eqref{Eq_Crl_AELRP_s_RPERR_ss_D_001_001},
we have\footnote{\label{ff_Prp_AELRP_s_EPR_ss_AELPR_001_001}
                 $\displaystyle
                  h(x+\xi\Delta x)\stackrel{\eqref{Eq_Crl_AELRP_s_RPERR_ss_D_001_001}}{=}
                                   \sum_{s=0}^{N_\tsc{tj}}\left(\sum_{\ell=0}^{\lfloor\frac{s}{2}\rfloor}\frac{\tau_{2\ell}\; \xi^{s-2\ell}}
                                                                                                              {(s-2\ell)!                  }\right)\Delta x^s\;f^{(s)}(x)+O(\Delta x^{N_\tsc{tj}+1})
                                  \stackrel{m:=s-2\ell}{=}
                                  \sum_{m=0}^{N_\tsc{tj}}\sum_{\ell=0}^{\lfloor\frac{N_\tsc{tj}-m}{2}\rfloor}\frac{\tau_{2\ell}\;\Delta x^{m+2\ell}\;f^{(m+2\ell)}(x)}
                                                                                                                  {m!                                                }\xi^m+O(\Delta x^{N_\tsc{tj}+1})
                 $
                }
\begin{subequations}
                                                                                                       \label{Eq_Prp_AELRP_s_EPR_ss_AELPR_001_005}
\begin{alignat}{6}
E_h(x_i+\xi\Delta x;\tsc{s}_{i,M_-,M_+},\Delta x)=p_h(x_i+\xi\Delta x;\tsc{s}_{i,M_-,M_+},\Delta x)-h(x_i+\xi\Delta x)
= \sum_{m=0}^{M}\sum_{k=0}^{\lfloor\frac{M-m}{2}\rfloor}\tau_{2k}\;\frac{\Delta x^{m+2k} f^{(m+2k)}_i}{m!}\xi^m
                                                                                                       \notag\\
+\sum_{m=0}^{M}\sum_{k=0}^{\lfloor\frac{M-m}{2}\rfloor}\sum_{s=M+1}^{N_\tsc{tj}}\frac{\tau_{2k}(m+2k)!}{m!}\nu_{M_-,M_+,m+2k,s}\frac{\Delta x^s f^{(s)}_i}{s!}\xi^m
-\sum_{m=0}^{N_\tsc{tj}}\sum_{k=0}^{\lfloor\frac{N_\tsc{tj}-m}{2}\rfloor}\frac{\tau_{2k}\;\Delta x^{m+2k}\;f^{(m+2k)}_i}
                                                                              {m!                                      }\xi^m+O(\Delta x^{N_\tsc{tj}+1})
                                                                                                      \label{Eq_Prp_AELRP_s_EPR_ss_AELPR_001_005a}
\end{alignat}
which simplifies to
\begin{alignat}{6}
E_h(x_i+\xi\Delta x;\tsc{s}_{i,M_-,M_+},\Delta x)
= \sum_{m=0}^{M}\sum_{k=0}^{\lfloor\frac{M-m}{2}\rfloor}\sum_{s=M+1}^{N_\tsc{tj}} \frac{\tau_{2k}(m+2k)!}
                                                                                       {              m!}\nu_{M_-,M_+,m+2k,s}
                                                                                                  \frac{\Delta x^{s}f_i^{(s)}}{s!}\xi^m
                                                                                                       \notag\\
+ \sum_{m=0}^{M}\;\sum_{k=\lfloor\frac{M-m}{2}\rfloor+1}^{\lfloor\frac{N_\tsc{tj}-m}{2}\rfloor}
                                  \frac{-\tau_{2k}}
                                       {        m!}\Delta x^{m+2k}f_i^{(m+2k)}\xi^m
+ \sum_{m=M+1}^{N_\tsc{tj}}\sum_{k=0}^{\lfloor\frac{N_\tsc{tj}-m}{2}\rfloor}
                                         \frac{-\tau_{2k}}
                                              {        m!}\Delta x^{m+2k}f_i^{(m+2k)}\xi^m+O(\Delta x^{N_\tsc{tj}+1})
                                                                                                      \label{Eq_Prp_AELRP_s_EPR_ss_AELPR_001_005b}
\end{alignat}

Using~\eqref{Eq_AELRP_s_AppendixA_003} and \eqref{Eq_AELRP_s_AppendixA_002}, \eqref{Eq_Prp_AELRP_s_EPR_ss_AELPR_001_005b} reads (the summation indices on line 1 remaining unchanged)
\begin{alignat}{6}
 E_h(x_i+\xi\Delta x;\tsc{s}_{i,M_-,M_+},\Delta x)
= \sum_{s=M+1}^{N_\tsc{tj}}\left(\sum_{m=0}^{M}\left(\sum_{k=0}^{\lfloor\frac{M-m}{2}\rfloor}\frac{\tau_{2k}(m+2k)!}
                                                                                                  {s!\; m!         }\nu_{M_-,M_+,m+2k,s}\right)\xi^m\right)\Delta x^{s}f_i^{(s)}
                                                                                                       \notag\\
+ \sum_{s=M+1}^{N_\tsc{tj}}\left(\sum_{k=\lceil\frac{s-M}{2}\rceil}^{\lfloor\frac{s}{2}\rfloor}\frac{-\tau_{2k}}
                                                                                                     {(s-2k)!   }\xi^{s-2k}\right)\Delta x^{s}f_i^{(s)}
+ \sum_{s=M+1}^{N_\tsc{tj}}\left(\sum_{k=0}^{\lceil\frac{s-M}{2}\rceil-1}\frac{-\tau_{2k}}
                                                                              {   (s-2k)!}\xi^{s-2k}\right)\Delta x^{s}f_i^{(s)}+O(\Delta x^{N_\tsc{tj}+1})
                                                                                                       \label{Eq_Prp_AELRP_s_EPR_ss_AELPR_001_005c}
\end{alignat}
and defining $\mu_{h,M_-,M_+,s}(\xi)$ by~\eqref{Eq_Prp_AELRP_s_EPR_ss_AELPR_001_001f} we obtain~\eqref{Eq_Prp_AELRP_s_EPR_ss_AELPR_001_001a} and \eqref{Eq_Prp_AELRP_s_EPR_ss_AELPR_001_001c}.

Finally, using~\eqref{Eq_Lem_AELRP_s_RPERR_ss_D_001_001a} in~\eqref{Eq_Prp_AELRP_s_EPR_ss_AELPR_001_001c}
\begin{alignat}{6}
 E_h(x_i+\xi\Delta x;\tsc{s}_{i,M_-,M_+},\Delta x)=&\sum_{s=M+1}^{N_\tsc{tj}}\mu_{h,M_-,M_+,s}(\xi)\sum_{\ell=0}^{\lfloor\frac{N_\tsc{tj}-s}{2}\rfloor}\frac{\Delta x^{s+2\ell}h_i^{(s+2\ell)}}
                                                                                                                                                                {2^{2\ell}\;(2\ell+1)!            }+O(\Delta x^{N_\tsc{tj}+1})
                                                                                                       \notag\\
                                                  =&\sum_{s=M+1}^{N_\tsc{tj}}\sum_{\ell=0}^{\lfloor\frac{N_\tsc{tj}-s}{2}\rfloor}\frac{\mu_{h,M_-,M_+,s}(\xi)}
                                                                                                                                      {2^{2\ell}\;(2\ell+1)!     }\Delta x^{s+2\ell}h_i^{(s+2\ell)}+O(\Delta x^{N_\tsc{tj}+1})
                                                                                                       \label{Eq_Prp_AELRP_s_PRBE_ss_PROA_001_005d}
\end{alignat}
\end{subequations}
which, by \eqref{Eq_AELRP_s_AppendixA_003} and \eqref{Eq_AELRP_s_AppendixA_002}, proves~\eqref{Eq_Prp_AELRP_s_EPR_ss_AELPR_001_001d}.\qed
\end{proof}
%
%-----------------------------------------------------------------------------------------------------------------------------------
%-----------------------------------------------------------------------------------------------------------------------------------
%
\begin{proposition}[Approximation error of Lagrange polynomial reconstruction on $\tsc{s}_{i,M_-,M_+}$]
\label{Prp_AELRP_s_EPR_ss_AELPR_002}
Asssume the conditions and definitions of \prprefnp{Prp_AELRP_s_EPR_ss_AELPR_001}. Then
\begin{subequations}
                                                                                                       \label{Eq_Prp_AELRP_s_EPR_ss_AELPR_002_001}
\begin{alignat}{6}
E_h(x_i+\xi\Delta x;\tsc{s}_{i,M_-,M_+},\Delta x)=&\sum_{n=M+1}^{N_\tsc{tj}}\lambda_{h,M_-,M_+,n}(\xi)\;\Delta x^n\;h^{(n)}(x_i+\xi\Delta x)+O(\Delta x^{N_\tsc{tj}+1})
                                                                                                       \label{Eq_Prp_AELRP_s_EPR_ss_AELPR_002_001a}\\
E_f(x_i+\xi\Delta x;\tsc{s}_{i,M_-,M_+},\Delta x)=&\sum_{n=M+1}^{N_\tsc{tj}}\lambda_{f,M_-,M_+,n}(\xi)\;\Delta x^n\;f^{(n)}(x_i+\xi\Delta x)+O(\Delta x^{N_\tsc{tj}+1})
                                                                                                       \label{Eq_Prp_AELRP_s_EPR_ss_AELPR_002_001b}
\end{alignat}
where $\lambda_{h,M_-,M_+,n}(\xi)$ and $\lambda_{f,M_-,M_+,n}(\xi)$ are polynomials of degree $n$ in $\xi$~\eqref{Eq_Prp_AELRP_s_EPR_ss_PR_001_001f}
\begin{alignat}{6}
\lambda_{h,M_-,M_+,n}(\xi):=&\sum_{\ell=0}^{n-M-1}\mu_{h,M_-,M_+,n-\ell}(\xi)
                                                  \dfrac{(-1)^{\ell+1}}
                                                        {(\ell+1)!    }\left((\xi-\tfrac{1}{2})^{\ell+1}
                                                                            -(\xi+\tfrac{1}{2})^{\ell+1}\right)
                                                                                                       \label{Eq_Prp_AELRP_s_EPR_ss_AELPR_002_001c}\\
\lambda_{f,M_-,M_+,n}(\xi):=&\sum_{\ell=0}^{n-M-1}\dfrac{(-\xi)^\ell}
                                                        {\ell!      }\mu_{f,M_-,M_+,n-\ell}(\xi)
                                                                                                       \label{Eq_Prp_AELRP_s_EPR_ss_AELPR_002_001d}
\end{alignat}
where $\mu_{h,M_-,M_+,n}(\xi)$ is defined by \eqref{Eq_Prp_AELRP_s_EPR_ss_AELPR_001_001f} and $\mu_{f,M_-,M_+,n}(\xi)$ is defined by \eqref{Eq_Prp_AELRP_s_EPR_ss_AELPR_001_001g}.
\end{subequations}
\end{proposition}
%
%-----------------------------------------------------------------------------------------------------------------------------------
%-----------------------------------------------------------------------------------------------------------------------------------
%
\begin{proof}
\begin{subequations}
                                                                                                       \label{Eq_Prp_AELRP_s_EPR_ss_AELPR_002_002}
Taylor-expanding $f_i$ in \eqref{Eq_Prp_AELRP_s_EPR_ss_AELPR_001_001e}, around the point $x_i+\xi\Delta x$,\footnote{\label{ff_Prp_AELRP_s_EPR_ss_AELPR_002_001}
                                                                                                                   ${\displaystyle f^{(n)}(x)=\sum_{\ell=0}^{N_\tsc{tj}}\dfrac{(-\xi)^\ell}
                                                                                                                                                              {\ell!      }\;\Delta x^\ell\;f^{(n+\ell)}(x+\xi\Delta x)
                                                                                                                             +O(\Delta x^{N_\tsc{tj}+1})}$}
we have
\begin{alignat}{6}
E_f(x_i+\xi\Delta x;\tsc{s}_{i,M_-,M_+},\Delta x)=&\sum_{s=M+1}^{N_\tsc{tj}}\mu_{f,M_-,M_+,s}(\xi)\Delta x^s
                                                                            \left(\sum_{\ell=0}^{N_\tsc{tj}-s}\dfrac{(-\xi)^\ell}
                                                                                                                    {\ell!      }\;\Delta x^\ell\;f^{(s+\ell)}(x_i+\xi\Delta x)\right)+O(\Delta x^{N_\tsc{tj}+1})
                                                                                                       \notag\\
                                                 =&\sum_{s=M+1}^{N_\tsc{tj}}\sum_{\ell=0}^{N_\tsc{tj}-s}\mu_{f,M_-,M_+,s}(\xi)
                                                                                                        \dfrac{(-\xi)^\ell}
                                                                                                              {\ell!      }\;\Delta x^{s+\ell}\;f^{(s+\ell)}(x_i+\xi\Delta x)+O(\Delta x^{N_\tsc{tj}+1})
                                                                                                       \notag\\
                                                 =&\sum_{n=M+1}^{N_\tsc{tj}}\sum_{\ell=0}^{n-M-1}\mu_{f,M_-,M_+,n-\ell}(\xi)
                                                                                                 \dfrac{(-\xi)^\ell}
                                                                                                       {\ell!      }\;\Delta x^n\;f^{(n)}(x_i+\xi\Delta x)+O(\Delta x^{N_\tsc{tj}+1})
                                                                                                       \label{Eq_Prp_AELRP_s_EPR_ss_AELPR_002_002a}
\end{alignat}
which proves \eqref{Eq_Prp_AELRP_s_EPR_ss_AELPR_002_001b}.

Replacing $f_i^{(s)}$ in~\eqref{Eq_Prp_AELRP_s_EPR_ss_AELPR_001_001c} by its expansion\footnote{\label{ff_Prp_AELRP_s_EPR_ss_AELPR_002_002}
                                                                                              Approximating $h(\zeta)$ (which was assumed to be of class $C^N$ in \lemrefnp{Lem_AELRP_s_RPERR_ss_RP_001})
                                                                                              in~\eqref{Eq_Def_AELRP_s_RPERR_ss_RP_001_001a}
                                                                                              by the corresponding Taylor-polynomial (Taylor-jet) of order $N_\tsc{tj}$~\cite[pp. 219--232]{Zorich_2004a}
                                                                                              around $\zeta=x+\xi\Delta x$ yields, $\forall N_\tsc{tj}\in{\mathbb N}: N_\tsc{tj}<N$,
\begin{alignat}{6}
f(x)=&\dfrac{1}{\Delta x}\int_{x-\frac{1}{2}\Delta x}^{x+\frac{1}{2}\Delta x}{\left\lgroup\left(
                                                                              \sum_{\ell=0}^{N_\tsc{tj}}\dfrac{(\zeta-x-\xi\Delta x)^\ell}
                                                                                                              {\ell!                     } h^{(\ell)}(x+\xi\Delta x)\right)
                                                                             +O\left((\zeta-x-\xi\Delta x)^{N_\tsc{tj}+1}\right)\right\rgroup d\zeta}
                                                                                                       \notag\\
    =&\dfrac{1}{\Delta x}\int_{x-\frac{1}{2}\Delta x}^{x+\frac{1}{2}\Delta x}{\left(\sum_{\ell=0}^{N_\tsc{tj}}\dfrac{(\zeta-x-\xi\Delta x)^\ell}
                                                                                                                    {\ell!                     } h^{(\ell)}(x+\xi\Delta x)\right)d\zeta}
                                                                              +O(\Delta x^{N_\tsc{tj}+1})
    = \dfrac{1}{\Delta x}\sum_{\ell=0}^{N_\tsc{tj}}\left(\int_{(-\frac{1}{2}-\xi)\Delta x}^{(+\frac{1}{2}-\xi)\Delta x}{\dfrac{\eta^\ell}{\ell!}d\eta}\right)h^{(\ell)}(x+\xi\Delta x)+O(\Delta x^{N_\tsc{tj}+1})
                                                                                                       \notag\\
    =&\sum_{\ell=0}^{N_\tsc{tj}}\left(\dfrac{(-1)^{\ell+1}}
                                            {(\ell+1)!    }\left((\xi-\tfrac{1}{2})^{\ell+1}
                                                                -(\xi+\tfrac{1}{2})^{\ell+1}\right)\Delta x^{\ell}\;h^{(\ell)}(x+\xi\Delta x)\right)+O(\Delta x^{N_\tsc{tj}+1})
                                                                                                       \notag
\end{alignat}}
in terms of the derivatives $\Delta x^{\ell}\;h^{(s+\ell)}(x_i+\xi\Delta x)$
we have
\begin{alignat}{6}
E_h(x_i+\xi\Delta x;\tsc{s}_{i,M_-,M_+},\Delta x)=&\sum_{s=M+1}^{N_\tsc{tj}}\mu_{h,M_-,M_+,s}(\xi)\Delta x^s
                                                                       \left\lgroup\sum_{\ell=0}^{N_\tsc{tj}-s}\dfrac{(-1)^{\ell+1}}
                                                                                                                     {(\ell+1)!    }\left((\xi-\tfrac{1}{2})^{\ell+1}
                                                                                                                                         -(\xi+\tfrac{1}{2})^{\ell+1}\right)
                                                                                                               \Delta x^{\ell}\;h^{(s+\ell)}(x_i+\xi\Delta x)\right\rgroup
                                                                                                       \notag\\
                                                 +&O(\Delta x^{N_\tsc{tj}+1})
                                                                                                       \notag\\
                                                 =&\sum_{s=M+1}^{N_\tsc{tj}}\sum_{\ell=0}^{N_\tsc{tj}-s}\mu_{h,M_-,M_+,s}(\xi)
                                                                                                        \dfrac{(-1)^{\ell+1}}
                                                                                                              {(\ell+1)!    }\left((\xi-\tfrac{1}{2})^{\ell+1}
                                                                                                                                  -(\xi+\tfrac{1}{2})^{\ell+1}\right)
                                                                                                        \Delta x^{s+\ell}\;h^{(s+\ell)}(x_i+\xi\Delta x)
                                                                                                       \notag\\
                                                 +&O(\Delta x^{N_\tsc{tj}+1})
                                                                                                       \notag\\
                                                 =&\sum_{n=M+1}^{N_\tsc{tj}}\sum_{\ell=0}^{n-M-1}\mu_{h,M_-,M_+,s}(\xi)
                                                                                                 \dfrac{(-1)^{\ell+1}}
                                                                                                       {(\ell+1)!    }\left((\xi-\tfrac{1}{2})^{\ell+1}
                                                                                                                           -(\xi+\tfrac{1}{2})^{\ell+1}\right)
                                                                                                 \Delta x^{n}\;h^{(n)}(x_i+\xi\Delta x)
                                                                                                       \notag\\
                                                 +&O(\Delta x^{N_\tsc{tj}+1})
                                                                                                       \label{Eq_Prp_AELRP_s_EPR_ss_AELPR_002_002a}
\end{alignat}
which proves \eqref{Eq_Prp_AELRP_s_EPR_ss_AELPR_002_001a}.
Obviously, by \eqref{Eq_Prp_AELRP_s_EPR_ss_AELPR_002_001d}, ${\rm deg}(\lambda_{f,M_-,M_+,n}(\xi))=n$.
It is easy\footnote{\label{ff_Prp_AELRP_s_EPR_ss_AELPR_002_003}
                    ${\displaystyle\left((\xi-\tfrac{1}{2})^{\ell+1}
                                        -(\xi+\tfrac{1}{2})^{\ell+1}\right)=\sum_{k=0}^{\ell+1}\left(\binom{\ell+1}{k}\xi^{\ell+1-k}\left(-\tfrac{1}{2}\right)^k\right)
                                                                           -\sum_{k=0}^{\ell+1}\left(\binom{\ell+1}{k}\xi^{\ell+1-k}\left(+\tfrac{1}{2}\right)^k\right)
                                                                           =\xi^{\ell+1}\left(\left(-\tfrac{1}{2}\right)^0-\left(+\tfrac{1}{2}\right)^0\right)
                                                                           +\sum_{k=1}^{\ell+1}\left(\binom{\ell+1}{k}\xi^{\ell+1-k}\left(\left(-\tfrac{1}{2}\right)^k
                                                                                                                                         -\left(+\tfrac{1}{2}\right)^k\right)\right)}$}
to verify that, by \eqref{Eq_Prp_AELRP_s_EPR_ss_AELPR_002_001c}, ${\rm deg}(\lambda_{h,M_-,M_+,n}(\xi))=n$, which completes the proof.\qed
\end{subequations}
\end{proof}
%
%-----------------------------------------------------------------------------------------------------------------------------------

%-----------------------------------------------------------------------------------------------------------------------------------
%
%
%
%
%
\subsection{Approximation error of $h_{i\pm\frac{1}{2}}$ and of $f_i'$}\label{AELRP_s_EPR_ss_AEhdf}
%
%
%
%
%
%-----------------------------------------------------------------------------------------------------------------------------------

One of the principal uses of the reconstructing polynomial being the numerical approximation of $f_i':=f'(x_i)$ via \eqref{Eq_Lem_AELRP_s_RPERR_ss_RP_001_002},
we give in this section the relations concerning the approximation error of $h_{i\pm\frac{1}{2}}:=h(x_i\pm\tfrac{1}{2}\Delta x)$ \crlref{Crl_AELRP_s_EPR_ss_AEhdf_001}
and of $f_i'$ \crlref{Crl_AELRP_s_EPR_ss_AEhdf_002}, which are readily obtained by application of \prprefnp{Prp_AELRP_s_EPR_ss_AELPR_002}.
%-----------------------------------------------------------------------------------------------------------------------------------
%
\begin{corollary}[Accuracy at $i+\tfrac{1}{2}$ of Lagrange polynomial reconstruction on $\tsc{s}_{i,M_-,M_+}$]
\label{Crl_AELRP_s_EPR_ss_AEhdf_001}
Let $p_f(x;\tsc{s}_{i,M_-,M_+},\Delta x)$ and $p_h(x;\tsc{s}_{i,M_-,M_+},\Delta x)$ be a polynomial \lemref{Lem_AELRP_s_RP_ss_PRP_001} reconstruction pair \defref{Def_AELRP_s_RPERR_ss_RP_001} $p_h=R_{(1;\Delta x)}(p_f)$,
satisfying the conditions of \prprefnp{Prp_AELRP_s_EPR_ss_PR_001}.
Then, the reconstructed value at $x_{i+\frac{1}{2}}:=x_i+\tfrac{1}{2}\Delta x$,
which will be noted $\hat h_{\tsc{s}_{i,M_-,M_+},i+\frac{1}{2}}$,
approximates $h_{i+\frac{1}{2}}:=h(x_{i+\frac{1}{2}})$ to $O(\Delta x^{M+1})$ with $M:=M_-+M_+\geq0$. The error of the approximation can be expanded in powers of $\Delta x$ with coefficients involving
the derivatives $h^{(m)}_{i+\frac{1}{2}}:=h^{(m)}(x_{i+\frac{1}{2}})$
\begin{subequations}
                                                                                                       \label{Eq_Crl_AELRP_s_EPR_ss_AEhdf_001_001}\\
\begin{alignat}{6}
\hat h_{\tsc{s}_{i,M_-,M_+},i+\frac{1}{2}}:=&p_h(x_{i+\frac{1}{2}};\tsc{s}_{i,M_-,M_+},\Delta x)
                                                                                                       \label{Eq_Crl_AELRP_s_EPR_ss_AEhdf_001_001a}\\
                                           =&\sum_{\ell=-M_-}^{M_+} a_{M_-,M_+,\ell}\;f_{i+\ell}
                                                                                                       \label{Eq_Crl_AELRP_s_EPR_ss_AEhdf_001_001b}\\
                                           =&h_{i+\frac{1}{2}}+
                                                \underbrace{\sum_{s=M+1}^{N_\tsc{tj}}\Lambda_{M_-,M_+,s}\Delta x^s h^{(s)}_{i+\frac{1}{2}}+O(\Delta x^{N_\tsc{tj}+1})}_{\displaystyle O(\Delta x^{M+1})}
                                                                                                       \label{Eq_Crl_AELRP_s_EPR_ss_AEhdf_001_001c}
\end{alignat}
where the constants $\Lambda_{M_-,M_+,s}$ are given by
\begin{equation}
\Lambda_{M_-,M_+,s}:=\lambda_{h,M_-,M_+,s}(\tfrac{1}{2})=\sum_{\ell=0}^{s-M-1}\frac{(-1)^\ell}{(\ell+1)!}\mu_{h,M_-,M_+,s-\ell}(\tfrac{1}{2})
                                                                                                       \label{Eq_Crl_AELRP_s_EPR_ss_AEhdf_001_001d}
\end{equation}
with
$\lambda_{h,M_-,M_+,s}(\xi)$ being the degree $s$ in $\xi$ polynomial defined by \eqref{Eq_Prp_AELRP_s_EPR_ss_AELPR_002_001c},
$\mu_{h,M_-,M_+,s}(\xi)$ being the degree $s$ in $\xi$ polynomial defined by \eqref{Eq_Prp_AELRP_s_EPR_ss_AELPR_001_001f},
and\footnote{\label{ff_Crl_AELRP_s_EPR_ss_AEhdf_001_001}
Notice that Shu~\cite{Shu_1998a}, following a different route, has shown that
\begin{equation}
a_{M_-,M_+,\ell}=\sum_{m=\ell+M_-+1}^{M+1}\dfrac{\displaystyle\sum_{\begin{array}{c}p=0\\p\neq m\\\end{array}}^{M+1}
                                                              \prod_{\begin{array}{c}q=0\\q\neq m\\q\neq p\\\end{array}}^{M+1}(M_--q+1)}
                                                {\displaystyle\prod_{\begin{array}{c}p=0\\p\neq m\\\end{array}}^{M+1}(m-p)}
                                                                                                       \notag
\end{equation}
is an equivalent expression for the coefficients $a_{M_-,M_+,\ell}$ \eqref{Eq_Crl_AELRP_s_EPR_ss_AEhdf_001_001e}.
}
\begin{equation}
a_{M_-,M_+,\ell}:=\alpha_{h,M_-,M_+,\ell}(\tfrac{1}{2})
                                                                                                       \label{Eq_Crl_AELRP_s_EPR_ss_AEhdf_001_001e}
\end{equation}
with $\alpha_{h,M_-,M_+,\ell}(\xi)$ being the degree $M$ in $\xi$ polynomial defined by \eqref{Eq_Prp_AELRP_s_EPR_ss_PR_001_001g}.
\end{subequations}
\end{corollary}
%
%-----------------------------------------------------------------------------------------------------------------------------------
%-----------------------------------------------------------------------------------------------------------------------------------
%
\begin{proof}
Using~\eqref{Eq_Prp_AELRP_s_EPR_ss_PR_001_001d} and~\eqref{Eq_Prp_AELRP_s_EPR_ss_AELPR_002_001a},
in the definition of $\hat h_{\tsc{s}_{i,M_-,M_+},i+\frac{1}{2}}$~\eqref{Eq_Crl_AELRP_s_EPR_ss_AEhdf_001_001a},
we have immediately
\begin{alignat}{6}
\hat h_{\tsc{s}_{i,M_-,M_+},i+\frac{1}{2}}:=&p_h(x_{i+\frac{1}{2}};\tsc{s}_{i,M_-,M_+},\Delta x)
                                           = \sum_{\ell=-M_-}^{M_+}\alpha_{h,M_-,M_+,\ell}(\tfrac{1}{2})f_{i+\ell}
                                                                                                       \notag\\
                                           =&h_{i+\frac{1}{2}}+
                                            \underbrace{\sum_{s=M+1}^{N_\tsc{tj}}\lambda_{h,M_-,M_+,s}(\tfrac{1}{2})\;\Delta x^s\;h_{i+\frac{1}{2}}^{(s)}
                                                       +O(\Delta x^{N_\tsc{tj}+1})}_{\displaystyle E_h(x_i+\tfrac{1}{2}\Delta x;\tsc{s}_{i,M_-,M_+},\Delta x)}
                                                                                                       \label{Eq_Crl_AELRP_s_EPR_ss_AEhdf_001_002}
\end{alignat}
and using the definition \eqref{Eq_Prp_AELRP_s_EPR_ss_AELPR_002_001c} to compute $\lambda_{h,M_-,M_+,s}(\tfrac{1}{2})$ completes the proof.\qed
\end{proof}
%
%-----------------------------------------------------------------------------------------------------------------------------------
%-----------------------------------------------------------------------------------------------------------------------------------
%
\begin{corollary}[Order-of-accuracy of Lagrange polynomial reconstruction]
\label{Crl_AELRP_s_EPR_ss_AEhdf_002}
Assume the conditions of~\prprefnp{Prp_AELRP_s_EPR_ss_PR_001}. Then
\begin{alignat}{6}
\frac{\hat h_{\tsc{s}_{i,M_-,M_+},i+\frac{1}{2}}-\hat h_{\tsc{s}_{i-1,M_-,M_+},i-\frac{1}{2}}}{\Delta x} = f'_i +\sum_{n=M+1}^{N_\tsc{tj}}\Lambda_{M_-,M_+,n}\Delta x^n f^{(n+1)}_i +O(\Delta x^{N_\tsc{tj}+1})
                                                                                                         = f'_i +O(\Delta x^{M+1})
                                                                                                       \label{Eq_Crl_AELRP_s_EPR_ss_AEhdf_002_001}
\end{alignat}
where $\hat h_{\tsc{s}_{i-1,M_-,M_+},i-\frac{1}{2}}=\hat h_{\tsc{s}_{i-1,M_-,M_+},i-1+\frac{1}{2}}$ \eqref{Eq_Crl_AELRP_s_EPR_ss_AEhdf_001_001a},
and the constants $\Lambda_{M_-,M_+,n}$ are defined by~\eqref{Eq_Crl_AELRP_s_EPR_ss_AEhdf_001_001d}.
\end{corollary}
%
%-----------------------------------------------------------------------------------------------------------------------------------
%-----------------------------------------------------------------------------------------------------------------------------------
%
\begin{proof}
The constants $\Lambda_{M_-,M_+,n}$~\eqref{Eq_Crl_AELRP_s_EPR_ss_AEhdf_001_001d} depend only on the 3 indices ($M_-,M_+,n$),
and not on the point index $i$~\rmkref{Rmk_AELRP_s_EPR_ss_PR_001},
because the polynomials $\mu_{h,M_-,M_+,s-\ell}(\xi)$~\eqref{Eq_Prp_AELRP_s_EPR_ss_AELPR_001_001f} are also independent of the point index $i$. Hence, we have, by \eqref{Eq_Crl_AELRP_s_EPR_ss_AEhdf_001_001c},
\begin{alignat}{6}
\hat h_{\tsc{s}_{i-1,M_-,M_+},i-\frac{1}{2}}=\hat h_{\tsc{s}_{i-1,M_-,M_+},i-1+\frac{1}{2}}
                                            =h_{i-\frac{1}{2}}+\sum_{s=M+1}^{N_\tsc{tj}}\Lambda_{M_-,M_+,s}\Delta x^s h^{(s)}_{i-\frac{1}{2}}+O(\Delta x^{N_\tsc{tj}+1})
                                                                                                       \label{Eq_Crl_AELRP_s_EPR_ss_AEhdf_002_002}
\end{alignat}
Subtracting~\eqref{Eq_Crl_AELRP_s_EPR_ss_AEhdf_002_002} from~\eqref{Eq_Crl_AELRP_s_EPR_ss_AEhdf_001_001c} yields
\begin{alignat}{6}
\frac{\hat h_{\tsc{s}_{i,M_-,M_+},i+\frac{1}{2}}-\hat h_{\tsc{s}_{i-1,M_-,M_+},i-\frac{1}{2}}}{\Delta x}= \frac{h_{i+\frac{1}{2}}-h_{i-\frac{1}{2}}}{\Delta x}
                                                   +\sum_{s=M+1}^{N_\tsc{tj}}\Lambda_{M_-,M_+,s}\Delta x^s \frac{h^{(s)}_{i+\frac{1}{2}}-h^{(s)}_{i-\frac{1}{2}}}{\Delta x}
                                                   +O(\Delta x^{N_\tsc{tj}+1})
                                                                                                       \label{Eq_Crl_AELRP_s_EPR_ss_AEhdf_002_003}
\end{alignat}
and using the exact relations~\eqref{Eq_Lem_AELRP_s_RPERR_ss_RP_001_001} we obtain~\eqref{Eq_Crl_AELRP_s_EPR_ss_AEhdf_002_001}.\qed
\end{proof}
%
%-----------------------------------------------------------------------------------------------------------------------------------
%-----------------------------------------------------------------------------------------------------------------------------------
%
\begin{remark}[Order-of-accuracy]
\label{Rmk_AELRP_s_PRBE_ss_PROA_003}
The previous result~\crlref{Crl_AELRP_s_EPR_ss_AEhdf_002} illustrates that the $O(\Delta x^{M+1})$ accuracy in approximating $f'$ is achieved,
using $O(\Delta x^{M+1})$ interpolates for $f$, because of the exact reconstruction relations~\lemref{Lem_AELRP_s_RPERR_ss_RP_001}.
Liu \etal~\cite{Liu_Osher_Chan_1994a} note this as an $O(\Delta x^{M})$ accuracy increased to $O(\Delta x^{M+1})$ at one chosen point, {\em viz} $x_i$.\qed
\end{remark}
%
%-----------------------------------------------------------------------------------------------------------------------------------

%-----------------------------------------------------------------------------------------------------------------------------------
%
%
%
%
%
%
%
%
%
%
\section{Interpolating and reconstructing polynomial}\label{AELRP_s_IRP}
%
%
%
%
%
%
%
%
%
%
%-----------------------------------------------------------------------------------------------------------------------------------

We briefly summarize how the existence and uniqueness properties of the interpolating polynomial
carry on to the reconstructing polynomial. Consider first the general case of a polynomial reconstruction pair (\S\ref{AELRP_s_RP_ss_PRP}).
Combining the existence \lemref{Lem_AELRP_s_RP_ss_PRP_001} and uniqueness \rmkref{Rmk_AELRP_s_RPERR_ss_D_002} of polynomial reconstruction pairs, we can formulate
%-----------------------------------------------------------------------------------------------------------------------------------
%
\begin{theorem}[Vector spaces of polynomial reconstruction pairs]
\label{Thm_AELRP_s_PRBE_ss_PR_001}
Consider the $(M+1)$-dimensional vector space of polynomials with real coefficients of degree $\leq M$ in $x$, $\mathbb{R}_M[x]$. Then the reconstruction mapping $R_{(1;\Delta x)}$ \defref{Def_AELRP_s_RPERR_ss_RP_001}
is a bijection of $\mathbb{R}_M[x]$ onto itself.
\end{theorem}
%
%-----------------------------------------------------------------------------------------------------------------------------------
%-----------------------------------------------------------------------------------------------------------------------------------
%
\begin{proof}
By construction \lemref{Lem_AELRP_s_RP_ss_PRP_001} $\forall\;p(x)\;\in\;\mathbb{R}_M[x]\;\exists\;q(x)=[R_{(1;\Delta x)}(p)](x)\;\in\;\mathbb{R}_M[x]$, and inversely
$\forall\;q(x)\;\in\;\mathbb{R}_M[x]\;\exists\;p(x)=[R^{-1}_{(1;\Delta x)}(q)](x)\;\in\;\mathbb{R}_M[x]$. Furthermore, since the elements of $\mathbb{R}_M[x]$ are continuous functions,
the reconstruction pair $q(x)=[R_{(1;\Delta x)}(p)](x)$ is unique \rmkref{Rmk_AELRP_s_RPERR_ss_D_002}, which completes the proof.\qed
\end{proof}
%
%-----------------------------------------------------------------------------------------------------------------------------------

In his recent review of \tsc{weno} schemes, Shu~\cite{Shu_2009a} stresses the difference between \tsc{weno} interpolation and \tsc{weno} reconstruction.
In this sense, $p_f(x;\tsc{s}_{i,M_-,M_+},\Delta x)$ in \prprefnp{Prp_AELRP_s_EPR_ss_PR_001} is the interpolating polynomial of $f(x)$ on $\tsc{s}_{i,M_-,M_+}$,
and $p_h(x;\tsc{s}_{i,M_-,M_+},\Delta x)$ is the reconstructing polynomial \defref{Def_AELRP_s_RPERR_ss_RP_002}. Of course
%-----------------------------------------------------------------------------------------------------------------------------------
%
\begin{proposition}[Lagrange reconstructing polynomial]
\label{Prp_AELRP_s_PRBE_ss_IRP_001}
Assume the conditions of \prprefnp{Prp_AELRP_s_EPR_ss_AELPR_001}. The Lagrange reconstructing polynomial $p_h(x;\tsc{s}_{i,M_-,M_+},\Delta x)$
approximates $h(x)$ to $O(\Delta x^{M+1})$ but, unless $f(x)$ is a polynomial of degree $\leq M$, it does not interpolate $h(x)$ on $\tsc{s}_{i,M_-,M_+}$,
{\em ie}, if $f(x)$ is not a polynomial of degree $\leq M$, we have in general
\begin{equation}
p_h(x_i+\ell\Delta x;\tsc{s}_{i,M_-,M_+},\Delta x)\neq h(x_i+\ell\Delta x)\qquad \forall \ell\in\{-M_-,\cdots,M_+\}
                                                                                                       \label{Eq_Prp_AELRP_s_PRBE_ss_IRP_001_001}
\end{equation}
\end{proposition}
%
%-----------------------------------------------------------------------------------------------------------------------------------
%-----------------------------------------------------------------------------------------------------------------------------------
%
\begin{proof}
Proof is obtained by contradiction. It suffices to give an example where the inequalities \eqref{Eq_Prp_AELRP_s_PRBE_ss_IRP_001_001} hold.
Consider the reconstruction pair \thmref{Thm_AELRP_s_RPERR_ss_GFtaunRPexp_001}
\begin{subequations}
                                                                                                       \label{Eq_Prp_AELRP_s_PRBE_ss_IRP_001_002}
\begin{equation}
f(x):={\rm e}^{x-x_i}\qquad;\qquad h(x)=[R_{(1;\Delta x)}(f)](x)=g_\tau(\Delta x){\rm e}^{x-x_i}
                                                                                                       \label{Eq_Prp_AELRP_s_PRBE_ss_IRP_001_002a}
\end{equation}
with $g_\tau$ defined by \eqref{Eq_Thm_AELRP_s_RPERR_ss_GFtaunRPexp_001_001b}. Consider the polynomial reconstruction of $f(x)$ \prpref{Prp_AELRP_s_EPR_ss_PR_001} on $\tsc{s}_{i,1,1}$.
By \eqref{Eq_Prp_AELRP_s_EPR_ss_PR_001_001d} and \eqref{Eq_Prp_AELRP_s_EPR_ss_PR_001_001g}
\begin{alignat}{6}
p_h(x_i+\xi\Delta x;\tsc{s}_{i,1,1},\Delta x)= f_{i-1}\left(\tfrac{1}{2}\xi^2-\tfrac{1}{2}\xi-\tfrac{1}{24}\right)
                                             + f_i    \left(\tfrac{13}{12}-\xi^2\right)
                                             + f_{i-1}\left(\tfrac{1}{2}\xi^2+\tfrac{1}{2}\xi-\tfrac{1}{24}\right)
                                                                                                       \label{Eq_Prp_AELRP_s_PRBE_ss_IRP_001_002b}
\end{alignat}
We have $f_i=1$ and $f_{i\pm1}={\rm e}^{\pm\Delta x}$, and evaluating $p_h(x_i+\ell\Delta x;\tsc{s}_{i,1,1},\Delta x)-h(x_i+\ell\Delta x)$,
using \eqref{Eq_Prp_AELRP_s_PRBE_ss_IRP_001_002b} and \eqref{Eq_Prp_AELRP_s_PRBE_ss_IRP_001_002a}, for $\ell=-1,0,1$, and for different values
of $\Delta x$ ({\em eg} $\Delta x=\tfrac{1}{100}$), we verify \eqref{Eq_Prp_AELRP_s_PRBE_ss_IRP_001_001}.
\end{subequations}
\qed
\end{proof}
%
%-----------------------------------------------------------------------------------------------------------------------------------
 
Most of the results of existence and uniqueness properties of the interpolating polynomial hold, with appropriate adjustments,
for the reconstructing polynomial, because of \thmrefnp{Thm_AELRP_s_PRBE_ss_PR_001}. We briefly summarize in the following those
necessary to prove \tsc{weno} reconstruction relations~\cite{Shu_1998a,
                                                             Shu_2009a}.
%-----------------------------------------------------------------------------------------------------------------------------------
%
\begin{theorem}[Existence and uniqueness of the Lagrange reconstructing polynomial]
\label{Thm_AELRP_s_PRBE_ss_IRP_001}
Assume the conditions of \prprefnp{Prp_AELRP_s_EPR_ss_AELPR_001}. There exists a unique Lagrange reconstructing polynomial $p_h(x;\tsc{s}_{i,M_-,M_+},\Delta x)$
of the form \eqref{Eq_Prp_AELRP_s_EPR_ss_PR_001_001d} which approximates $h(x)$ to $O(\Delta x^{M+1})$.
\end{theorem}
%
%-----------------------------------------------------------------------------------------------------------------------------------
%-----------------------------------------------------------------------------------------------------------------------------------
%
\begin{proof}
Existence, with $\alpha_{h,M_-,M_+,\ell}(\xi)$ given by \eqref{Eq_Prp_AELRP_s_EPR_ss_PR_001_001g}, is proved in \prprefnp{Prp_AELRP_s_EPR_ss_PR_001} by construction.
We know from approximation theory~\cite{Henrici_1964a,
                                        Phillips_2003a}
that there is a unique Lagrange interpolating polynomial $p_f(x;\tsc{s}_{i,M_-,M_+},\Delta x)$ on $\tsc{s}_{i,M_-,M_+}$,
and that the reconstruction pair $p_h(x;\tsc{s}_{i,M_-,M_+},\Delta x)=[R_{(1;\Delta x)}(p_f)](x;\tsc{s}_{i,M_-,M_+},\Delta x)$ is unique \rmkref{Rmk_AELRP_s_RPERR_ss_D_002},
which completes the proof.\qed
\end{proof}
%
%-----------------------------------------------------------------------------------------------------------------------------------

%-----------------------------------------------------------------------------------------------------------------------------------
%
%
%
%
%
%
%
%
%
%
\section{Examples of applications}\label{AELRP_s_EA}
%
%
%
%
%
%
%
%
%
%
%-----------------------------------------------------------------------------------------------------------------------------------

The analytical relations developed in the present work can prove quite useful in the analysis of practical \tsc{weno} schemes, and more generally in the
development of discretization schemes. Providing detailed analysis of such applications is beyond the scope of the present paper.
We sketch, nonetheless, in the following, 3 applications (the complete proofs will be given elsewhere), to illustrate the usefulness of the reconstruction
pair concept, of the associated application of the deconvolution \lemrefnp{Lem_AELRP_s_RPERR_ss_D_001}, and of the explicit expressions for the Lagrange reconstructing polynomial.

%-----------------------------------------------------------------------------------------------------------------------------------
%
%
%
%
%
\subsection{Representation of the Lagrange reconstructing polynomial by combination of substencils}\label{AELRP_s_EA_ss_RRPCS}
%
%
%
%
%
%-----------------------------------------------------------------------------------------------------------------------------------

All \tsc{weno}~\cite{Shu_1998a,
                     Shu_2009a}
schemes for reconstruction on the general homogeneous stencil $\tsc{s}_{i,M_-,M_+}$ \defref{Def_AELRP_s_EPR_ss_PR_001} are based
on the weighted combination of the reconstructions on $K_{\rm s}+1\leq M:=M_-+M_+$ substencils\footnote{\label{ff_AELRP_s_EA_ss_RRPCS_001}
                                                                                                        Notice that the family of subdivisions \eqref{Eq_AELRP_s_EA_ss_RRPCS_001}
                                                                                                        includes (when varying $K_{\rm s}$) all possible subdivisions
                                                                                                        to substencils of equal length ($M-K_{\rm s}$ intervals) whose union is the entire stencil $\tsc{s}_{i,M_-,M_+}$.
                                                                                                       }
\begin{alignat}{6}
\tsc{s}_{i,M_--k_{\rm s},M_+-K_{\rm s}+k_{\rm s}}=\left\{i-M_-+k_{\rm s},\cdots,M_+-K_{\rm s}+k_{\rm s}\right\} \qquad\begin{array}{l}M_\pm\in\mathbb{Z}:M=M_-+M_+\geq2\\
                                                                                                                                      1\leq K_{\rm s}\leq M-1\\
                                                                                                                                      k_{\rm s}\in\{0,\cdots,K_{\rm s}\}\\\end{array}
                                                                                                       \label{Eq_AELRP_s_EA_ss_RRPCS_001}
\end{alignat}
with appropriate weights, which are nonlinear in the cell-averages $f(x)$, to ensure monotonicity at discontinuities~\cite{vanLeer_2006a},
and such that the weighted combination of the Lagrange reconstructing polynomials on the substencils, at regions where $h(x)$ is smooth,
approximates to $O(\Delta x^{M+1})$ \cite{Jiang_Shu_1996a} or higher \cite{Henrick_Aslam_Powers_2005a} the Lagrange reconstructing polynomial on the big stencil $\tsc{s}_{i,M_-,M_+}$,
which \prpref{Prp_AELRP_s_EPR_ss_AELPR_001} is $O(\Delta x^{M+1})$-accurate. The starting point for scheme design is the determination of the
underlying linear scheme, {\em ie} the determination of weight-functions $\sigma_{h,M_-,M_+,K_{\rm s},k_{\rm s}}(\xi)$ ($k_{\rm s}\in\{0,\cdots,K_{\rm s}\}$), independent of $f(x)$,
which combine the Lagrange reconstructing polynomials on the substencils exactly into the Lagrange reconstructing polynomial on the big stencil $\tsc{s}_{i,M_-,M_+}$
\begin{subequations}
                                                                                                       \label{Eq_AELRP_s_EA_ss_RRPCS_002}
\begin{alignat}{6}
p_h(x_i+\xi\Delta x;\tsc{s}_{i,M_-,M_+},\Delta x)=\sum_{k_{\rm s}=0}^{K_{\rm s}} \sigma_{h,M_-,M_+,K_{\rm s},k_{\rm s}}(\xi)\;p_h(x_i+\xi\Delta x;\tsc{s}_{i,M_--k_{\rm s},M_+-K_{\rm s}+k_{\rm s}},\Delta x)
                                                                                                       \label{Eq_AELRP_s_EA_ss_RRPCS_002a}
\end{alignat}
Obviously, using \eqref{Eq_Prp_AELRP_s_EPR_ss_AELPR_002_001a} in \eqref{Eq_AELRP_s_EA_ss_RRPCS_002a},
the weight functions $\sigma_{h,M_-,M_+,K_{\rm s},k_{\rm s}}(\xi)$ must satisfy the consistency condition
\begin{alignat}{6}
\sum_{k_{\rm s}=0}^{K_{\rm s}}\sigma_{h,M_-,M_+,K_{\rm s},k_{\rm s}}(\xi)=1
                                                                                                       \label{Eq_AELRP_s_EA_ss_RRPCS_002b}
\end{alignat}
\end{subequations}
Shu~\cite{Shu_1998a} indicated examples of instances where it was impossible to find such weights, as well as instances where convexity of the combination was lost (presence of negative weights).
The corresponding problem for the Lagrange interpolating polynomial
\begin{subequations}
                                                                                                       \label{Eq_AELRP_s_EA_ss_RRPCS_003}
\begin{alignat}{6}
p_f(x_i+\xi\Delta x;\tsc{s}_{i,M_-,M_+},\Delta x)=\sum_{k_{\rm s}=0}^{K_{\rm s}} \sigma_{f,M_-,M_+,K_{\rm s},k_{\rm s}}(\xi)\;p_f(x_i+\xi\Delta x;\tsc{s}_{i,M_--k_{\rm s},M_+-K_{\rm s}+k_{\rm s}},\Delta x)
                                                                                                       \label{Eq_AELRP_s_EA_ss_RRPCS_003a}
\end{alignat}
\begin{alignat}{6}
\sum_{k_{\rm s}=0}^{K_{\rm s}}\sigma_{f,M_-,M_+,K_{\rm s},k_{\rm s}}(\xi)=1
                                                                                                       \label{Eq_AELRP_s_EA_ss_RRPCS_003b}
\end{alignat}
\end{subequations}
is directly related to the Neville-Aitken algorithm~\cite[pp. 11--13]{Phillips_2003a}, which constructs the interpolating polynomial on $\{i-M_-,\cdots,i+M_+\}$ by recursive combination of the
interpolating polynomials on substencils, with weight-functions which are also polynomials of $x$~\cite[pp. 11--13]{Phillips_2003a}.
Carlini \etal~\cite{Carlini_Ferretti_Russo_2005a} have given the explicit representation of the polynomial weights for the construction of the Lagrange interpolating polynomial on $\tsc{s}_{i,M_-,M_+}$,
by combination of the Lagrange interpolating polynomials on the $K_{\rm s}+1$ substencils, for a certain family of stencils/subdivisions. Liu \etal~\cite{Liu_Shu_Zhang_2009a}
have extended the family of stencils/subdivisions studied. Since for the Lagrange interpolating polynomial case the weight-functions
are polynomials, \eqref{Eq_AELRP_s_EA_ss_RRPCS_003} are valid $\forall\xi\in\mathbb{R}$.

The  corresponding problem for the Lagrange reconstructing polynomial \eqref{Eq_AELRP_s_EA_ss_RRPCS_002} was studied,
only very recently, by Liu \etal~\cite{Liu_Shu_Zhang_2009a}.
It turns out that, for the reconstruction case, the weight-functions which satisfy \eqref{Eq_AELRP_s_EA_ss_RRPCS_002}
are rational functions of $\xi$, implying that \eqref{Eq_AELRP_s_EA_ss_RRPCS_002} is valid $\forall\xi\in\mathbb{R}\setminus{\mathcal S}_{\sigma_{h,M_-,M_+,K_{\rm s}}}$,
{\em ie} everywhere except at the union ${\mathcal S}_{\sigma_{h,M_-,M_+,K_{\rm s}}}$ of the poles of the $K_{\rm s}+1$ rational functions $\sigma_{h,M_-,M_+,K_{\rm s},k_{\rm s}}(\xi)$ ($k_{\rm s}\in\{0,\cdots,K_{\rm s}\}$).
Liu \etal~\cite{Liu_Shu_Zhang_2009a} studied the family of stencils $\tsc{s}_{i,\lfloor\frac{M}{2}\rfloor,M-\lfloor\frac{M}{2}\rfloor}$,
for the $K_{\rm s}=\left\lceil\frac{M}{2}\right\rceil$-level subdivision, in the range $M\in\{2\cdots,11\}$,
and used symbolic computation to give explicit expressions of the weight-functions, and to study their poles and regions of convexity.

We highlight in the following how the identification of reconstruction pairs \defref{Def_AELRP_s_RPERR_ss_RP_001},
and the analytical expressions for the Lagrange reconstructing polynomial \prpref{Prp_AELRP_s_EPR_ss_PR_001} and its approximation error \prpref{Prp_AELRP_s_EPR_ss_AELPR_002},
can be used to develop general analytical expressions (valid $\forall M_\pm\in\mathbb{Z}:M:=M_-+M_+\geq2$ and $\forall K_{\rm s}\in\{1,\cdots,M-1\}$) for the weight-functions,
prove that there can be no poles at cell-interfaces ($n+\tfrac{1}{2}\;\forall n\in\mathbb{Z}$), and extend the important results obtained in Liu \etal~\cite{Liu_Shu_Zhang_2009a}.
It is quite straightforward, using the definition of the Lagrange reconstructing polynomial \defref{Def_AELRP_s_RPERR_ss_RP_002},
to show by \eqref{Eq_Def_AELRP_s_RPERR_ss_RP_001_001a} that
the polynomial $\alpha_{h,M_-,M_+,\ell}(\xi)$ \eqref{Eq_Prp_AELRP_s_EPR_ss_PR_001_001g} appearing in the representation \eqref{Eq_Prp_AELRP_s_EPR_ss_PR_001_001d} of the Lagrange reconstructing polynomial
on the stencil $\tsc{s}_{i,M_-,M_+}$ is the reconstruction pair, on a unit-spacing grid,
of the corresponding polynomial $\alpha_{f,M_-,M_+,\ell}(\xi)$ \eqref{Eq_Prp_AELRP_s_EPR_ss_PR_001_001h}
appearing in the representation \eqref{Eq_Prp_AELRP_s_EPR_ss_PR_001_001e} of the Lagrange interpolating polynomial on the same stencil
\begin{alignat}{6}
\alpha_{h,M_-,M_+,\ell}(\xi)=\left[R_{(1;1)}(\alpha_{f,M_-,M_+,\ell})\right](\xi)\iff\alpha_{f,M_-,M_+,\ell}(\xi)=\int_{\xi-\frac{1}{2}}^{\xi+\frac{1}{2}}\alpha_{h,M_-,M_+,\ell}(\eta)\;d\eta
\quad\left\{\begin{array}{l}\forall\;\ell\in\{-M_-,\cdots,M_+\}\\
                            \forall\xi\in{\mathbb R}\\\end{array}\right.
                                                                                                       \label{Eq_AELRP_s_EA_ss_RRPCS_004}
\end{alignat}
It is easy to prove by \eqref{Eq_AELRP_s_EA_ss_RRPCS_004}, using the mean value theorem for the definite integral~\cite[pp. 350--359]{Zorich_2004a},
and the knowledge of the $M$ roots of $\alpha_{f,M_-,M_+,\ell}(\xi)$ ($\alpha_{f,M_-,M_+,\ell}(n)=0\;\forall n\in\{M_-,\cdots,M_+\}\setminus\{\ell\}$~\cite{Henrici_1964a,
                                                                                                                                                            Phillips_2003a})
that
\begin{alignat}{6}
\alpha_{h,M_-,M_+,\ell}(n+\tfrac{1}{2})\neq 0
\quad\left\{\begin{array}{l}\forall\;\ell\in\{-M_-,\cdots,M_+\}\\
                            \forall n\in{\mathbb Z}\\\end{array}\right.
                                                                                                       \label{Eq_AELRP_s_EA_ss_RRPCS_005}
\end{alignat}
and that all of the $M$ roots of $\alpha_{h,M_-,M_+,\ell}(\xi)$ are real.
It can be shown that both $\left\{\alpha_{h,M_-,M_+,\ell}(\xi),\;\ell\in\{-M_-,\cdots,M_+\}\right\}$
and $\left\{\alpha_{f,M_-,M_+,\ell}(\xi),\;\ell\in\{-M_-,\cdots,M_+\}\right\}$ form a basis of
the $(M+1)$-dimensional vector space of polynomials of degree $\leq M$ in $\xi$, $\mathbb{R}_M[\xi]$, and therefore, none of these polynomials is identically $0$.
We can work out several identities for the polynomials $\alpha_{h,M_-,M_+,\ell}(\xi)$, with corresponding identities for $\alpha_{f,M_-,M_+,\ell}(\xi)$ because of \eqref{Eq_AELRP_s_EA_ss_RRPCS_004},
and show that an analytical expression for the rational weight-functions $\sigma_{h,M_-,M_+,K_{\rm s},k_{\rm s}}(\xi)$ is given by the recurrence\footnote{\label{ff_AELRP_s_EA_ss_RRPCS_002}
                                                                                                                                                           The recurrence relation \eqref{Eq_AELRP_s_EA_ss_RRPCS_006} for $K_{\rm s}\geq2$
                                                                                                                                                           holds also for $\sigma_{f,M_-,M_+,K_{\rm s},k_{\rm s}}(\xi)$.
                                                                                                                                                          }
\begin{alignat}{6}
&\sigma_{h  ,M_-,M_+,K_{\rm s},k_{\rm s}}(\xi)=\left\{\begin{array}{ll}\dfrac{\alpha_{h  ,M_-          ,M_+            ,-M_-+k_{\rm s}M}(\xi)}
                                                                             {\alpha_{h  ,M_--k_{\rm s},M_+-1+k_{\rm s},-M_-+k_{\rm s}M}(\xi)}
                                                                                                                                                                   &K_{\rm s}=1           \\
                                                                                                                                                                   &                      \\
                                                                      {\displaystyle\sum_{\ell_{\rm s}=\max(0,k_{\rm s}-1)}^{\min(K_{\rm s}-1,k_{\rm s})}
                                                                                    \sigma_{h  ,M_-             ,M_+,K_{\rm s}-1                 ,\ell_{\rm s}}(\xi)\;
                                                                                    \sigma_{h  ,M_--\ell_{\rm s},M_+-(K_{\rm s}-1)+\ell_{\rm s},1,k_{\rm s}-\ell_{\rm s}}(\xi)
                                                                      }                                                                                            &K_{\rm s}\geq2        \\\end{array}\right.
                                                                                                       \notag\\
&\forall k_{\rm s}=0,\cdots,K_{\rm s}\leq M-1
                                                                                                       \label{Eq_AELRP_s_EA_ss_RRPCS_006}
\end{alignat}
This analytical formulation, which only requires \eqref{Eq_Prp_AELRP_s_EPR_ss_PR_001_001g} as input, is easily programmed in any symbolic computation package, and can generate the rational
weight-functions $\forall M_\pm\in\mathbb{Z}:M=M_-+M_+\geq2$. Since all of the $M$ roots of any of the polynomials $\alpha_{h,M_-,M_+,\ell}(\xi)$ are real, by \eqref{Eq_AELRP_s_EA_ss_RRPCS_006},
we can show by induction that all of the poles of the rational weight-functions $\sigma_{h,M_-,M_+,K_{\rm s},k_{\rm s}}(\xi)$ are real.
These analytical results can then be used to extend the results of Liu \etal~\cite{Liu_Shu_Zhang_2009a} $\forall M\in\mathbb{N}_{\geq2}$, and for arbitrarily biased stencils in a homogeneous
grid, providing at the same time simple symbolic computation routines for roots and poles.

Going into further details and results is beyond the scope of the present work. We include however the following result.
For the $K_{\rm s}=\left\lceil\frac{M}{2}\right\rceil$-level subdivision of the usual \tsc{weno} stencils~\cite{Liu_Shu_Zhang_2009a} $\tsc{s}_{i,\lfloor\frac{M}{2}\rfloor,M-\lfloor\frac{M}{2}\rfloor}$,
we know from direct computation~\cite{Liu_Shu_Zhang_2009a} that the weight-functions
$\sigma_{h,\left\lfloor\frac{M}{2}\right\rfloor,M-\left\lfloor\frac{M}{2}\right\rfloor,\left\lceil\frac{M}{2}\right\rceil,k_{\rm s}}(\tfrac{1}{2})\geq0$.
Using the analytical expression \eqref{Eq_AELRP_s_EA_ss_RRPCS_006} we have obtained computationally the following result
%-----------------------------------------------------------------------------------------------------------------------------------
%
\begin{result}[Positivity of linear weights at $i+\tfrac{1}{2}$]
\label{Rst_AELRP_s_EA_ss_RRPCS_001}
\begin{subequations}
                                                                                                       \label{Eq_AELRP_s_EA_ss_AJSSI_001}
Assume that $|M_\pm|\leq 9$, satisfying \eqref{Eq_AELRP_s_EA_ss_RRPCS_001}.
Then if for the subdivision level $K_{\rm s}$ of $\tsc{s}_{i,M_-,M_+}$ \eqref{Eq_AELRP_s_EA_ss_RRPCS_001} all substencils contain either point $i$ or point $i+1$
\begin{equation}
\tsc{s}_{i,M_--k_{\rm s},M_+-K_{\rm s}+k_{\rm s}}\cap\{i,i+1\}\neq\varnothing\quad \forall k_{\rm s}\in\{0,\cdots,K_{\rm s}\}
\iff\left\{\begin{array}{l}-M_-\leq0<M_+\\
                           K_{\rm s}\leq\min(M_-+1,M_+)>0\\\end{array}\right.
                                                                                                       \label{Eq_Rst_AELRP_s_EA_ss_RRPCS_001_001a}
\end{equation}
then the rational weight-functions \eqref{Eq_AELRP_s_EA_ss_RRPCS_006} satisfy
\begin{equation}
\sigma_{h,M_-,M+,K_{\rm s},k_{\rm s}}(\tfrac{1}{2}) >0\qquad\forall k_{\rm s}\in\{0,\cdots,K_{\rm s}\}
                                                                                                       \label{Eq_Rst_AELRP_s_EA_ss_RRPCS_001_001b}
\end{equation}
\end{subequations}
\qed
\end{result}
%
%-----------------------------------------------------------------------------------------------------------------------------------

%-----------------------------------------------------------------------------------------------------------------------------------
%
%
%
%
%
\subsection{Truncation error of \tsc{weno} approximations to $f'(x)$}\label{AELRP_s_EA_ss_TEWENOAdfdx}
%
%
%
%
%
%-----------------------------------------------------------------------------------------------------------------------------------

Nonlinearity, ensuring monotonicity~\cite{vanLeer_2006a}, in \tsc{weno} schemes is introduced by nonlinear weighting-out of stencils for which
the reconstructing polynomial is nonsmooth. Smoothness is almost invariably~\cite{Jiang_Shu_1996a,
                                                                                  Shu_1998a,
                                                                                  Levy_Puppo_Russo_1999a,
                                                                                  Balsara_Shu_2000a,
                                                                                  Qiu_Shu_2002a,
                                                                                  Qiu_Shu_2003a,
                                                                                  Henrick_Aslam_Powers_2005a,
                                                                                  Borges_Carmona_Costa_Don_2008a,
                                                                                  Shu_2009a,
                                                                                  Gerolymos_Senechal_Vallet_2009a}
measured using the Jiang-Shu smoothness indicators~\cite{Jiang_Shu_1996a}. Let
$u:\mathbb{R}\longrightarrow\mathbb{R}$, $M\in\mathbb{N}_{\geq1}$ and $\Delta x\in\mathbb{R}_{>0}$, and define
\begin{subequations}
                                                                                                       \label{Eq_AELRP_s_EA_ss_AJSSI_001}
\begin{alignat}{6}
\beta_M(x;\Delta x;u):=&\sum_{k=1}^{M}\int_{x-\frac{1}{2}\Delta x}^{x+\frac{1}{2}}
                                           {\Delta x^{2k-1}\left(\frac{d^k }
                                                                      {dx^k} u(\zeta)\right)^2d\zeta}
                                                                                                       \label{Eq_AELRP_s_EA_ss_TEWENOAdfdx_001a}\\
                                =&\sum_{k=1}^{M}\int_{-\tfrac{1}{2}}^{\tfrac{1}{2}}
                                                {               \left(\frac{d^k }
                                                                           {d\xi^k} u(x+\xi \Delta x)\right)^2d\xi}
                                                                                                       \label{Eq_AELRP_s_EA_ss_TEWENOAdfdx_001b}
\end{alignat}
By \eqref{Eq_AELRP_s_EA_ss_TEWENOAdfdx_001b} it is seen that, upon normalization by $\Delta x$, $\beta_M(x;\Delta x;u)$
is~\cite{Tartar_2007a} the usual norm of $\left(d_\xi u\right)$ in the Sobolev space $H^{M-1}\Bigl((-\tfrac{1}{2},+\tfrac{1}{2})\Bigr):=W^{M-1,2}\Bigl((-\tfrac{1}{2},+\tfrac{1}{2})\Bigr)$.
The Jiang-Shu smoothness indicator, when considering reconstruction at $x_{i+\frac{1}{2}}$ on $\tsc{s}_{i,M_-,M_+}$ \crlref{Crl_AELRP_s_EPR_ss_AEhdf_001},
is defined\footnote{\label{ff_AELRP_s_EA_ss_TEWENOAdfdx_001}
                    The interval of integration was defined by Jiang and Shu~\cite{Jiang_Shu_1996a} as the cell $[x_{i-\frac{1}{2}},x_{i+\frac{1}{2}}]$, with center the pivot point $x_i$ where we
                    wish to approximate the derivative $f_i':=f'(x_i)$. This introduces some sort of upwinding, since the interval $[x_i,x_{i+1}]$ could
                    have been used (this would correspond to using $\beta_M(x_{i+\frac{1}{2}};\Delta x;p_h(\cdot;\tsc{s}_{i,M_-,M_+},\Delta x))$ instead in \eqref{Eq_AELRP_s_EA_ss_TEWENOAdfdx_003}.
                    This is the choice made in the context of central \tsc{weno} interpolation by Carlini \etal~\cite{Carlini_Ferretti_Russo_2005a}.
                    The choice of $[x_i,x_{i+1}]$ as the integration interval for upwind-biased \tsc{weno} schemes has not been studied. Nonetheless,
                    in the context of cell-centered finite-volume schemes~\cite{Shi_Hu_Shu_2002a} the choice of the cell as the volume of integration seems natural.
                   }~\cite{Jiang_Shu_1996a,
                           Shu_1998a,
                           Balsara_Shu_2000a,
                           Shu_2009a}
by
\begin{alignat}{6}
\beta_{p_h,\tsc{s}_{i,M_-,M_+},i+\frac{1}{2}}:=&\beta_M\Big(x_i;\Delta x;p_h(\cdot;\tsc{s}_{i,M_-,M_+},\Delta x)\Big)
                                                                                                       \label{Eq_AELRP_s_EA_ss_TEWENOAdfdx_003}
\end{alignat}
\end{subequations}
with, as usual, $M:=M_-+M_+$.
The realization of the optimal order-of-accuracy by the \tsc{weno} schemes hinges upon the fact~\cite{Jiang_Shu_1996a,
                                                                                                      Henrick_Aslam_Powers_2005a,
                                                                                                      Borges_Carmona_Costa_Don_2008a,
                                                                                                      Shu_2009a}
that $\beta_{p_h,\tsc{s}_{i,M_-,M_+},i+\frac{1}{2}}$ for different stencils of equal length $M:=M_-+M_+$ but different biasing around the pivot-point $x_i$ only differ to $O(\Delta x^{M+2})$, the lower-order
part being common to all stencils of equal length $M$. Expansions in powers of $\Delta x^s f_i^{n}f_i^{s-n}$ ($n\leq\lfloor\frac{s}{2}\rfloor$) have been given in the literature,
up to the \tscn{weno}{$17$} (composed by 9 substencils of length $M=8$ cells~\cite{Gerolymos_Senechal_Vallet_2009a}), using symbolic calculation~\cite{Jiang_Shu_1996a,
                                                                                                                                                       Henrick_Aslam_Powers_2005a,
                                                                                                                                                       Borges_Carmona_Costa_Don_2008a,
                                                                                                                                                       Gerolymos_Senechal_Vallet_2009a}.
Using the analytical expressions for the error of the reconstructing polynomial with respect to the unknown function $h(x)$ which is reconstructed,
{\em eg} \eqref{Eq_Prp_AELRP_s_EPR_ss_AELPR_002_001c}, \eqref{Eq_Prp_AELRP_s_EPR_ss_AELPR_001_001c} or \eqref{Eq_Prp_AELRP_s_EPR_ss_AELPR_001_001d},
it is quite straightforward to explain the existence of the common part.
Using \eqref{Eq_Prp_AELRP_s_EPR_ss_AELPR_002_001a} in \eqref{Eq_AELRP_s_EA_ss_TEWENOAdfdx_003} yields, by \eqref{Eq_AELRP_s_EA_ss_TEWENOAdfdx_001b}
\begin{alignat}{6}
\beta_{p_h,\tsc{s}_{i,M_-,M_+},i+\frac{1}{2}}=&\sum_{k=1}^{M}\int_{-\tfrac{1}{2}}^{\tfrac{1}{2}}
                                                 {               \left(\frac{d^k }
                                                                            {d\xi^k}\biggl(h(x_i+\xi\Delta x)+E_h(x_i+\xi\Delta x;\tsc{s}_{i,M_-,M_+},\Delta x)\biggr)\right)^2d\xi}
                                                                                                       \notag\\
                                             =&\sum_{k=1}^{M}\int_{-\tfrac{1}{2}}^{\tfrac{1}{2}}
                                                 {               \left(\frac{d^k }
                                                                            {d\xi^k}\biggl(h(x_i+\xi\Delta x)\biggr)\right)^2d\xi}
                                                                                                       \notag\\
                                             +&\sum_{k=1}^{M}\int_{-\tfrac{1}{2}}^{\tfrac{1}{2}}
                                                 {               \left(2\left(\frac{d^k }
                                                                                   {d\xi^k}h(x_i+\xi\Delta x)\right)
                                                                        \left(\frac{d^k }
                                                                                   {d\xi^k}E_h(x_i+\xi\Delta x;\tsc{s}_{i,M_-,M_+},\Delta x)\right)
                                                                       +\left(\frac{d^k }
                                                                                   {d\xi^k}E_h(x_i+\xi\Delta x;\tsc{s}_{i,M_-,M_+},\Delta x)\right)^2\right)d\xi}
                                                                                                       \label{Eq_AELRP_s_EA_ss_TEWENOAdfdx_004}
\end{alignat}
{\em ie} the common (stencil-independent) part is indeed the Sobolev norm of $d_\xi h(x_i+\xi\Delta x)$ and the non-common part involves the approximation error of the Lagrange reconstructing
polynomial \prpref{Prp_AELRP_s_EPR_ss_AELPR_002}, and is therefore stencil-dependent. The deconvolution \lemrefnp{Lem_AELRP_s_RPERR_ss_D_001} is necessary to compute analytically the expansion of
the common part in terms of powers $\Delta x^s f_i^{n}f_i^{s-n}$ ($n\leq\lfloor\frac{s}{2}\rfloor$), and \eqref{Eq_Prp_AELRP_s_EPR_ss_AELPR_001_001c} combined with the deconvolution \lemrefnp{Lem_AELRP_s_RPERR_ss_D_001}
is used to evaluate analytically the expansion of the stencil-dependent part of \eqref{Eq_AELRP_s_EA_ss_TEWENOAdfdx_004}, whose knowledge is essential for the evaluation and
improvement~\cite{Jiang_Shu_1996a,
                  Henrick_Aslam_Powers_2005a,
                  Borges_Carmona_Costa_Don_2008a}
of the design of nonlinear weights, especially when interested in developing weights maintaining one of the great advantages of the Jiang-Shu weights, {\em viz} the straightforward
extension to arbitrarily high-order accuracy~\cite{Balsara_Shu_2000a,
                                                   Gerolymos_Senechal_Vallet_2009a}.
Furthermore these expressions were used to compute analytically the leading 2 terms of the asymptotic expansions of the Jiang-Shu nonlinear weights~\cite{Jiang_Shu_1996a}
and from these the leading term of the truncation error of \tsc{weno} and \tsc{wenom}~\cite{Henrick_Aslam_Powers_2005a} schemes. These developments are quite lengthy, and will be reported elsewhere.

%-----------------------------------------------------------------------------------------------------------------------------------
%
%
%
%
%
\subsection{Extension to higher derivatives by multiple reconstruction}\label{AELRP_s_EA_ss_WHDMR}
%
%
%
%
%
%-----------------------------------------------------------------------------------------------------------------------------------

The reconstruction approach can also be used to approximate $f''(x)$ (and in general $f^{(n)}(x)$), and in particular to compute interface fluxes,
for high-order conservative discretization of diffusive terms in finite-volume methods, {\em eg} in the direct numerical simulation (\tsc{dns}) of turbulence~\cite{Gerolymos_Senechal_Vallet_2010a}.

Assume the conditions of \defrefnp{Def_AELRP_s_RPERR_ss_RP_001}, for $f(x)$ and $h(x)$, but defined on $I=[a-\Delta x,b+\Delta x]\subset{\mathbb R}$, and assume that
$\exists\;\mathfrak{h}:I\longrightarrow{\mathbb R}$ which satisfies
\begin{subequations}
                                                                                                        \label{Eq_AELRP_s_EA_ss_WHDMR_001}
\begin{alignat}{6}
\mathfrak{h}=R_{(1;\Delta x)}(h)\;\stackrel{\eqref{Eq_Def_AELRP_s_RPERR_ss_RP_001_001a}}{\Longrightarrow}&
h(x)=\dfrac{1}{\Delta x}\int_{x-\frac{1}{2}\Delta x}^{x+\frac{1}{2}\Delta x}{\mathfrak{h}(\zeta)d\zeta}             &\forall x\in[a-\tfrac{1}{2}\Delta x,b+\tfrac{1}{2}\Delta x]
                                                                                                        \label{Eq_AELRP_s_EA_ss_WHDMR_001a}\\
                                  \stackrel{\eqref{Eq_Lem_AELRP_s_RPERR_ss_RP_001_001}}{\Longrightarrow}&
h^{(n)}(x)=\dfrac{\mathfrak{h}^{(n-1)}(x+\frac{1}{2}\Delta x)-\mathfrak{h}^{(n-1)}(x-\frac{1}{2}\Delta x)}{\Delta x}&\;\begin{array}{l}\forall x\in[a-\tfrac{1}{2}\Delta x,b+\tfrac{1}{2}\Delta x]\\
                                                                                                                                       \forall n\in\{1,\cdots,N\}\\\end{array}
                                                                                                        \label{Eq_AELRP_s_EA_ss_WHDMR_001b}
\end{alignat}
\end{subequations}
assuming $f,h,\mathfrak{h}\in C^N[a-\Delta x,b+\Delta x]$, with $N\in\mathbb{N}_{\geq2}$.
Recall \rmkref{Rmk_AELRP_s_RPERR_ss_D_002} that reconstruction pairs of continuous functions, if they exist, are unique.
Combining \eqref{Eq_Def_AELRP_s_RPERR_ss_RP_001_001a} with \eqref{Eq_AELRP_s_EA_ss_WHDMR_001a}, we have
\begin{alignat}{6}
f(x)=[R^{-1}_{(1;\Delta x)}(h)](x)\stackrel{\eqref{Eq_AELRP_s_EA_ss_WHDMR_001a}}{=}[\underbrace{R^{-1}_{(1;\Delta x)}\circ R^{-1}_{(1;\Delta x)}}_{R^{-1}_{(2;\Delta x)}}\left(\mathfrak{h}\right)](x)
                                                                                =\dfrac{1}{\Delta x^2}\int_{   x-\frac{1}{2}\Delta x}^{   x+\frac{1}{2}\Delta x}
                                                                                                \left(\int_{\eta-\frac{1}{2}\Delta x}^{\eta+\frac{1}{2}\Delta x}{\mathfrak{h}(\zeta)d\zeta}\right)d\eta
                                                                                                        \label{Eq_AELRP_s_EA_ss_WHDMR_002}
\end{alignat}
and we may write $\mathfrak{h}=R_{(2;\Delta x)}(f)$ the reconstruction pair of $f(x)$ for the computation of the 2-derivative.
Differentiating \eqref{Eq_AELRP_s_EA_ss_WHDMR_002} twice with respect to $x$, we readily obtain the \underline{exact} relation
\begin{subequations}
                                                                                                        \label{Eq_AELRP_s_EA_ss_WHDMR_003}
\begin{alignat}{6}
f''(x)=&\dfrac{\mathfrak{h}(x+\Delta x)-2\mathfrak{h}(x)+\mathfrak{h}(x-\Delta x)}{\Delta x^2}
                                                                                                        \label{Eq_AELRP_s_EA_ss_WHDMR_003a}\\
      =&\dfrac{1}{\Delta x}\Bigg[\underbrace{\dfrac{\mathfrak{h}(x+\Delta x)-\mathfrak{h}(x         )}{\Delta x}}_{\displaystyle\stackrel{\eqref{Eq_AELRP_s_EA_ss_WHDMR_001b}}{=}h'(x+\tfrac{1}{2}\Delta x)}
                                -\underbrace{\dfrac{\mathfrak{h}(x         )-\mathfrak{h}(x-\Delta x)}{\Delta x}}_{\displaystyle\stackrel{\eqref{Eq_AELRP_s_EA_ss_WHDMR_001b}}{=}h'(x-\tfrac{1}{2}\Delta x)}\Bigg]
                                                                                                        \label{Eq_AELRP_s_EA_ss_WHDMR_003b}
\end{alignat}
\end{subequations}
By successive application of the deconvolution \lemrefnp{Lem_AELRP_s_RPERR_ss_D_001}, assuming the conditions of \lemrefnp{Lem_AELRP_s_RPERR_ss_D_001},
we can obtain the deconvolution relations for $\mathfrak{h}(x)$
\begin{subequations}
                                                                                                       \label{Eq_AELRP_s_EA_ss_WHDMR_004}
\begin{alignat}{6}
f^{(n)}(x)=&\sum_{\ell=0}^{\lfloor\frac{N_\tsc{tj}}{2}\rfloor}\left(\sum_{s=0}^{\ell}\binom{2\ell+2}{2s+1}\right)
                                                              \dfrac{\Delta x^{2\ell}     }
                                                                    {2^{2\ell}\;(2\ell+2)!} \mathfrak{h}^{(n+2\ell)}(x)+O(\Delta x^{2\lfloor\frac{N_\tsc{tj}}{2}\rfloor+2})
&\;\begin{array}{c}\forall x\in[a,b]                                                \\
                   \forall n\in{\mathbb N}_0:n<N-2\lfloor\frac{N_\tsc{tj}}{2}\rfloor\\\end{array}
                                                                                                       \label{Eq_AELRP_s_EA_ss_WHDMR_004a}
\end{alignat}
\begin{alignat}{6}
\mathfrak{h}^{(n)}(x)=&\sum_{\ell=0}^{\lfloor\frac{N_\tsc{tj}}{2}\rfloor}\left(\sum_{s=0}^{\ell}\tau_{2s}\tau_{2\ell-2s}\right)\Delta x^{2\ell}f^{(n+2\ell)}(x)+O(\Delta x^{2\lfloor\frac{N_\tsc{tj}}{2}\rfloor+2})
&\;\begin{array}{c}\forall x\in[a,b]                                                \\
                   \forall n\in{\mathbb N}_0:n<N-2\lfloor\frac{N_\tsc{tj}}{2}\rfloor\\\end{array}
                                                                                                       \label{Eq_AELRP_s_EA_ss_WHDMR_004b}
\end{alignat}
\end{subequations}
and redevelop all the results concerning $R_{(1;\Delta x)}$ for $R_{(2;\Delta x)}$, and in general for reconstruction procedures determining the interface fluxes for the computation
of $f''(x)$. In an analogous manner we can tackle the important practical problem of very-high-order conservative discretizations of $(f_\tsc{a}(x)f_\tsc{b}'(x))'$~\cite{Zingg_DeRango_Nemec_Pulliam_2000a}.

%-----------------------------------------------------------------------------------------------------------------------------------
%
%
%
%
%
%
%
%
%
\section{Conclusions}\label{AELRP_s_C}
%
%
%
%
%
%
%
%
%
%-----------------------------------------------------------------------------------------------------------------------------------

The results in this paper concern both the general relations between two functions constituting a reconstruction pair \defref{Def_AELRP_s_RPERR_ss_RP_001},
and the analysis of the approximation error of the reconstructing polynomial \defref{Def_AELRP_s_RPERR_ss_RP_002}.

We call a function $h(x)$ whose sliding averages over a constant length $\Delta x$ are equal to $f(x)$ the reconstruction pair
of $f(x)$, $h=R_{(1;\Delta x)}(f)$ \defref{Def_AELRP_s_RPERR_ss_RP_001}.
The exact relations $\Delta x f^{(n)}(x)=h^{(n-1)}(x+\tfrac{1}{2}\Delta x)-h^{(n-1)}(x-\tfrac{1}{2}\Delta x)$ \eqref{Eq_Lem_AELRP_s_RPERR_ss_RP_001_001} are the basis
of the numerical approximation of $f'(x)$ by reconstruction~\cite{Harten_Osher_1987a,
                                                                  Harten_Engquist_Osher_Chakravarthy_1987a,
                                                                  Shu_Osher_1988a,
                                                                  Shu_Osher_1989a}.

The reconstruction pair of the exponential function is $[R_{(1;\Delta x)}({\rm exp})](x)=g_\tau(\Delta x){\rm e}^x$ \thmref{Thm_AELRP_s_RPERR_ss_GFtaunRPexp_001}.
The function $g_\tau(x)$ \eqref{Eq_Thm_AELRP_s_RPERR_ss_GFtaunRPexp_001_001b} is the generating function of the numbers $\tau_n$ \tabref{Tab_Lem_AELRP_s_RPERR_ss_D_001_001}
satisfying recurrence \eqref{Eq_Lem_AELRP_s_RPERR_ss_D_001_001c}.
The numbers $\tau_n$ \eqref{Eq_Thm_AELRP_s_RPERR_ss_GFtaunRPexp_001_001c} define the coefficients of the analytical solution \lemref{Lem_AELRP_s_RPERR_ss_D_001}
of the deconvolution problem for Taylor polynomials~\cite[(3.13), pp. 244--246]{Harten_Engquist_Osher_Chakravarthy_1987a}.
This analytical solution \lemref{Lem_AELRP_s_RPERR_ss_D_001} is one of the main results of this work. It was also obtained by an alternative matrix-algebra oriented approach (\S\ref{AELRP_s_RP_ss_MIPPRPLem}).

The reconstruction pair of a polynomial of degree $M\in{\mathbb N}$ is also a polynomial of degree $M$ \lemref{Lem_AELRP_s_RP_ss_PRP_001},
whose coefficients can be explicitly determined by \eqref{Eq_Lem_AELRP_s_RP_ss_PRP_001_001f} using the numbers $\tau_n$ \tabref{Tab_Lem_AELRP_s_RPERR_ss_D_001_001},
$R_{(1;\Delta x)}$ being a bijection of the vector space of polynomials of degree $\leq M\in{\mathbb N}$ onto itself \thmref{Thm_AELRP_s_PRBE_ss_PR_001}.

The Lagrange reconstructing polynomial \defref{Def_AELRP_s_RPERR_ss_RP_002} on an arbitrary stencil $\tsc{s}_{i,M_-,M_+}:=\left\{i-M_-,\cdots,i+M_+\right\}$, on a homogeneous grid,
in the neighbourhood of point $i$ \defref{Def_AELRP_s_EPR_ss_PR_001}, is the reconstruction pair of the Lagrange interpolating polynomial on $\tsc{s}_{i,M_-,M_+}:=\left\{i-M_-,\cdots,i+M_+\right\}$.
The Lagrange reconstructing polynomial on $\tsc{s}_{i,M_-,M_+}$ is of degree $M:=M_-+M_+$ \prpref{Prp_AELRP_s_EPR_ss_PR_001} and approximates $h(x)$ to $O(\Delta x^{M+1})$ \prpref{Prp_AELRP_s_EPR_ss_AELPR_002}.
The complete expansion \eqref{Eq_Prp_AELRP_s_EPR_ss_AELPR_002_001a} of the approximation error of the Lagrange reconstructing polynomial
in terms of powers of $\Delta x$ can be expressed using the polynomials $\lambda_{h,M_-,M_+,n}(\xi)$ defined by \eqref{Eq_Prp_AELRP_s_EPR_ss_AELPR_002_001c}.
Most of the standard results of existence and uniqueness of the interpolating polynomial apply to the reconstructing polynomial \thmref{Thm_AELRP_s_PRBE_ss_IRP_001}.

Typical applications include the analytical expression and results on the roots and poles of the rational weight-functions combining the Lagrange reconstructing
polynomials on substencils of $\tsc{s}_{i,M_-,M_+}:=\left\{i-M_-,\cdots,i+M_+\right\}$ into the Lagrange reconstructing polynomial on the entire stencil, 
the analytical expression of the Taylor expansions of the Jiang-Shu~\cite{Jiang_Shu_1996a} smoothness indicators and of the truncation error of \tsc{weno} schemes, and
the analysis of the discretization of $f^{(n)}(x)$ by $n$-reconstruction. It is hoped that the theoretical relations on reconstruction pairs and the analytical expressions
of the approximation error of the reconstructing polynomial will be useful in the analysis and improvement of practical discretization schemes.

%-----------------------------------------------------------------------------------------------------------------------------------
%
%
%
%
%
%
%
%
%
\appendix
%
%
%
%
%
%
%
%
%
%-----------------------------------------------------------------------------------------------------------------------------------
%-----------------------------------------------------------------------------------------------------------------------------------
%
%
%
%
%
%
%
%
%
\section{Useful relations for summation indices}\label{AELRP_s_AppendixA}
%
%
%
%
%
%
%
%
%
%-----------------------------------------------------------------------------------------------------------------------------------

We summarize here several relations~\cite{Knuth_1992a,
                                          Graham_Knuth_Patashnik_1994a} used in the text to manipulate the limits of summation indices, and some other useful formulas.
%-----------------------------------------------------------------------
\begin{equation}
\begin{array}{ccc}\alpha\leq n     &\iff& \lceil \alpha\rceil \leq n                  \\
                                   &    &                                             \\
                  \alpha  <  n     &\iff& \lfloor\alpha\rfloor  <  n                  \\
                                   &    &                                             \\
                  n       <  \beta &\iff& n                     < \lceil \beta\rceil  \\
                                   &    &                                             \\
                  n     \leq \beta &\iff& n                   \leq \lfloor\beta\rfloor\\\end{array}\qquad\qquad\begin{array}{lcc}\forall\alpha,\beta&\in&{\mathbb R}\\
                                                                                                                               \forall  n           &\in&{\mathbb Z}\\\end{array}
                                                                                                       \label{Eq_AELRP_s_AppendixA_001}
\end{equation}
%-----------------------------------------------------------------------
\begin{equation}
\begin{array}{ccc} s    \leq 2k     &\iff& \lceil \frac{s}{2}\rceil \leq k                         \\
                                    &    &                                                         \\
                   s      <  2k     &\iff& \lfloor\frac{s}{2}\rfloor  <  k                         \\
                                    &    &                                                         \\
                  2k      <   s     &\iff& k                          < \lceil  \frac{s}{2}\rceil  \\
                                    &    &                                                         \\
                  2k    \leq  s     &\iff& k                        \leq \lfloor\frac{s}{2}\rfloor \\\end{array}\qquad\qquad\forall s,k \in{\mathbb Z}
                                                                                                       \label{Eq_AELRP_s_AppendixA_002}
\end{equation}
%-----------------------------------------------------------------------
\begin{alignat}{6}
\sum_{n=N_{\min}}^{N_{\max}}\sum_{m=M_{\min}}^{M_{\max}} a_{nm} = \sum_{s=N_{\min}+M_{\min}}^{N_{\max}+M_{\max}}\sum_{n=\max(N_{\min},s-M_{\max})}^{\min(N_{\max},s-M_{\min})} a_{n,s-n}
                                                                                                      %\notag\\
                                                                = \sum_{s=N_{\min}+M_{\min}}^{N_{\max}+M_{\max}}\sum_{m=\max(M_{\min},s-N_{\max})}^{\min(M_{\max},s-N_{\min})} a_{s-m,m}
                                                                                                       \label{Eq_AELRP_s_AppendixA_003}
\end{alignat}
%-----------------------------------------------------------------------
\begin{equation}
\frac{1        }
     {\ell+2k+1}\binom{\ell+2k+1}
                      {\ell     }=\frac{1   }
                                       {2k+1}\binom{\ell+2k}
                                                   {2k     }
                                                                                                       \label{Eq_AELRP_s_AppendixA_004}
\end{equation}
%-----------------------------------------------------------------------

%-----------------------------------------------------------------------------------------------------------------------------------
%
%
%
%
%
%
%
%
%
\section*{Acknowledgments}
%
%
%
%
%
%
%
%
%
%-----------------------------------------------------------------------------------------------------------------------------------

The present work was partly supported by the \tsc{eu}-funded research project ProBand,
(\tsc{strep}--\tscn{fp}{$6$} \tscn{ast}{$4$}--\tsc{ct}--{\small 2005}--{\small 012222}).
Computations were performed using \tsc{hpc} resources from \tsc{genci--idris} (Grants 2010--066327 and 2010--022139).

Symbolic calculations were performed using {\tt maxima} ({\tt http://sourceforge.net/projects/maxima}).
The corresponding package {\tt reconstruction.mac} is available at {\tt http://www.aerodynamics.fr}.

%-----------------------------------------------------------------------------------------------------------------------------------
%
%
%
%
%
%
%
%
%
%\footnotesize\bibliographystyle{model3-num-names}%\bibliography{Aerodynamics,GV,GV_news,Aerodynamics_in_press}
%
%
%
%
%
%
%
%
%
%-----------------------------------------------------------------------------------------------------------------------------------

%\bibliographystyle{amsplain}
%\begin{thebibliography}{10}
%
%\bibitem {A} T. Aoki, \textit{Calcul exponentiel des op\'erateurs
%microdifferentiels d'ordre infini.} I, Ann. Inst. Fourier (Grenoble)
%\textbf{33} (1983), 227--250.
%
%\bibitem {B} R. Brown, \textit{On a conjecture of Dirichlet},
%Amer. Math. Soc., Providence, RI, 1993.
%
%\bibitem {D} R. A. DeVore, \textit{Approximation of functions},
%Proc. Sympos. Appl. Math., vol. 36,
%Amer. Math. Soc., Providence, RI, 1986, pp. 34--56.
%
%\end{thebibliography}

\end{document}

%% file: table_tauk.tex
\scalebox{1.0}{
\begin{tabular}{lclclc}\hline\addlinespace[2pt]
$\tau_0   =$&$1$                                                                                                                &\\[2pt]
$\tau_1   =$&$0$                                                                                                                &\\[2pt]
$\tau_2   =$&$\tfrac{                                  -1}{                                                                 24}$&\\[2pt]
$\tau_3   =$&$0$                                                                                                                &\\[2pt]
$\tau_4   =$&$\tfrac{                                   7}{                                                              5,760}$&\\[2pt]
$\tau_5   =$&$0$                                                                                                                &\\[2pt]
$\tau_6   =$&$\tfrac{                                 -31}{                                                            967,680}$&\\[2pt]
$\tau_7   =$&$0$                                                                                                                &\\[2pt]
$\tau_8   =$&$\tfrac{                                 127}{                                                        154,828,800}$&\\[2pt]
$\tau_9   =$&$0$                                                                                                                &\\[2pt]
$\tau_{10}=$&$\tfrac{                                 -73}{                                                      3,503,554,560}$&\\[2pt]
$\tau_{11}=$&$0$                                                                                                                &\\[2pt]
$\tau_{12}=$&$\tfrac{                           1,414,477}{                                              2,678,117,105,664,000}$&\\[2pt]
$\tau_{13}=$&$0$                                                                                                                &\\[2pt]
$\tau_{14}=$&$\tfrac{                              -8,191}{                                                612,141,052,723,200}$&\\[2pt]
$\tau_{15}=$&$0$                                                                                                                &\\[2pt]
$\tau_{16}=$&$\tfrac{                          16,931,177}{                                         49,950,709,902,213,120,000}$&\\[2pt]
$\tau_{17}=$&$0$                                                                                                                &\\[2pt]
$\tau_{18}=$&$\tfrac{                      -5,749,691,557}{                                    669,659,197,233,029,971,968,000}$&\\[2pt]
$\tau_{19}=$&$0$                                                                                                                &\\[2pt]
$\tau_{20}=$&$\tfrac{                      91,546,277,357}{                                420,928,638,260,761,696,665,600,000}$&\\[2pt]
$\tau_{21}=$&$0$                                                                                                                &\\[2pt]\hline\addlinespace[2pt]
\end{tabular}
}\\